\documentclass[12pt]{article}
\usepackage{bbm, amsmath,amsthm,verbatim,amssymb,amsfonts,amscd,diagrams, graphicx, mathrsfs, hyperref,mathtools,tipa, enumitem}
\usepackage{amstext}

\makeatletter
\def\foo#1\endgraf\unskip#2\foo{\def\row@to@buffer{#1\endgraf\unskip\unskip#2}}
\expandafter\foo\row@to@buffer\foo
\makeatother

\theoremstyle{plain}
\newtheorem{theorem}{Theorem}[section]
\newtheorem{corollary}[theorem]{Corollary}
\newtheorem{lemma}[theorem]{Lemma}
\newtheorem{proposition}[theorem]{Proposition}
\newtheorem*{theorem*}{Theorem}
\newtheorem*{corollary*}{Corollary}

\theoremstyle{definition}
\newtheorem{definition}[theorem]{Definition}

\theoremstyle{remark}
\newtheorem{remark}[theorem]{Remark}
\newtheorem{example}[theorem]{Example}
\newcommand{\codim}{\text{codim}}
\newarrow{ul}---->
\newarrow{Backwards}<----   

\newcommand{\primeset}[1]{#1}
\newcommand{\id}{\textup{id}}

\newcommand{\td}[1]{\tilde{#1}}
\newcommand{\into}{\hookrightarrow}
\newcommand{\PP}{\mathbb{P}}
\newcommand{\bS}{\mathbb{S}}
\newcommand{\X}{\mathbb{X}}
\newcommand{\Z}{\mathbb{Z}}
\newcommand{\Q}{\mathbb{Q}}
\newcommand{\R}{\mathbb{R}}

\newcommand{\N}{\mathbb{N}}

\renewcommand{\H}{\mathbb H}

\newcommand{\mc}[1]{\mathcal{#1}}
\newcommand{\ms}[1]{\mathscr{#1}}

\newcommand{\mf}{\mathfrak}

\newarrow{Onto}----{>>}
\newarrow{Equals}=====

\newcommand{\p}{\mathbbm{p}}

\newcommand{\xr}{\xrightarrow}

\DeclareRobustCommand{\zvec}[1]{%
  \mathrlap{\vec{\mkern-2mu\phantom{#1}}}#1%
}

\DeclareMathAlphabet{\mathpzc}{OT1}{pzc}{m}{it}

\begin{document}

\title{Topological invariance of \\
torsion-sensitive intersection homology}
\author{Greg Friedman\thanks{This work was partially supported by a grant from the Simons Foundation (\#839707 to Greg Friedman)}\\Texas Christian University, Fort Worth, TX, USA\\g.friedman@tcu.edu }

\date{July 31, 2023}

\maketitle
\tableofcontents

\bigskip



\begin{abstract}
Torsion-sensitive intersection homology was introduced to unify several versions of Poincar\'e duality for stratified spaces into a single theorem. This unified duality theorem holds with ground coefficients in an arbitrary PID and with no local cohomology conditions on the underlying space. In this paper we consider for torsion-sensitive intersection homology analogues of another important property of classical intersection homology: topological invariance. In other words, we consider to what extent the defining sheaf complexes of the theory are independent (up to quasi-isomorphism) of choice of stratification. In addition to providing torsion sensitive versions of the existing invariance theorems for classical intersection homology, our techniques provide some new results even in the classical setting.
\end{abstract}

\section{Introduction}

In \cite{GBF32} we introduced categories of \emph{torsion-sensitive perverse sheaves} (more briefly \emph{ts-perverse sheaves})  and studied their duality properties. In the classical category of perverse sheaves on a stratified pseudomanifold \cite{BBD}, the intermediate extensions of the coefficient systems are the ``Deligne sheaves'' whose hypercohomology groups are the intersection homology groups of Goresky and MacPherson. The primary motivation in \cite{GBF32} was to create a generalization of these Deligne sheaves for which the various intersection homology duality theorems of Goresky-MacPherson \cite{GM1,GM2}, Goresky-Siegel \cite{GS83}, and Cappell-Shaneson \cite{CS91}
all arise as special cases of a single more general duality theorem that incorporates certain torsion phenomena into the sheaf complexes but does not require the special local cohomological conditions on spaces that are needed for some of the original theorems. Indeed, the \emph{ts-Deligne} sheaves of \cite{GBF32}, which are the intermediate extensions of \emph{ts-coefficient systems}, fulfill that goal, and furthermore they can be characterized by a simple set of axioms generalizing the Deligne sheaf axioms of Goresky and MacPherson.  

After providing a generalization of Poincar\'e duality for singular spaces, the next most important property of intersection homology is its topological invariance: while the intersection homology groups are defined in terms of a stratification of the space, the resulting intersection homology groups are independent of the choice of stratification, at least assuming certain restrictions on the perversity parameters. In this paper we consider the topological invariance of the ts-Deligne sheaves up to quasi-isomorphism, including confirming a conjecture made in \cite{GBF32}. In addition to extending versions of past topological invariance results to the torsion sensitive category, our techniques specialize to improve the previous known results about ordinary intersection homology. 

In particular, our main theorem will be the following:

\begin{theorem*}[Theorem \ref{T: top}]
Suppose that $X$ and $\X$ denote two CS set stratifications of the same underlying space with $\X$ coarsening $X$. Let $\mc E$ be a ts-coefficient system such that $\X$ is adapted to $\mc E$ (and hence so is $X$).  Let $\vec p$ and $\vec \p$ be respective ts-perversities on $X$ and $\X$ that are $\mc E$-compatible, and let $\mc P^*$ and $\PP^*$ be the respective ts-Deligne sheaves with coefficients $\mc E$. 
Then $\mc P^*$ is quasi-isomorphic to $\PP^*$.
\end{theorem*}
Here CS sets are a class of stratified space generalizing pseudomanifolds, and so it includes irreducible algebraic and analytic varieties, and the $\mc E$-compatibility condition imposes both growth rate conditions on the perversities in relation to each other as well as conditions on how they interact with the coefficient system $\mc E$. 
While it will take some time below to explain all of the definitions in detail, we note as a corollary the special case in which our ts-perversity $\vec p$ is simply one of the original perversity parameters $\bar p$ of Goresky and MacPherson. In this case, our result implies, by a different route, the original topological invariance result for intersection homology; see \cite{GBF44} for a detailed treatment of just this special case.

\begin{corollary*}[Goresky-MacPherson \cite{GM2}]
Let $X$ be an $n$-dimensional topological stratified pseudomanifold, e.g.\ a Whitney stratified irreducible complex variety, and let $Y$ be the same space with a different stratification as a stratified pseudomanifold. Let $\bar p$ be a perversity as defined by Goresky and MacPherson, i.e.\ a function $\bar p: \Z_{\geq 2}\to \Z$ such that $\bar p(2)=0$ and $\bar p(k)\leq \bar p(k+1)\leq \bar p(k)+1$. Then letting $I^{\bar p}H_*$ denote the Goresky-MacPherson intersection homology groups, $$I^{\bar p}H_*(X) \cong I^{\bar p}H_*(Y).$$ 
\end{corollary*}

To explain further, we briefly outline some of the history.

\paragraph{History.}
The original intersection homology groups of Goresky and MacPherson \cite{GM2} are defined on stratified pseudomanifolds and depend on perversity parameters $\bar p:\Z_{\geq 2}\to \Z$ satisfying the original Goresky-MacPherson conditions: $\bar p(2)=0$ and $\bar p(k)\leq \bar p(k+1)\leq \bar p(k)+1$. If $X$ is an $n$-dimensional stratified pseudomanifold, so in particular a filtered space $X=X^n\supset X^{n-2}\supset\cdots \supset X^0$ with each $X^m-X^{m-1}$ an $m$-manifold (possibly empty), then the Deligne sheaf  is constructed beginning with a local system $E$ on $X-X^{n-2}$ and then performing a sequence of pushforwards and truncations over strata of increasing codimension. The perversity value $\bar p(k)$ determines the truncation degree following the pushforward to the codimension $k$ strata. In \cite{GM2}, Goresky and MacPherson showed that for a fixed perversity and local system the resulting sheaves are independent (up to quasi-isomorphism) of the precise choice of pseudomanifold stratification; a more detailed exposition was  provided by Borel in \cite[Section V.4]{Bo}.

King \cite{Ki} later gave a proof of the topological invariance of intersection homology without using sheaves and requiring only that $\bar p$ be nonnegative as well as the growth condition. Furthermore, King worked in the broader category of CS sets and allowed strata of codimension one. However, it should be noted that when $\bar p$ has values such that $\bar p(k)>k-2$ the singular chain intersection homology of \cite{Ki} is not quite the same thing as the sheaf-theoretic intersection homology of \cite{GM2}; see \cite{GBF26} for a discussion. In our book on singular chain intersection homology \cite{GBF35}, we call King's singular chain intersection homology  ``GM intersection homology,'' while that arising from the Deligne sheaf hypercohomology is called ``non-GM intersection homology.'' If $\bar p(k)\leq k-2$ for all $k$, as is the case in the original work of Goresky and MacPherson \cite{GM2}, then these theories all agree. A sheaf theoretic approach to GM intersection homology and its topological invariance can be found in Habegger and Saper \cite{HS91}, while a singular chain approach to non-GM intersection homology has been developed in \cite{GBF10, Sa05, GBF35}. Topological invariance of non-GM intersection homology is considered in \cite{GBF11}, where it is shown that topological invariance holds with $\bar p(1)>0$ and $\bar p(k)\leq \bar p(k+1)\leq \bar p(k)+1$ so long as all changes to the stratification occur within a fixed choice of $n-1$ skeleton $X^{n-1}$, but not in general otherwise. 

\begin{table*}[h]
{\small
\begin{tabular}{|l|c|c|l|}
\hline
\textbf{Type} & \textbf{Definition} & \textbf{Conditions} &\textbf{Top. Invariance}\\
\hline
GM perversity & $\Z_{\geq 2}\to \Z$  & \begin{tabular}{@{}c}$\bar p(2)=0$\\ $\bar p(k)\leq \bar p(k+1)\leq \bar p(k)+1$ \end{tabular}&\begin{tabular}{@{}l}PL IH \cite{GM1}\\Sheaf IH \cite{GM2}\end{tabular}\\
\hline
King perversity & $\Z_{\geq 1}\to \Z$  & $\bar p(k)\leq \bar p(k+1)\leq \bar p(k)+1$ & \begin{tabular}{@{}l}Singular IH \cite{Ki}\\ Partial for sheaf IH \cite{GBF11} \end{tabular} \\
\hline
\begin{tabular}{@{}l}General\\perversity\end{tabular}&$\{\text{sing. strata}\}\to \Z$&& With conditions \cite{Va14, CST-inv}\\
\hline
ts-perversity&$\{\text{sing. strata}\}\to \Z\times 2^{P(R)}$&&See Section \ref{S: inv}\\
\hline
\end{tabular}}
\caption{Types of perversities and resulting topological invariance}
\end{table*}

It has since become apparent that it is useful to utilize perversities that depend not just on codimension but on the strata themselves so that we define $\bar p:\{\text{singular strata}\}\to \Z$ (if $X$ is an $n$-dimensional CS set, the $n$-dimensional strata are called \emph{regular} and the lower dimensional strata are \emph{singular}). A version of Deligne sheaves suited to such general perversities is defined in \cite{GBF23}, and the corresponding singular chain non-GM intersection homology is studied in this generality in \cite{GBF35}. Clearly in this generality topological invariance becomes a more subtle issue. Nonetheless, there are such results, typically comparing just two stratifications of the same space, $X$ and $\X$, with $X$ refining $\X$ (or, equivalently, $\X$ coarsening $X$). In \cite{Va14}, Valette works with piecewise linear intersection homology on piecewise linear pseudomanifolds  and arbitrary perversities $\bar p:\{\text{singular strata}\}\to \N$ satisfying $\bar p(S)\leq \codim(S)-2$ for each singular stratum $S$. He shows, in our notation, that if $X$ refines $\X$ and if their respective perversities $\bar p$ and $\bar \p$ satisfy $\bar \p(\bS)\leq \bar p(S)\leq \bar \p(\bS)+\codim(S)-\codim(\bS)$ whenever $S$ is a singular stratum of $X$ contained in the singular stratum $\bS$ of $\X$ then the intersection homology groups agree, i.e.\ $I^{\bar p}H_*(X)\cong I^{\bar \p}H_*(\X)$. Note that with Valette's assumptions the GM and non-GM intersection homologies automatically agree.

More recently  in \cite{CST-inv}, Chataur, Saralegi-Aranguren, and Tanr\'e consider what they call $K^*$-perversities and show that a $K^*$-perversity on a CS set $X$ can be pushed forward to a perversity on the intrinsic coarsest stratification $X^*$ and that the two resulting intersection homology groups are isomorphic. This theorem holds for non-GM intersection homology (which is called ``tame intersection homology'' in \cite{CST-inv}), and there is also a version for GM-intersection homology with fewer conditions on the perversities. They also show that it is similarly possibly to pull a $K^*$-perversity back to any refinement of $X$ and obtain isomorphic intersection homology groups. Our results below include the non-GM (tame) intersection homology versions of these theorems as well as those of Valette as special cases.  

\paragraph{Results.} We now outline our results, mostly in order of presentation below. We work throughout from the sheaf-theoretic point of view, which has the benefit of easily allowing for twisted coefficient systems and also in that quasi-isomorphism of sheaves implies isomorphism of the hypercohomology groups with any system of supports. Thus, in particular, our sheaf quasi-isomorphisms imply isomorphisms of intersection homology groups both with compact supports and with closed supports, the latter corresponding to intersection homology of locally-finite singular chains \cite{GBF10, GBF23}. 

In Section \ref{S: background} we review background material, including definitions of ts-perversities, ts-coefficient systems, and ts-Deligne sheaves, all of which generalize the standard versions. In particular, a ts-perversity is a function $\vec p:\{\text{singular strata}\}\to \Z\times 2^{P(R)}$, where $P(R)$ is the set  of primes (up to unit) of our ground PID $R$ and $2^{P(R)}$ is its power set. We write $\vec p(S)=(\vec p_1(S),\vec p_2(S))$. Using this additional information about primes, torsion data is incorporated into the definition of the ts-Deligne sheaf utilizing the ``torsion-tipped truncation'' functor constructed in \cite{GBF32} in place of the standard truncation. If $\vec p_2(S)=\emptyset$ for all $S$ and the ts-coefficient system is just a local system in degree $0$ then the ts-Deligne sheaf reduces to the classical Deligne sheaf  \cite{GM2,GBF23}. Subsection \ref{S: maximal E} discusses some further natural assumptions about coefficient systems that will be utilized in our broadest topological invariance results. 

In Section \ref{S: inv} we define what we call \emph{$\mc E$-compatibility} between ts-perversities $\vec \p$ and $\vec p$ on a CS set $\X$ and its refinement $X$. Here $\mc E$ is a ts-coefficient system common to $\X$ and $X$. This compatibility depends on $\mc E$ only over the regular strata of $\X$ and is necessary to get the compatibility started. From there compatibility is essentially a Goresky-MacPherson-type growth condition, but involving also the torsion information from $\vec p_2$ and $\vec \p_2$. This compatibility condition generalizes that of Valette \cite{Va14}, which itself stems from the Goresky-MacPherson growth condition, by incorporating the torsion information and also allowing $\vec p_1(S)>\codim(S) -2$. The central result of the paper is Theorem \ref{T: top}, which shows that the ts-Deligne sheaves from $\mc E$-compatible ts-perversities are quasi-isomorphic. 
In Sections \ref{S: pullback} and \ref{S: pushforward} we apply Theorem \ref{T: top} to pullback and pushforward perversities, recovering generalizations of the theorems of Chataur-Saralegi-Tanr\'e \cite{CST-inv}. In particular, we discuss pushforwards to arbitrary coarsenings, not just the intrinsic coarsening.

In Section \ref{S: constrained} we consider quasi-isomorphisms of ts-Deligne sheaves arising from two CS set stratifications of a space without the assumption that one refines the other. This requires restricting ourselves to ts-perversities that depend only on codimension, i.e.\ functions $\vec p:\Z_{\geq 1}\to \Z\times 2^{P(R)}$, such that $\vec p_1$ satisfies the Goresky-MacPherson growth condition and $\vec p_2$ satisfies certain growth conditions on sets of primes. These ts-perversities are called \emph{constrained} or \emph{weakly constrained} depending on our requirements for the value of  $\vec p_1(2)$. For constrained ts-perversities that are also appropriately compatible with the ts-coefficient systems $\mc E$ (by a condition relating torsion information about $\mc E$ with $\vec p_2(2)$), we show in Theorem \ref{T: constrained} that any two stratifications yield quasi-isomorphic ts-Deligne sheaves so long as the closures of their codimension one strata agree. In particular, this theorem holds if one makes the classical assumption that codimension one strata are forbidden. We also show in the same theorem that we can weaken the hypotheses to weakly constrained perversities and no compatibility requirement between $\vec p$ and $\mc E$ so long as the the two stratifications have the same regular strata (or, equivalently, the same codimension one skeleta). These two results generalize the classical Goresky-MacPherson topological invariance in \cite{GM2} and that for ``superperversities'' in \cite{GBF11}. The key idea is to apply  Theorem \ref{T: top} using appropriate common coarsenings of the two stratifications. Such intrinsic stratifications, relative to coefficient systems and fixed subspaces, are constructed in Section \ref{S: intrinsic}, generalizing those in \cite{Ki} and \cite{HS91}. 

Section \ref{S: necessity} concerns the extent to which the conditions for $\mc E$-compatibility between ts-perversities are necessary in order to obtain quasi-isomorphic ts-Deligne sheaves. We show that the conditions on singular strata of $X$ contained in regular strata of the coarsening $\X$ are strictly necessary: if they fail for any stratum the sheaves cannot be quasi-isomorphic. By contrast, the conditions on singular strata of $X$ contained in singular strata of the coarsening $\X$ are only ``necessary in general,'' meaning that we can construct examples in which failure of the conditions implies failure of quasi-isomorphism. However, these conditions may not be necessary in special cases, for example if certain stalk cohomology groups vanish due to the specific topology of some space; see Section \ref{S: sing in sing} for further details. One of our main tools in this section will be a formula for computing the ts-Deligne sheaf hypercohomology for a join $S^k*X$ in terms of the ts-Deligne sheaf hypercohomology of $X$; see Corollary \ref{C: sphere join}. This formula is obtained by first computing the hypercohomology for the suspension $\Sigma X$ in Proposition \ref{P: susp}, which is illuminating in its own right, and then making a nice application of Theorem \ref{T: top} to the iterated suspension.

The original Goresky-MacPherson proof of topological invariance involved support and ``cosupport'' axioms concerning the dimensions on which $H^i(f^*_x\mc P^*)$ and $H^i(f^!_x\mc P^*)$ may fail to vanish, $\mc P^*$ being the Deligne sheaf and $f_x$ the inclusion of the point $x$ into $X$. Our arguments to this point do not involve these axioms and so are fundamentally different from those in \cite{GM2}. In Section \ref{S: dimension} we develop versions of these support and cosupport axioms for ts-Deligne sheaves with strongly or weakly constrained ts-perversities. Strongly constrained ts-perversities require $\vec p_1(2)=0$ while simply ``constrained'' is a bit weaker; the stronger constraint in this section is not strictly necessary but simplifies the discussion.
In the strongly constrained case we provide criteria to recognize ts-Deligne sheaves without reference to any specific stratification, leading to Theorem \ref{T: constrained2}, a statement of topological invariance more analogous to the original Goresky-MacPherson invariance theorem of \cite[Uniqueness Theorem]{GM2} or \cite[Theorem 4.15]{Bo}. The weakly constrained version, Theorem \ref{T: constrainedS}, again requires a fixed choice of the regular strata and is more analogous to the main theorem of \cite{GBF11}. 

Lastly,  Section \ref{S: intrinsic} concerns the details about relative intrinsic stratifications.

\paragraph{Remarks.} When $\vec p_2(S)=\emptyset$ for all $S$ and $\mc E$ is a local system concentrated in degree $0$, our ts-Deligne sheaves reduce to the Deligne sheaves of \cite{GM2,GBF23}. With this assumption, many, though not all, of our results reduce to some previously-known theorems, as outlined above. However, we believe that even in these cases our proofs are quite different, as our main invariance results in Section \ref{S: inv} do not require analogues of the Goresky-MacPherson support and cosupport axioms. For the reader interested only in the classical Deligne sheaves and Goresky-MacPherson perversities, we have extracted a simplified version of this new argument and presented it in \cite{GBF44} together with a very short second proof of the topological invariance of classical intersection homology that does use support and cosupport axioms. 

More generally, some of the assumptions below simplify whenever $\mc E$ is just a globally defined local system of free modules, and we attempt to provide some flags in the exposition to help the reader primarily interested in that case. 

We thank J{\"o}rg Sch{\"u}rmann for pointing out some very helpful references and Scott Nollet for many useful conversations. We also thank David Chataur, Martin Saralegi-Aranguren, and Daniel Tanr\'e for both ongoing stimulating mathematical discussion and their generous hospitality. Finally, we thank the anonymous referee for suggesting several improvements to the exposition. 

\section{Definitions and background}\label{S: background}

\subsection{Spaces} 
Our spaces will be paracompact dimensionally homogeneous CS sets, whose precise definition we recall below. CS sets include topological and piecewise linear pseudomanifolds. In fact, we have the hierarchy
$$\{\text{PL pseudomanifolds}\} \subset \{\text{topological pseudomanifolds}\}\subset \{\text{CS sets}\}.$$
The primary difference between CS sets and topological stratified pseudomanifolds, given our additional dimensional homogeneity condition, is that the links of points of pseudomanifolds must themselves be stratified pseudomanifolds, while the links of points in CS sets need only be compact filtered spaces. PL pseudomanifolds are defined just as topological pseudomanifolds are with the added condition that all spaces and maps describing local conditions must be piecewise linear. Classical PL pseudomanifolds, which are those simplicial complexes consisting exclusively of $n$-simplices such that each $n-1$ face of each $n$-simplex is glued to exactly one $n-1$ face of another $n$-simplex, are a special case. 
All irreducible complex algebraic and analytic varieties can be given stratifications that realize them as PL pseudomanifolds; this is also true of real varieties that possess a dense manifold subset.
Connected orbit spaces of manifolds under smooth actions of compact Lie groups are also PL pseudomanifolds. For more details and an overview of all these spaces, see \cite[Chapter 2]{GBF35}.

We choose to work with CS sets both for their added generality but also because one of our key tools will be intrinsic stratification and the intrinsic stratification of a CS set is also a CS set. In \cite{GM2}, Goresky and MacPherson construct ``canonical $\bar p$-filtrations'' for topological pseudomanifolds, but, in addition to depending on the choice of perversity, the resulting filtrations do not necessarily give the space the structure of a stratified pseudomanifold. 
So, while Goresky-MacPherson \cite{GM2} and Borel \cite{Bo} treat topological pseudomanifolds as the primary objects, we follow King \cite{Ki} by working in the even more general but self-contained class of CS sets. 
The class of PL pseudomanifolds is also preserved on passing to intrinsic stratifications (see \cite[Corollary 2.10.19.]{GBF35}), but we note that topological pseudomanifolds include such important non-PL examples as suspensions of topological manifolds and topological manifolds stratified by the inclusion of locally-flat, but not PL, embedded submanifolds. This includes, for example, locally flat topological knots in high dimensions. Moreover, as our fundamental language is sheaf theory, the added generality of CS sets does not much increase the difficulty of our results, and in some cases this choice provides simplifications. For example, we don't at any point need to take care about PL structures.

We now recall the definition of CS sets due to Siebenmann \cite{Si72}. An $n$-dimensional CS set $X$ is a Hausdorff space equipped with a filtration $$X=X^n\supset X^{n-1}\supset\cdots X^0\supset X^{-1}=\emptyset$$ such that $X_k:=X^k-X^{k-1}$ is a $k$-manifold (possibly empty) and for $x\in X_k$ there is an open neighborhood $U$ of $x$ in $X_k$, an open neighborhood
$N$ of $x$ in $X$, a compact filtered space $L$ (which may be empty), and, letting $cL$ denote the open cone on $L$, a homeomorphism
$h:U\times cL \to N$ such that $h(U \times c(L^j )) = X^{k+j+1}\cap N$ for all $j$. The space $L$ is called a \emph{link} of $x$, and $N$ is called a \emph{distinguished neighborhood} of $x$. Note that if $L=\emptyset$ then $cL=(cL)^0$ is a point.  Dimensional homogeneity means that we assume $X-X^{n-1}$ is dense. Such spaces are locally compact \cite[Lemma 2.3.15]{GBF35}, metrizable \cite[Proposition 1.11]{CST-inv}, and of finite cohomological dimension (\cite[Lemma 6.3.46]{GBF35} and \cite[Theorem II.16.8]{Br}). 
See \cite[Section 2.3]{GBF35} for more details about CS sets in general. All CS sets in this paper will be assumed paracompact and dimensionally homogeneous without further mention. We also assume $X$ is $n$-dimensional unless specified otherwise.

Following Borel \cite[Section V.2]{Bo}, we let $U_k=X-X^{n-k}$, and noting that $U_{k+1}$ is the disjoint union of $U_k$ and $X_{n-k}$, we also take $i_k:U_k\into U_{k+1}$ and  
$j_k:X_{n-k}\into U_{k+1}$. For any $x\in X$, we write $f_x:\{x\}\into X$.
The connected components of $X_k$ are the $k$-dimensional \emph{strata}. 
Strata in  $X_n=X^n-X^{n-1}$ are \emph{regular strata} and strata in $X_k$ for $k\leq n-1$ are  \emph{singular strata}. 
Note that strata may have codimension one, which is sometimes forbidden in other contexts. 

We often abuse notation and use $X$ to refer both to the underlying space and to the space equipped with the stratification; when we wish to emphasize the underlying space or do not yet want to specify the stratification we also write $|X|$.
If $X$ and $\X$ are two stratifications of the same space $|X|$, we say that $\X$ \emph{coarsens} $X$, or that $X$ is a \emph{refinement} of $\X$, if each stratum of $\X$ is a union of strata of $X$. Our standard notation will be $X$ for a CS set and $\X$ for a coarsening of $X$. We will use $\mf X$ for the intrinsic stratifications constructed in Section \ref{S: intrinsic}. If we wish to speak of the same space with two \emph{a priori} unrelated stratifications, we write the stratifications $\mc X$ and $\mc Y$; if we construct a common coarsening of $\mc X$ and $\mc Y$, we will sometimes call that $\mc Z$. 

\subsection{Algebra}
Algebraically, we fix a PID $R$ as our ground ring throughout, and we let $P(R)$ be the set of primes of $R$ up to unit. This means that the elements of $P(R)$ are technically equivalence classes such that $p\sim q$ if $p=uq$ for some unit $u$, though we will abuse notation by letting a prime stand for its equivalence class; cf.\ \cite[Section 2]{GBF32}.
The following is Definition 2.1 of \cite{GBF32}. 

\begin{definition}
If $A$ is a finitely-generated $R$-module  and $\wp \subset P(R)$, we define the \emph{$\wp$-torsion submodule of $A$} to be 
$$T^{\wp}A =\left\{x\in A\mid nx=0\text{ for some  product } n=\prod^s_{i=1} p_i^{m_i}\text{ such that $p_i\in\wp$ and $m_i,s\in \Z_{\geq 0}$}\right\},$$
i.e.\ $T^{\wp}A$ is the submodule annihilated by products of powers of primes in $\wp$. If $T^\wp A=A$, we say that $A$ is \emph{$\wp$-torsion}. If $T^{\wp}A=0$, we say that $A$ is \emph{$\wp$-torsion free.} We take the empty product to be $1$, so in particular if $\wp=\emptyset$ then $T^{\wp}A=0$ and every $A$ is $\emptyset$-torsion free. If $\mf p\in P(R)$ is a single element, we abuse notation and write $T^{\mf p}A$ instead of $T^{\{\mf p\}}A$.
\end{definition}

\subsection{ts-Deligne sheaves} We now recall some material from \cite{GBF32}, leading to the definition of ts-Deligne sheaves. All sheaves are sheaves of $R$-modules, and we think of ourselves as working in the derived category so that $\cong$ denotes quasi-isomorphism. If $\mc S^*$ is a sheaf complex, then $\mc H^i(\mc S^*)$ denotes the derived cohomology sheaf and $\H^i(X;\mc S^*)$ denotes hypercohomology.

We begin with ts-perversities  \cite[Definitions 4.1 and 4.18]{GBF32}:

\begin{definition}\label{D: ts-perv}
For a PID $R$, let $\primeset{P}(R)$ be the set of primes of $R$ (up to unit), and let $2^{\primeset{P}(R)}$ be its power set.
A \emph{torsion-sensitive perversity} (or simply \emph{ts-perversity}) on a CS set $X$ is a function $\vec p: \{\text{singular strata of $X$}\}\to \Z\times 2^{\primeset{P}(R)}$. We denote the components of $\vec p(S)$ by $\vec p(S)=(\vec p_1(S),\vec p_2(S))$. 

The \emph{complementary ts-perversity}, or \emph{dual ts-perversity}, $D\vec p$ is defined by $D\vec p(S)=(\codim(S)-2-\vec p_1(S),P(R)-\vec p_2(S))$, i.e.\ the first component is the complementary perversity to $\bar p$ in the Goresky-MacPherson sense and the second component is the set of primes in $R$ complementary to $\vec p_2(S)$.
\end{definition}

We also recall the notion of a $\wp$-coefficient system, slightly generalizing \cite[Definition 4.2]{GBF32}. On a pseudomanifold, these are the objects in the heart of the natural t-structures ${}^\wp D^\heartsuit$ constructed in \cite[Definition 5.1]{GBF32}. We give an explicit description here. 

\begin{definition}\label{D: ts-coeff}
Let $\wp \subset \primeset{P}(R)$ be a set of primes of the PID $R$. We will call a complex of sheaves\footnote{Even though $\mc E$ is a complex of sheaves, we do not write $\mc E^*$ in order to emphasize the role of $\mc E$ as coefficients.}  $\mc E$ on a space $M$ a \emph{$\wp$-coefficient system} if 
\begin{enumerate}
\item  $\mc H^1(\mc E)$ is a locally constant sheaf of finitely generated $\wp$-torsion modules,
\item $\mc H^0(\mc E)$ is a locally constant sheaf of finitely generated $\wp$-torsion free modules, and
\item $\mc H^i(\mc E)=0$ for $i\neq 0,1$.
\end{enumerate}
More generally, if $M$ is a disjoint union of spaces, we call $\mc E$ a \emph{ts-coefficient system} if it restricts on each component of $M$ to a $\wp$-coefficient system for some $\wp$ (which may vary by component).

Suppose $X$ is a CS set and $\mc E$ is a ts-coefficient system defined over a subset $U\subset X$. We call $U$  the \emph{domain} of $\mc E$, denoted by $\text{Dom}(\mc E)$.
We say that the stratification of $X$ is \emph{adapted to $\mc E$} if $X-X^{n-1}\subset \text{Dom}(\mc E)$, i.e.\ if $\mc E$ is defined on (at least) the regular strata of $X$; cf.\ \cite[Section V.4.12]{Bo}. Of course this is automatically satisfied if $\mc E$ is defined on all of $X$, for example if $\mc E$ is a local system on $X$. 
\end{definition}

\begin{remark}\label{R: maladapted}
As noted in \cite[Remark V.4.14.c]{Bo}, even if we restrict our coefficients to local systems concentrated in a single degree and defined on dense open submanifolds of pseudomanifolds, there can exist $\mc E$ for which there does not exist a CS set stratification adapted to $\mc E$. For example, define $\mc E$ to be the local system on $\R^2-(0,0)\cup\{(1/n,0)|n\in \Z_{\geq 1}\}$ with stalk $\Z$ and nontrivial monodromy on a small loop around each $(1/n,0)$.
\end{remark}

Given a ts-perversity $\vec p$ and a ts-coefficient system $\mc E$ to which $X$ is adapted, the associated ts-Deligne sheaf is defined for pseudomanifolds in \cite[Definition 4.4]{GBF32} and shown there for pseudomanifolds to be an intermediate extension of $\mc E$ with respect to a certain $t$-structure \cite[Proposition 5.12]{GBF32}. The construction holds as well for CS sets, and the ts-Deligne sheaf $\mc P^*_{X,\vec p,\mc E}$, often written simply as $\mc P^*$, is defined as  

$$\mc P^*_{X,\vec p,\mc E}=  \mf t_{\leq \vec p}^{X_0}Ri_{n*}\ldots\mf t_{\leq \vec p}^{X_{n-1}}Ri_{1*}\mc E.$$
Here each $\mf t_{\leq \vec p}^{X_{k}}$ is a \emph{locally torsion-tipped truncation} functor as defined in \cite[Section 3]{GBF32}. We refer the reader there for more details but note that for $\mc S^*$ defined on $U_{k+1}$ we have

\begin{enumerate}
\item $\left(\mf t_{\leq \vec p}^{X_{n-k}}\mc S^*\right)_x=\mc S^*_x$ if $x\in U_k$,

\item if $x\in S\subset  S_{n-k}$ for a singular stratum $S$ then 

\begin{equation*}
H^i\left(\left(\mf t_{\leq \vec p}^{X_{n-k}}\mc S^*\right)_x\right)\cong \begin{cases}
0,& i>\vec p_1(S)+1,\\
T^{\vec p_2(S)}H^i(\mc S^*_x),& i= \vec p_1(S)+1,\\
H^i(\mc S^*_x),& i\leq \vec p_1(S).
\end{cases}
\end{equation*}
\end{enumerate}

If $\mc E$ is a local system (i.e.\ a locally constant sheaf of finitely generated $R$-modules) concentrated in degree $0$, if $\vec p_1$ satisfies the Goresky-MacPherson conditions, and if $\vec p_2(S)=\emptyset$ for all $S$, then this is just the classical Deligne sheaf of Goresky and MacPherson from \cite{GM2}. 

As for the traditional Deligne sheaves, the critical feature of the ts-Deligne sheaves is the ``cone formula.'' As we will prove below in Lemma \ref{L: ngbd}, if $x$ is a point of a CS set with a neighborhood of the form $\R^k\times cL$ and if $S$ is the stratum containing the cone points, then 
\begin{equation*}
H^i(\mc P^*_x)\cong 
\begin{cases}
0,&i>\vec p_1(S)+1,\\
T^{\vec p_2(S)}\H^i(L;\mc P^*|_L),&i=\vec p_1(S)+1,\\
\H^i(L;\mc P^*|_L),&i\leq \vec p_1(S).
\end{cases}
\end{equation*}
We use this below to compute the hypercohomology of a suspension in Proposition \ref{P: susp}.

Analogously to the Goresky-MacPherson Deligne sheaves, the ts-Deligne sheaves can be characterized by axioms. 
Here is the first set of axioms from \cite[Definition 4.7]{GBF32}, generalized for CS sets. We write $\ms S^*_k$ for $\ms S^*|_{U_k}$.

\begin{definition}
Let $X$ be an $n$-dimensional CS set, and let $\mc E$ be a ts-coefficient system on $U_1$. 
 We say that the sheaf complex $\ms S^*$ on $X$ satisfies the \emph{Axioms TAx1$(X,\vec p, \mc E)$} if 

\begin{enumerate}[label=\alph*., ref=\alph*]
\item\label{A: bounded} $\ms S^*$ is quasi-isomorphic to a complex that is bounded and that is $0$ in negative degrees;
\item \label{A: coeffs} $\ms S^*|_{U_1}\cong\mc E|_{U_1}$;
\item \label{A: truncate} if $x\in S\subset X_{n-k}$, where $S$ is a singular stratum, then
 $H^i(\ms S_x)=0$ for $i>\vec p_1(S)+1$ and $H^{\vec p_1(S)+1}(\ms S_x)$ is $\vec p_2(S)$-torsion;
\item \label{A: attach} if $x\in S\subset X_{n-k}$, where $S$ is a singular stratum, then the attachment map $\alpha_k:\ms S_{k+1}\to Ri_{k*}\ms S_k$ induces stalkwise cohomology isomorphisms at $x$  in degrees $\leq \vec p_1(S)$ and it induces stalkwise cohomology isomorphisms $H^{\vec p_1(S)+1}(\ms S_{k+1,x})\to T^{\vec p_2(S)}H^{\vec p_1(S)+1}( (Ri_{k*}\ms S_k)_x)$. 
\end{enumerate} 
\end{definition}

Theorem 4.8 of \cite{GBF32}, which also works for CS sets, shows that  the ts-Deligne sheaf complex  $\mc P^*_{X,\vec p,\mc E}$  satisfies the axioms  TAx1$(X,\vec p, \mc E)$, and conversely any sheaf complex satisfying  TAx1$(X,\vec p, \mc E)$ is quasi-isomorphic to $\mc P^*_{X,\vec p,\mc E}$. It is also observed in \cite[Theorem 4.10]{GBF32} that these sheaf complexes are $X$-clc, meaning that each sheaf $\mc H^i(\mc P^*)$ is locally constant on each stratum. This also continues to hold for CS sets, which have the property that if $j$ is any inclusion of a locally closed subset that is a union of strata then $j^*$, $j^!$, $j_!$, and $Rj_*$ all preserve  \emph{constructibility} by \cite[Proposition 4.0.2.3]{Sch03} (see also \cite[Proposition 4.2.1.2.b]{Sch03}).

As in \cite{GM2, Bo, GBF32}, we can reformulate some of these axioms.

\begin{definition}\label{T: Ax1'}
We say $\ms S^*$ satisfies the \emph{Axioms TAx1'$(X,\vec p, \mc E)$} if 

\begin{enumerate}[label=\alph*., ref=\alph*]
\item\label{A': bounded} $\ms S^*$ is $X$-clc and it is 
quasi-isomorphic to a complex that is bounded and that is $0$ in negative degrees;
\item \label{A': coeffs} $\ms S^*|_{U_1}\cong\mc E|_{U_1}$;
\item \label{A': truncate} if $x\in S\subset X_{n-k}$, where $S$ is a singular stratum, then
$H^i(\ms S^*_x)=0$ for $i>\vec p_1(S)+1$ and $H^{\vec p_1(S)+1}(\ms S^*_x)$ is $\vec p_2(S)$-torsion;

\item \label{A': attach}  if $x\in S\subset X_{n-k}$, where $S$ is a singular stratum, and $f_x:x\into X$ is the inclusion, then

\begin{enumerate}
\item  $H^i(f_x^!\ms S)=0$ for $i\leq \vec p_1(S)+n-k+1$
\item $H^{\vec p_1(S)+n-k+2}(f_x^!\ms S)$ is $\vec p_2(S)$-torsion free.

\end{enumerate}
\end{enumerate}
\end{definition}

The following theorem is a slight generalization of \cite[Theorem 4.13]{GBF32}:

\begin{theorem}\label{T: ax equiv}
On a CS set, the axioms TAx1'$(X,\vec p, \mc E)$ are equivalent to the axioms TAx1$(X,\vec p, \mc E)$ and so any sheaf complex satisfying TAx1'$(X,\vec p, \mc E)$ is quasi-isomorphic to $\mc P^*_{X,\vec p,\mc E}$.
\end{theorem}

The proof is the same as that of  \cite[Theorem 4.13]{GBF32}, replacing the theorems about constructibility invoked from Borel (e.g.\ \cite[Lemma V.3.10.d]{Bo}) with Sch{\"u}rmann's \cite[Proposition 4.0.2.3]{Sch03}.

\subsection{Maximal ts-coefficient systems}\label{S: maximal E}

Many of our theorems below compare ts-Deligne sheaves on two stratifications of a single CS set, one coarsening another. For this it suffices to have a ts-coefficient system $\mc E$ defined on the regular strata of the coarser of the two stratifications for then it restricts also to a ts-coefficient system on the regular strata of the finer stratification. However, we will also be interested in theorems concerning arbitrary stratifications, and in these cases we will need to construct common coarsenings that remain adapted to $\mc E$. The full details will be provided in Section \ref{S: intrinsic}, though we discuss here some notions about coefficient systems that will be necessary at that point as these will also be needed in some of our earlier theorem statements. In particular, to construct these common coarsenings we will need to make some minor assumptions about the domains of our ts-coefficient systems.

To motivate our restrictions, we recall that for classical intersection homology theory on a stratified pseudomanifold $X$ it is observed by Borel in \cite[Section V.4]{Bo} that 
if $E$ is a local system (i.e.\ a locally constant sheaf of finitely-generated $R$-modules) defined on a dense open submanifold  of $X$ whose complement has codimension  $\geq 2$ then there is a unique largest submanifold of $X$ to which $E$ can be extended uniquely\footnote{This is no longer true if the local system is only defined on a dense open set whose complement has codimension 1. For example let $X=S^1$ with stratification $S^1\supset \{pt\}$. Suppose $E$ is the constant sheaf with stalk $\Z$ on $S^1-\{pt\}$. Then there are two non-isomorphic extensions of $E$ to $S^1$, namely the constant sheaf with stalk $\Z$ and the twisted sheaf with stalks $\Z$ such that a generator of $\pi_1(S^1)$ acts by multiplication by $-1$.
} up to isomorphism \cite[Lemma V.4.11]{Bo} (though this submanifold may  not necessarily be the largest $n$-dimensional manifold contained in  $X$ due to monodromy). 
Since it is not clear that such a statement holds for the more general ts-coefficient systems, we instead build a maximality assumption into our coefficients when necessary. Since local systems have unique such maximal extensions, we can convince ourselves that we therefore do not lose much generality.  Alternatively, if we limit ourselves to $\mc E$ composed of local systems, then Proposition \ref{P: local} below shows that maximality can be guaranteed.

\begin{definition}\label{D: maximal}
Let $X$ be an $n$-dimensional CS set. We will call a sheaf complex $\mc E$ on $X$ a \emph{maximal ts-coefficient system} if

\begin{enumerate}
\item $\text{Dom}(\mc E)$ includes an open $n$-dimensional submanifold $U_{\mc E}$ of $X$ whose complement has codimension $\geq 2$,

\item $\mc E$ is a ts-coefficient system over $U_{\mc E}$ (see Definition \ref{D: ts-coeff}), and

\item there is no larger submanifold of $X$ to which $\mc E$ extends as a ts-coefficient system.

\end{enumerate}
\end{definition}

Clearly ts-coefficient systems composed of constant sheaves defined on all of $X$ are maximal. The following lemma shows that ts-coefficient systems composed of locally constant sheaves (on open submanifolds of codimension at least 2) can be made maximal.

\begin{proposition} \label{P: local}
Suppose $\mc E$ is a ts-coefficient system defined on an open dense submanifold whose complement has codimension at least 2. If $\mc E$ is bounded (i.e.\ $\mc E^i=0$ for sufficiently large $|i|$) and each $\mc E^i$ is a local system (a locally constant sheaf of finitely-generated $R$-modules), then $\mc E$ has a maximal extension that is unique up to isomorphism. Furthermore, if $X$ is adapted to $\mc E$ then $X$ remains adapted to the extension.
\end{proposition}
\begin{proof}
By assumption, each $\mc E^i$ is defined on an open dense submanifold $U\subset X$ whose complement has codimension at least 2, and so by \cite[Lemma V.4.11]{Bo} each $\mc E^i$ has a unique (up to isomorphism) extension $\td {\mc E}^i$ to a maximal open subset $U_i$. Let $W=\cap_i U_i$, which remains open and dense since all but finitely many of the $U_i$ will be the largest open submanifold of $X$. Since  $U\subset U_i$, we also have $U\subset W$, and we let $\bar {\mc E}^i=\td{\mc E}^i|_W $. Also by  \cite[Lemma V.4.11]{Bo}, each boundary map $\mc E^i\to \mc E^{i+1}$ extends uniquely to a map $\bar{\mc E^i}\to \bar{\mc E}^{i+1}$. This gives us a unique (up to isomorphisms) complex $\bar {\mc E}^*$ on $W$ that cannot be extended to a larger submanifold of $X$.

The last statement of the lemma is trivial.
\end{proof}

Another nice property of local systems  is that if $E$ is a maximal local system, $X$ is adapted to $E$ with no codimension one strata, and $U_E$ is the maximal submanifold over which $E$ is defined, then $U_E$ is a union of strata of $X$. This is shown at the bottom of \cite[page 92]{Bo}. We will also need a property like this to define our common coarsenings, which motivates the following definition. Once again we will then show that this condition is automatic when $\mc E$ consists of local systems and there are no codimension one strata.

\begin{definition}\label{D: fully}
Suppose $\mc E$ is a maximal ts-coefficient system on $X$ and that $U_{\mc E}$ is the largest open submanifold on which $\mc E$ is defined. We say that the stratification of $X$ is \emph{fully adapted to $\mc E$} if 

\begin{enumerate}
\item $X-X^{n-1}\subset \mc U_{\mc E}$, and

\item $\mc U_{\mc E}$ is a union of strata of $X$.
\end{enumerate}
\end{definition}

\begin{proposition}
Suppose $\mc E$ is a maximal ts-coefficient system on $X$ such that each $\mc E^i$ is a local system. If $X$ has no codimension one strata and is adapted to $\mc E$ then it is fully adapted.
\end{proposition}
\begin{proof}
The proof is essentially the same as the argument on \cite[page 92]{Bo}: We will proceed by contradiction. Let $S$ be a stratum of $X$ of minimal codimension so that $S$ intersects $U_{\mc E}$ but is not contained in it. Since $X$ is adapted to $\mc E$ we must have $\codim(S)\geq 1$. Suppose $x\in S$ has a distinguished neighborhood $N\cong B\times cL$ such that there is some point $y\in B\times \{v\}$ (with $v$ the cone vertex) such that $y\in U_{\mc E}$. We claim that then $x\in U_{\mc E}$. 

First, since $x$ and $y$ are in the same stratum of $X$ and as $y$ must have a Euclidean neighborhood in $X$ (since $y\in U_{\mc E}$), \cite[Lemma 2.10.4]{GBF35} implies that $x$ also has a Euclidean neighborhood; thus both $x$ and $y$ (and similarly all points of $B\times \{v\}$) are contained in the maximal submanifold of $X$. Furthermore, by assumption $U_{\mc E}$ must contain $N-(N\cap S)=B\times (cL-\{v\})$, and if $\pi:B\times cL\to cL$ is the projection then $\pi^*(\mc E|_{\{y\}\times cL})$ is a local system on $N$ whose restriction to $N-S$ is isomorphic to $\mc E|_{N-S}$. Since extension of local systems is unique when there are no codimension one strata by \cite[Lemma V.4.11]{Bo},  $\pi^*(\mc E|_{\{y\}\times cL})$ must agree with $\mc E$ where they overlap, and so we must have $N\subset U_{\mc E}$ or else the maximality of $U_{\mc E}$ would be contradicted. 

Now, since $U_{\mc E}$ is open in $X$, we have $U_{\mc E}\cap S$ open in $S$. The above argument shows that if $x$ is in the closure of $U_{\mc E}\cap S$ in $S$ then $x\in U_{\mc E}\cap S$. So $U_{\mc E}\cap S$ is open and closed in the connected set $S$ and is thus all of $S$. 
\end{proof}

The preceding proposition can fail if there are codimension one strata:

\begin{example} Let $E$ be the local system on $\R^2-\{0\}$ with $\Z$ stalks and nontrivial monodromy around the origin. Let $X=\R^2$ filtered as $\R^2\supset \text{$x$-axis}$. Then $E$ is maximal and $X$ is adapted to $E$, but it is not fully adapted, though it can be refined to be so. 

As a more dramatic example, consider the example from Remark \ref{R: maladapted} of a maximal local system $E$ that is defined on the complement in $\R^2$ of $(0,0)\cup\{(0,1/n)|n\in \Z_{\geq 1}\}$. If we again filter $X=\R^2$ as $\R^2\supset \text{$x$-axis}$ then again $X$ is adapted to $E$, but there is no fully adapted refinement. 
\end{example}

\section{Topological invariance}\label{S: inv}

In this section we prove our main topological invariance theorems. These are mostly sufficiency statements, demonstrating that if certain conditions hold between different perversities on different stratifications of the same space, as well as certain relations between the perversity on the more refined stratification and the ts-coefficient system, then the two corresponding ts-Deligne sheaves are quasi-isomorphic. We consider necessity in Section \ref{S: necessity}.

The following definition establishes our main criteria for comparison of ts-Deligne sheaves. The first set of conditions is in a sense more important, as the second set is satisfied automatically for sufficiently simple coefficient systems -- see Remark \ref{R: simple system}.  

\begin{definition}\label{D: comp perv}
Suppose that $X$ and $\X$ denote two CS set stratifications of the same underlying space with $\X$ coarsening $X$. Let $\vec p$ and $\vec \p$ be respective ts-perversities on $X$ and $\X$, and let $\mc E$ be a ts-coefficient system to which $\X$ (and hence also $X$) is adapted. We will say that $\vec p$ and $\vec \p$ are \emph{$\mc E$-compatible} ts-perversities  if the following conditions hold whenever a singular stratum $S$ of $X$ is contained in a (singular or regular) stratum $\bS$ of $\X$:

\begin{enumerate}
\item\label{I: sing in sing} If $\bS$ is singular then  $\vec \p_1(\bS)\leq  \vec p_1(S)\leq\vec \p_1(\bS)+\codim(S)-\codim(\bS)$, and furthermore

\begin{enumerate}

\item if $\vec p_1(S)=\vec \p_1(\bS)$ then $\vec p_2(S)\supset\vec \p_2(\bS)$, 

\item if $\vec p_1(S)=\vec \p_1(\bS)+\codim(S)-\codim(\bS)$, then $\vec p_2(S)\subset\vec \p_2(\bS)$.
\end{enumerate}

\item\label{I: sing in reg}  If $\bS$ is regular then $-1\leq \vec p_1(S)\leq \codim(S)-1$, and furthermore

\begin{enumerate}

\item if $\vec p_1(S)=-1$ then $H^1(\mc E_x)=0$ and $H^0(\mc E_x)$ is $\vec p_2(S)$-torsion for all $x\in S$,

\item if $\vec p_1(S)=0$ then $H^1(\mc E_x)$ is $\vec p_2(S)$-torsion for all $x\in S$,

\item if $\vec p_1(S)= \codim(S)-2$ then $H^0(\mc E_x)$ is $\vec p_2(S)$-torsion free for all $x\in S$,

\item if $\vec p_1(S)=\codim(S)-1$  then $H^0(\mc E_x)=0$ and $H^1(\mc E_x)$ is $\vec p_2(S)$-torsion free for all $x\in S$.
\end{enumerate}
\end{enumerate}

\end{definition}

\begin{remark}\label{R: simple system}
If $0\leq \vec p_1(S)\leq \codim(S)-2$ for all $S$ and $\mc E$ is a globally defined local system of free modules concentrated in degree $0$, then all of condition \ref{I: sing in reg} holds automatically.
\end{remark}

\begin{remark}
Since $\mc E$ is clc on the regular strata by assumption, the torsion conditions of property \ref{I: sing in reg} hold for all $x\in S$ if and only if they hold for some $x\in S$.
\end{remark}

\begin{remark}\label{R: dichot}
We notice that if $\bS$ is singular then compatibility as just defined places no absolute constraints on the values of $\vec p(S)$ and $\vec \p(\bS)$ but only relative constraints on how these must relate to each other. By contrast, if $\bS$ is regular then there are absolute constraints on the values of $\vec p_1(S)$ and also, for the extreme values of $\vec p_1(S)$, constraints on how $\vec p_2(S)$ relates to $\mc E$. We will see such a  dichotomy throughout. One consequence is that these conditions forbid any codimension one stratum of $X$ being contained in a regular stratum of $\X$ (unless $\mc E$ is trivial on that regular stratum) as this would require either $\vec p_1(S)=-1=\codim(S)-2$ or $\vec p_1(S)=0=\codim(S)-1$, and in either case the combined torsion assumptions imply each $H^*(\mc E_x)=0$ so that $\mc E$ is trivial. 
\end{remark}

\begin{remark}\label{R: efficient}
One situation in which the conditions of Definition \ref{D: comp perv} are \emph{not} necessary for topological invariance is when perversity values are so extreme that their specific values become irrelevant. For example, if $\vec p_1(S)<-1$ then $\mc P^*_x=0$ for any $x\in S$ regardless of the actual value of $\vec p(S)$. At the other extreme,  \cite[Theorem 4.15]{GBF32} implies that if $X$ is a stratified pseudomanifold, $\mc P^*$ a ts-Deligne sheaf on $X$, and $x\in X_{n-k}$ then $H^i(\mc P^*_x)=0$ for $i>k$. The main technical tool in the proof of that theorem is  \cite[Lemma 4.14]{GBF32}, but the argument for this lemma applies to any manifold stratified space (cf.\ \cite[Lemma V.9.5]{Bo}). The proof of  \cite[Theorem 4.15]{GBF32} therefore generalizes to CS sets, using Lemma \ref{L: prod} below in the argument instead of the citation to \cite[Lemma V.3.8.b]{Bo}. Consequently, if  $\vec p_1(S)\geq k$ then again the specific values don't matter. Therefore, for the purposes of Theorem \ref{T: top} we could add to the conditions of Definition \ref{D: comp perv} the possibility that if $S\subset \bS$ then either both $\vec p_1(S)$ and $\vec \p_1(\bS)$ are $<-1$ or that they are both sufficiently large. However, rather than complicate Definition \ref{D: comp perv} even further, when considering the necessity of the conditions it is reasonable to assume that the ts-perversities are \emph{efficient}, i.e.\ that $-1\leq \vec p_1(S)\leq \codim(S)$ for all singular strata $S$. If $\vec p$ is not efficient, it can always be replaced by a ts-perversity that is efficient without altering $\mc P^*$. In any case, Theorem \ref{T: top} does not require such assumptions.
\end{remark}

We now come to the main theorem of the paper, which will be the basis for the results in the rest of Section \ref{S: inv}. 

\begin{theorem}\label{T: top}
Suppose that $X$ and $\X$ denote two CS set stratifications of the same underlying space with $\X$ coarsening $X$. Let $\mc E$ be a ts-coefficient system such that $\X$ is adapted to $\mc E$ (and hence so is $X$).  Let $\vec p$ and $\vec \p$ be respective ts-perversities on $X$ and $\X$ that are $\mc E$-compatible, and let $\mc P^*$ and $\PP^*$ be the respective ts-Deligne sheaves with coefficients $\mc E$. 
Then $\mc P^*$ is quasi-isomorphic to $\PP^*$.
\end{theorem}

\begin{proof}
By Theorem \ref{T: ax equiv}, $\mc P^*$ is characterized uniquely up to quasi-isomorphism by the axioms TAx1'$(X,\vec p,\mc E)$. Therefore, to prove the proposition, it is sufficient to show that $\PP^*$ also satisfies these axioms. We will use that $\PP^*$ already satisfies  the axioms TAx1'$(\X,\vec \p,\mc E)$.

\emph{Axiom \ref{A': bounded}.} Since  $\PP^*$ is $\X$-clc, it is also $X$-clc since $X$ refines $\X$.  Furthermore, $\PP^*$ is $0$ for $*<0$ by construction, and the cohomology is nontrivial in a finite range of degrees by Axioms b and c, which we verify below, and by the definition of ts-coefficient systems.

\emph{Axiom \ref{A': coeffs}.}
 We know $\PP^*|_{\X-\X^{n-1}}\cong \mc E|_{\X-\X^{n-1}}$, but   $X-X^{n-1}\subset \X-\X^{n-1}$ by assumption, so also 
$\PP^*|_{X-X^{n-1}}\cong \mc E|_{X-X^{n-1}}$.

\emph{Axiom \ref{A': truncate}.} Suppose $x\in S\subset X_{n-k}$ for $k\geq 1$. We must show that $H^i(\PP^*_x)=0$ for $i>\vec p_1(S)+1$ and that  $H^{\vec p_1(S)+1}(\PP^*_x)$ is $\vec p_2(S)$-torsion.

 First suppose that $x$ is contained in a regular stratum of $\X$. Then $\PP^*_x\cong \mc E_x$, so $H^i(\PP^*_x)\cong H^i(\mc E_x)$. We recall that $H^i(\mc E_x)$ is automatically $0$ if $i\neq 0,1$ and that $-1\leq \vec p_1(S)\leq \codim(S)-1$ by the compatibility assumption. If $\vec p_1(S)\geq 1$, then $\vec p_1(S)+1\geq 2$ and
 $H^i(\mc E_x)=0$ for $i\geq\vec p_1(S)+1$. If $\vec p_1(S)=0$ then $H^i(\mc E_x)=0$ for $i>\vec p_1(S)+1=1$ while $H^{\vec p_1(S)+1}(\mc E_x)=H^1(\mc E_x)$ is $\vec p_2(S)$-torsion in this case by the $\mc E$-compatibility assumptions. Finally, if $\vec p_1(S)=-1$, the compatibility assumptions imply that $H^1(\mc E_x)=0$, and so $H^i(\mc E_x)=0$ for $i>0$, while $H^0(\mc E_x)$ is $\vec p_2(S)$-torsion\footnote{We see in this argument why $\vec p_1(S)\geq -1$ is required, as well as our torsion assumptions; see also Section \ref{S: sing in reg}, below.}.

Now suppose that $x$ is contained in the singular stratum $\bS$ of $\X$. Then we know $H^i(\PP^*_x)=0$ for $i>\vec \p_1(\bS)+1$ and $H^{\vec \p(\bS)+1}(\PP^*_x)$ is $\vec \p_2(\bS)$-torsion. But by assumption $\vec \p_1(\bS)\leq  \vec p_1(S)$, and if $\vec \p_1(\bS)= \vec p_1(S)$ then $\vec \p_2(\bS)\subset \vec p_2(S)$. It follows that $\PP^*$ satisfies Axiom c of TAx1'$(X,\vec p,\mc E)$.

\emph{Axiom \ref{A': attach}.}  Again suppose  $x\in S\subset X_{n-k}$ for $k\geq 1$, and let $f_x:\{x\}\into X$ be the inclusion. 
We must show that $H^i(f_x^!\PP^*)=0$ for $i\leq \vec p_1(S)+n-k+1$ and that it is $\vec p_2(S)$-torsion free when $i=\vec p_1(S)+n-k+2$. 

Once again we first suppose that $x$ is contained in a regular stratum of $\X$. Then we have $f_x^!\PP^*\cong f_x^*\PP^*[-n]\cong \mc E_x[-n]$ by \cite[Proposition V.3.7.b]{Bo} and by assumption, and so  $H^i(f_x^!\PP^*)\cong H^i(\mc E_x[-n])\cong H^{i-n}(\mc E_x)$. So we must show $H^j(\mc E_x)$ is $0$ for $j\leq \vec p_1(S)-k+1$ and that it is $\vec p_2(S)$-torsion free when $j=\vec p_1(S)-k+2$. Recall that $H^j(\mc E_x)=0$ automatically for $j\neq 0,1$. 
Now, as $\vec p$ and $\vec \p$ are $\mc E$-compatible, we have that $\vec p_1(S)\leq \codim(S)-1=k-1$.
If $\vec p_1(S)\leq k-3$ then $\vec p_1(S)-k+2\leq -1$ and so $H^j(\mc E_x)=0$ for $j\leq \vec p_1(S)-k+2$. If $\vec p_1(S)=k-2$, then $\vec p_1(S)-k+2=0$ so similarly  $H^j(\mc E_x)=0$ for $j<\vec p_1(S)-k+2$ while $H^{\vec p_1(S)-k+2}(\mc E_x)=H^0(\mc E_x)$ is $\vec p_2(S)$-torsion free in this case by the compatibility assumptions. Finally, if $\vec p_1(S)=k-1$ so that  $\vec p_1(S)-k+2=1$, we have that $H^0(\mc E_x)=0$ and $H^1(\mc E_x)$ is $\vec p_2(S)$-torsion free again by the compatibility conditions\footnote{We see in this argument why $\vec p_1(S)\leq \codim(S)-1$ is required, as well as our torsion-free assumptions; see also Section \ref{S: sing in reg}, below.}.

Next we suppose that $x\in S\subset X_{n-k}$ and that $S\subset \bS$ for $\bS\subset \X_{n-\ell}$ a singular stratum of $\X$. 
Let $U\cong \R^{n-\ell}\times cL$ be a distinguished neighborhood of $x$ in the $\X$ stratification. By abuse of notation, we can identify  $U$ with $\R^{n-\ell}\times cL$, letting $\PP^*$ also denote its pullback to this product neighborhood, which remains $\X$-clc. 
Let $x=(y,v)$ with $f_y:\{y\}\into \R^{n-\ell}$ and $f_v:\{v\}\into cL$ the vertex inclusion.
Let $\pi_1:\R^{n-\ell}\times cL\to \R^{n-\ell}$ and $\pi_2:\R^{n-\ell}\times cL\to cL$ be the projections, and let $s:cL\into \{y\}\times cL$ be the inclusion. By \cite[Proposition 2.7.8]{KS} (letting the $Y_n$ there be close balls in $\R^{n-\ell}$), we have $\PP^*\cong \pi_2^*R\pi_{2*}\PP^*$. Let $\mc R_A$ denote the constant sheaf on the space $A$ with stalks in our ground ring $R$. Then  $$\PP^*\cong\mc R_{\R^{n-\ell}\times cL} \overset{L}{\otimes} \PP^*\cong 
\pi_1^*\mc R_{\R^{n-\ell}}\overset{L}{\otimes} \pi_2^*R\pi_{2*}\PP^*.$$  By \cite[Remark V.10.20.c]{Bo}, whose hypotheses are satisfied due to the constructibility  (see \cite[Proposition 4.0.2.2]{Sch03}), $$f^!_{x}\PP^*\cong f^!_y \mc R_{\R^{n-\ell}}\overset{L}{\otimes} f_v^!R\pi_{2*}\PP^*.$$

For the first factor, by \cite[Proposition V.3.7.b]{Bo} we have $f^!_y \mc R_{\R^{n-\ell}}\cong f^*_y\mc R_{\R^{n-\ell}}[-(n-\ell)]=R[-(n-\ell)]$, and so 
$$f^!_{x}\PP^*\cong f_v^!R\pi_{2*}\PP^*[-(n-\ell)].$$

To compute  $H^i\left(f_v^!R\pi_{2*}\PP^*\right)$, consider the long exact 
 sequence \cite[Section V.1.8]{Bo}
$$\to H^i\left(f_v^!R\pi_{2*}\PP^*\right)\to \H^i\left(cL;R\pi_{2*}\PP^*\right)\xr{\alpha} \H^i\left(cL-\{v\}; \bar{\mf i}^*R\pi_{2*}\PP^*\right)\to,$$
where $\bar{\mf i}:cL-\{v\}\to cL$ fits into the 
Cartesian square

\begin{diagram}
\R^{n-\ell}\times cL-\R^{n-\ell}\times \{v\}&\rTo^{\bar \pi_2}&cL-\{v\}\\
\dInto_{\mf i}&&\dInto_{\bar{\mf i}} \\
\R^{n-\ell}\times cL&\rTo^{\pi_2}&cL.
\end{diagram}
We can see that $\bar{\mf i}^*R\pi_{2*}\PP^*\cong R\bar \pi_{2*}\mf i^*\PP^*$ by replacing $\PP^*$ with an injective resolution and then considering sections over open sets. It follows that  $\alpha$ is isomorphic to the attaching map $\H^i(\R^{n-\ell}\times cL;\PP^*)\to  \H^i(\R^{n-\ell}\times cL; R\mf i_*\mf i^*\PP^*)= \H^i(\R^{n-\ell}\times (cL-\{v\}); \mf i^*\PP^*)$. As $\PP^*$ and $R\mf i_*\mf i^*\PP^*$ are $\X$-clc, \cite[Proposition 4.0.2]{Sch03} implies that restriction to smaller distinguished neighborhoods of $x$ in $\X$ yields a constant map of constant direct systems, 
and so this hypercohomology attaching map is isomorphic to the map it induces stalkwise in the direct limit. 
And by the axioms TAx1$(\X,\vec \p,\mc E)$, which $\PP^*$ satisfies, this attaching map induces stalk-wise cohomology isomorphisms for $i\leq \vec \p_1(\bS)$ and an isomorphism onto the $\vec \p_2(\bS)$-torsion module for $i=\vec \p_1(\bS)+1$. So  $H^i\left(f_v^!R\pi_{2*}\PP^*\right)=0$ for $i\leq \vec \p_1(\bS)+1$.  
We also know that  $\H^{\vec \p_1(\bS)+2}\left(cL;R\pi_{2*}\PP^*\right)\cong \H^{\vec \p_1(\bS)+2}\left(\R^{n-\ell}\times cL;\PP^*\right)\cong \H^{\vec \p_1(\bS)+2}\left(\PP^*_x\right)$, the latter by \cite[Proposition 4.0.2.2]{Sch03} again, and this is $0$ by axiom TAx1'c for $\PP^*$. Since we have already noted that in degree  $\vec \p_1(\bS)+1$ the map $\alpha$ is an isomorphism onto the $\vec \p_2(\bS)$-torsion module of $H^{\vec \p_1(\bS)+1}\left(cL-\{v\}; \bar{\mf i}^*R\pi_{2*}\PP^*\right)$, it follows that $H^{\vec \p_1(\bS)+2}\left(f_v^!R\pi_{2*}\PP^*\right)$ is $\vec \p_2(\bS)$-torsion free.

Returning now to $H^i(f_x^!\PP^*)\cong H^i(f_v^!R\pi_{2*}\PP^*[-(n-\ell)])= H^{i-n+\ell}(f_v^!R\pi_{2*}\PP^*)$, we conclude that 
$H^i(f_x^!\PP^*)=0$ when $i-n+\ell\leq \vec \p_1(\bS)+1$, i.e.\ when $i\leq \vec \p_1(\bS)+n-\ell+1$, and that $H^{\vec \p_1(\bS)+n-\ell+2}(f_x^!\PP^*)$ is  $\vec \p_2(\bS)$-torsion free. By assumption,   $\vec p_1(S)\leq\vec \p_1(\bS)+\codim(S)-\codim(\bS)=\vec \p_1(\bS)+k-\ell$, so 
$\vec \p_1(\bS)+n-\ell+1\geq   \vec p_1(S)-k+\ell+n-\ell+1=\vec p_1(S)-k+n+1$. So $H^i(f_x^!\PP^*)=0$ for $i\leq \vec p_1(S)-k+n+1$ as desired. Furthermore, $H^{\vec p_1(S)-k+n+2}(f_x^!\PP^*)$ will also be $0$ unless the above inequalities are equalities, in which case 
$H^{\vec p_1(S)-k+n+2}(f_x^!\PP^*)=H^{\vec \p_1(\bS)-\ell+n+2}(f_x^!\PP^*)$ is $\vec \p_2(\bS)$-torsion free. But we have assumed for this scenario that $\vec \p_2(\bS)\supset \vec p_2(S)$ so that this module is also $\vec p_2(S)$-torsion free. 

We have now demonstrated all the axioms, completing the proof.
\end{proof}

\subsection{Pullback and pushforward perversities}

In \cite{CST-inv}, Chataur, Saralegi, and Tanr\'e consider the invariance of intersection homology under refinement/coarsening when the perversity on the finer stratification is pulled back from a perversity on the coarser stratification or when the perversity on the coarser stratification is pushed forward from the finer stratification. In the following two subsections we consider such constructions for ts-perversities. We will see that pullback perversities can always be constructed and always result in quasi-isomorphic ts-Deligne sheaves,  generalizing \cite[Corollary 6.13]{CST-inv}. By contrast, pushforward perversities require certain conditions to be defined and then further conditions to provide quasi-isomorphic ts-Deligne sheaves. Our results about pushforwards generalize \cite[Theorem C]{CST-inv}. 

In this section we assume the following situation: We suppose that $X$ and $\X$ denote two CS set stratifications of the same underlying space with $\X$ coarsening $X$. Let $\nu:X\to\X$ denote the identity map, which is a stratified map; we sometimes refer to $\nu$ as a coarsening map. If $S$ is a stratum of $X$, let $S^{\nu}$ denote the stratum of $\X$ containing it.

\subsubsection{Pullback perversities}\label{S: pullback}
 
We first define pullback perversities. 

\begin{definition}
Let $\nu:X\to \X$ be a coarsening map. Suppose $\X$ is adapted to a ts-coefficient system $\mc E$, and let $\vec \p$ be a ts-perversity on $\X$. We define the \emph{$\mc E$-compatible pullback perversity}  $\nu_{\mc E}^*\vec \p$ on the refinement $X$ of $\X$ by:

\begin{enumerate}
\item if $S^\nu$ is singular then $\nu^*_{\mc E}\vec \p(S)=\vec \p(S^\nu)$,

\item if $S^\nu$ is regular then $\nu^*_{\mc E}\vec \p(S)=(0,\wp)$, where $\wp$ is the smallest subset of $P(R)$ such that $\mc E$ is a $\wp$-coefficient system on $S$.
\end{enumerate}
\end{definition}

\begin{remark}
Since $X$ refines $\X$ and $\mc E$ is adapted to $\X$, for each singular stratum $S$ of $X$ contained in a regular stratum of $\X$ the  $\wp$ in the second condition always exists. In fact, we could use in this condition any choice of $\wp\subset P(R)$ such that $\mc E$ is a $\wp$-coefficient system on $S$ for the purposes of the following theorem; we choose the smallest such $\wp$ just for definiteness.
\end{remark}

\begin{theorem}\label{T: pullback}
If $X$ has no stratum of codimension one contained in a regular stratum of $\X$ then
$\vec \p$ and  $\nu_{\mc E}^*\vec \p$ are $\mc E$-compatible. Consequently  $\mc P^*_{\nu^*_{\mc E} \vec \p}\cong \PP^*_{\vec \p}$. 
\end{theorem}
\begin{proof}
The first statement is immediate from the definitions.
The second follows by Theorem \ref{T: top}.
\end{proof}

\begin{remark}
If we assume in  Theorem \ref{T: pullback} that $E$ is a local system concentrated in degree $0$ (and so, in particular, a $\emptyset$-coefficient system) and if $\vec p_2(S)=\emptyset$ for all $S$, then the theorem becomes a statement about ordinary intersection homology that generalizes 
 \cite[Corollary 6.12]{CST-inv}. 
\end{remark}

\subsubsection{Pushforward perversities}\label{S: pushforward}

Chataur, Saralegi, and Tanr\'e also introduce \emph{pushforward perversities} in \cite[Section 6]{CST-inv}. More specifically, they establish conditions under which a perversity can be pushed forward from a stratification of a CS set to its intrinsic stratification and for which the corresponding intersection homology groups are isomorphic.
We first generalize  pushforwards and place them in our context:

\begin{definition}
We will say that the ts-perversity $\vec p$ on $X$ \emph{can be pushed} to $\X$ if $\vec p(Y)=\vec p(Z)$ for the two singular strata $Y$ and $Z$ of $X$ whenever \emph{all} the following conditions hold:
\begin{enumerate}
\item $Y^\nu=Z^\nu$,
\item $Y^\nu$ is singular in $\X$, and

\item $\dim(Y)=\dim(Z)=\dim(Y^{\nu})$.
\end{enumerate}

If this property holds then we define the \emph{pushforward} $\nu_*\vec p$ by $\nu_*\vec p(\bS)=\vec p(S)$ if $S$ is a stratum of $X$ such that $\bS=S^\nu$ and $\dim(S)=\dim(\bS)$. Every singular stratum of $\X$ must contain such a stratum $S$, so $\nu_*\vec p$ is well defined and without any ambiguity due to our assumptions. 
\end{definition}

\begin{remark}\label{R: pushforward abuse}
If $\vec p$ is a ts-perversity that depends only on codimension, i.e.\ $\vec p(S)=\vec p(T)$ whenever $\codim(S)=\codim(T)$, then $\vec p$ can be pushed forward to any coarsening. In this case we may abuse notation and also write $\nu_*\vec p$ simply as $\vec p$. As a further abuse, we can also treat $\Z_{\geq 1}$ as the domain of $\vec p$ writing $\vec p(\codim(S))=\vec p(S)$.
\end{remark}

Unfortunately, in contrast to Theorem \ref{T: pullback}, a perversity and its pushforward are not necessarily $\mc E$-compatible, even when the pushforward is defined. For example, we need only let $\X$ be a trivially filtered manifold and let $X$ be a refinement with a stratum $S$ on which $\vec p_1(S)<-1$.

Of course Theorem \ref{T: top} and the necessity condition in Section \ref{S: necessity} show that $\mc E$-compatibility is our most general criterion for topological invariance, but for comparison with earlier results, especially \cite[Theorem C]{CST-inv}, it is useful to delineate in terms of $\vec p$ exactly when $\vec p$ and $\nu_*\vec p$ will be $\mc E$-compatible. Inspection of the definitions yields the following (cf.\ the definition of $K^*$-perversities in \cite[Definition 6.8]{CST-inv}):

\begin{proposition}
Let $\nu:X\to \X$ be a coarsening map, and suppose $\X$ is adapted to the ts-coefficient system $\mc E$. Suppose $\vec p$ is a perversity on $X$ that can be pushed to $\X$. For any singular stratum $S\subset X$, let $\td S$ denote\footnote{These $\td S$ are called \emph{source strata} in \cite{CST-inv}. Note that the precise choice of $\td S$ does not matter in the conditions that follow as the assumption that $\vec p $ can be pushed forward assures us that any two choices would give the same perversity conditions.} any stratum of $X$ such that $(\td S)^\nu=S^\nu$ and $\dim(\td S)=\dim(S^{\nu})$.
Then $\vec p$ and $\nu_*\vec p$ are $\mc E$-compatible if the following conditions hold on $\vec p$:

\begin{enumerate}
\item\label{I: push1} If $\td S$ is singular then  $\vec p_1(\td S)\leq  \vec p_1(S)\leq\vec p_1(\td S)+\codim(S)-\codim(\td S)$, and furthermore

\begin{enumerate}

\item if $\vec p_1(\td S)= \vec p_1(S)$ then $\vec p_2(\td S)\subset \vec p_2(S)$, 

\item if $\vec p_1(S)=\vec p_1(\td S)+\codim(S)-\codim(\td S)$, then $\vec p_2(\td S)\supset \vec p_2(S)$.
\end{enumerate}

\item\label{I: push2}  If $\td S$ is regular then $-1\leq \vec p_1(S)\leq \codim(S)-1$, and furthermore

\begin{enumerate}

\item If $\vec p_1(S)=-1$ then $H^1(\mc E_x)=0$ and $H^0(\mc E_x)$ is $\vec p_2(S)$-torsion for all $x\in S$,

\item if $\vec p_1(S)=0$ then $H^1(\mc E_x)$ is $p_2(S)$-torsion for all $x\in S$

\item if $\vec p_1(S)= \codim(S)-2$, then $H^0(\mc E_x)$ is $p_2(S)$-torsion free for all $x\in S$,

\item if $\vec p_1(S)=\codim(S)-1$  then $H^0(\mc E_x)=0$ and $H^1(\mc E_x)$ is $\vec p_2(S)$-torsion free for all $x\in S$.
\end{enumerate}
\end{enumerate}
\end{proposition}

\begin{remark}
Condition \ref{I: push1} on $\vec p_1$ is essentially Condition B of \cite[Definition 6.8]{CST-inv}, while Condition \ref{I: push2} on $\vec p_1$ corresponds there to Conditions A and D. Condition C of \cite{CST-inv} is built into our assumption that $\vec p$ can be pushed forward. Note, however, that in \cite{CST-inv} only pushforwards to the intrinsic stratification are considered. 
\end{remark}

\begin{remark}
If $0\leq \vec p_1(S)\leq \codim(S)-2$ for all $S$ and $\mc E$ is a globally defined local system of free modules in degree $0$, then all of condition \ref{I: sing in reg} holds automatically.
\end{remark}

\subsection{Constrained perversities}\label{S: constrained}

So far we have considered only the situation in which we have two stratifications, $X$ and $\X$, one refining the other. In this case we have a point of comparison between any two ts-perversities $\vec p$ and $\vec \p$ on these stratifications since any stratum $S$ of $X$ is contained in a stratum $\bS$ of $\X$, allowing us to compare $\vec p(S)$ with $\vec \p(\bS)$. 
If we are given instead two arbitrary stratifications, then our best hope for relating them seems to be if we can find a common refinement or coarsening, in which case we can perhaps apply Theorem \ref{T: top} twice, connecting each given stratification with our common intermediary. More generally, we can attempt to compare \emph{all} stratifications of a given space by finding a universal common refinement or coarsening. We will see that such common refinements do not always exist, even for just two stratifications (Remark \ref{R: unrefined}), but intrinsic common coarsenings do exist, as we'll show in Section \ref{S: intrinsic}.

Of course we also need to be able to assign sufficiently compatible perversities to all these stratifications. The simplest way to proceed when faced with all possible stratifications seems to be to revert more closely to the original definition of Goresky and MacPherson \cite{GM1} in which perversities were assumed to be functions only of codimension. This allows one to define perversities without any reference to a specific stratification: the perversity simply assigns the same predetermined value to all strata of the same codimension. In such a setting, Goresky and MacPherson proved a topological invariance statement of the form, ``if a perversity satisfies certain conditions then all stratifications (without codimension one strata) of the same space with that perversity and a fixed coefficient system yield quasi-isomorphic Deligne sheaves''  \cite[Theorem 4.1]{GM2}. This theorem was slightly refined by Borel \cite[Section V.4]{Bo}. In both cases, the proofs utilize common coarsenings, with the assumptions about the perversities being strong enough to imply (in our language) $\mc E$-compatibility among the manifestations of the ``same'' perversity on all stratifications.

Analogously, our goal in this section is to construct ts-perversities that depend only on codimension and such that all stratifications yield quasi-isomorphic ts-Deligne sheaves from these ts-perversities. Of course we must also take into account adaptability to coefficient systems and place certain limitations on the behavior of codimension one strata. The result will be a sequence of theorems each with a pair of results. One of each pair places stricter conditions on the perversities but allows for any stratifications up to some limitations on compatibility of codimension one strata; the other removes some of the restrictions on the perversities but forces us to fix the regular strata of the stratifications. 

We start with the following definitions:

\begin{definition}\label{D: constrained}
We call a ts-perversity $\vec p$ \emph{constrained} if it satisfies all of the following:

\begin{enumerate}
\item $\vec p$ depends only on the codimension of strata (so we can write $\vec p=(\vec p_1,\vec p_2):\Z_{\geq 1}\to \Z\times 2^{\primeset{P}(R)}$),

\item\label{I: growth} $\vec p_1$ satisfies the Goresky-MacPherson growth condition  $\vec p_1(k)\leq \vec p_1(k+1)\leq \vec p_1(k)+1$ for  $k\geq 1$,

\item\label{I: p2 is 0} $\vec p_1(2)\in\{-1,0,1\}$,

\item\label{I: flat} if  $\vec p_1(k+1)=\vec p_1(k)$ then $\vec p_2(k+1)\supset \vec p_2(k)$,

\item\label{I: top} if $\vec p_1(k+1)=\vec p_1(k)+1$ then $\vec p_2(k+1)\subset \vec p_2(k)$. 

\end{enumerate}

If $\vec p$ satisfies all these conditions except \eqref{I: p2 is 0} then we say that $\vec p$ is \emph{weakly constrained}. If $\vec p_1(2)=0$, we say that $\vec p$ is \emph{strongly constrained}.

\end{definition}

When $\vec p$ is constrained, it will also often be useful for us to consider ts-coefficient systems $\mc E$ that are ts-perverse with respect to $\vec p(2)$, in other words $\mc E\in {}^{\vec p(2)}D^\heartsuit(\text{Dom}(\mc E))$ using the t-structure of \cite[Remark 5.4]{GBF32} with the ts-perversity there taking the constant value $\vec p(2)$ on the domain of $\mc E$. Explicitly this means the following:

\begin{enumerate}

\item if $\vec p_1(2)=-1$ then $H^1(\mc E_x)=0$ and $H^0(\mc E_x)$ is $\vec p_2(2)$-torsion for all $x\in \text{Dom}(\mc E)$,

\item if $\vec p_1(2)=0$ then $H^1(\mc E_x)$ is $\vec p_2(2)$-torsion and  $H^0(\mc E_x)$ is $\vec p_2(2)$-torsion free for all $x\in \text{Dom}(\mc E)$,

\item if $\vec p_1(2)=1$  then $H^0(\mc E_x)=0$ and $H^1(\mc E_x)$ is $\vec p_2(2)$-torsion free for all $x\in \text{Dom}(\mc E)$.
\end{enumerate}

\begin{remark}
If $\vec p_1(2)=0$ and $\mc E$ is a globally defined local system of free modules in degree $0$, then these conditions hold automatically.
\end{remark}

\begin{remark}
Note that there are a few conditions specific to codimension $2$, even though  $\vec p$ is also defined for codimension one. This is because these conditions will govern what happens when a singular stratum of $X$ is contained in a regular stratum of $\X$, but we know that for topological invariance we must preclude this for codimension one strata of $X$ anyway and so these conditions only come into play starting at codimension $2$. 
\end{remark}

\begin{remark}
We will not utilize strongly constrained ts-perversities until Section \ref{S: dimension}, where they will be convenient.
\end{remark}

We begin by considering when a constrained $\vec p$ is $\mc E$-compatible with itself across two stratifications. 

\begin{proposition}\label{P: constrained}
Suppose $\X$ is a coarsening of $X$ and that $\X$ is adapted to a ts-coefficient system $\mc E$. 
 Suppose either
\begin{enumerate}
\item no codimension one stratum of $X$ is contained in a regular stratum of $\X$,  $\vec p$ is constrained, and $\mc E\in {}^{\vec p(2)}D^\heartsuit(\text{Dom}(\mc E))$, or

\item $\X^{n-1}=X^{n-1}$ and $\vec p$ is weakly constrained.
\end{enumerate}

 Then $\vec p$ on $X$ and $\vec p$ on $\X$ are $\mc E$-compatible and so the ts-perversity $\vec p$ ts-Deligne sheaves $\mc P^*$ on $X$  and $\PP^*$ on $\X$ are quasi-isomorphic.
\end{proposition}

\begin{remark}\label{R: dabble}
The dichotomy here is a reflection of that of Remark \ref{R: dichot}. If we want to allow singular strata of $X$ in regular strata of $\X$ then we need an assumption that will force the perversities on such strata into the absolute bounds $-1\leq \vec p_1(S)\leq \codim (S)-1$ of Definition \ref{D: comp perv}. The requirement $\vec p_1(2)\in\{-1,0,1\}$, together with the Goresky-MacPherson growth condition (which is in general necessary anyway for the singular-stratum-in-singular-stratum cases --- see Section \ref{S: necessity}) accomplishes this and is the weakest possible such requirement once we have eliminated codimension one strata of $X$ in regular strata of $\X$ (see Remark \ref{R: dichot}). If we wish to dispense with this additional constraint on $\vec p_1(2)$ then  we can instead ask that no singular stratum of $X$ be contained in a regular one of $\X$, leading to the second alternative of the proposition.

Of course we could also dabble in many more specific cases --- for example we might allow $\vec p(2)=-2$ and $\vec p(3)=-1$ but then simply forbid codimension $2$ strata of $X$ from appearing in regular strata of $\X$ (plus conditions involving $\mc E$).
Such scenarios are in any case still captured by Theorem \ref{T: top}, so we leave the reader to formulate his or her own variants and instead consider just these two broad situations that are closest in keeping to the previous theorems of \cite{GM2,GBF11,CST-inv}.
\end{remark}

\begin{proof}[Proof of Proposition \ref{P: constrained}]
In each case we check the conditions of Definition \ref{D: comp perv}. As noted in Remark \ref{R: pushforward abuse}, in this case we simply write $\vec p$ for both perversities, as well as writing $\vec p(\codim(S))=\vec p(S)$ when convenient.

Assuming the first hypotheses, suppose $S$ is a singular stratum of $X$ in the singular stratum $\bS$ of $\X$. Then $\codim(\bS)\leq\codim(S)$ so the first condition of Definition \ref{D: comp perv} follows by iterating Condition \ref{I: growth} of Definition \ref{D: constrained}. Furthermore, since $\vec p_1$ is non-decreasing, $\vec p_1(S)=\vec p_1(\bS)$ only if $\vec p_1(k)$ is constant from $\codim(\bS)$ to $\codim(S)$, in which case Condition \ref{I: flat} of Definition \ref{D: constrained} inductively implies that $\vec p_2(S)\supset \vec p_2(\bS)$ as needed. Similarly, if $\vec p_1(S)=\vec p_1(\bS)+\codim(S)-\codim(\bS)$ then $\vec p_1(k+1)=\vec p_1(k)+1$ for $\codim (\bS)\leq k <\codim(S)$ and so Condition \ref{I: top} of Definition \ref{D: constrained} implies  $\vec p_2(S)\subset \vec p_2(\bS)$. 

Now suppose $\bS$ is regular. By assumption $\codim(S)\geq 2$ and so the Goresky-MacPherson growth condition on $\vec p_1$, together with $\vec p_1(2)\in \{-1,0,1\}$, ensures that $-1\leq \vec p_1(S)\leq \codim(S)-1$.

\begin{itemize}
\item If $\vec p_1(S)=-1$, the growth condition implies that $\vec p_1(k)=-1$ for $2\leq k\leq \codim(S)$. Since $\mc E\in {}^{\vec p(2)}D^\heartsuit(\text{Dom}(\mc E))$, we have that $H^1(\mc E_x)=0$ and $H^0(\mc E_x)$ is $\vec p_2(2)$-torsion for all $x$ in the domain of $\mc E$, so in particular for $x\in S$. Now Condition \ref{I: flat} of Definition \ref{D: constrained} implies that $\vec p_2(S)\supset \vec p_2(2)$, so $H^0(\mc E_x)$ is also $\vec p_2(S)$-torsion for all $x\in S$.

\item If $\vec p_1(S)=\codim(S)-1$, the growth condition implies that $\vec p_1(k)=k-1$ for $2\leq k\leq \codim(S)$ and in particular that $\vec p_1(2)=1$. Since $\mc E\in {}^{\vec p(2)}D^\heartsuit(\text{Dom}(\mc E))$, we have $H^0(\mc E_x)=0$ and $H^1(\mc E_x)$ is $\vec p_2(2)$-torsion free for all $x$ in the domain of $\mc E$, so in particular for $x\in S$. Condition \ref{I: top} of Definition \ref{D: constrained} implies that $\vec p_2(S)\subset \vec p_2(2)$, so $H^1(\mc E_x)$ is also $\vec p_2(S)$-torsion free for all $x\in S$. 

\item If $\vec p_1(S)=0$ the growth condition implies that $\vec p_1(2)\in \{-1,0\}$. If $\vec p_1(2)=-1$ then since $\mc E\in {}^{\vec p(2)}D^\heartsuit(\text{Dom}(\mc E))$, we have $H^1(\mc E_x)=0$ for all $x$ in the domain of $\mc E$, in which case certainly $H^1(\mc E_x)$ is $\vec p_2(S)$-torsion for all $x\in S$. If $\vec p_1(2)=0$ the growth condition implies that $\vec p_1(k)=0$ for $2\leq k\leq \codim(S)$. Also since $\mc E\in {}^{\vec p(2)}D^\heartsuit(\text{Dom}(\mc E))$,  $H^1(\mc E_x)$ is $\vec p_2(2)$-torsion. Condition \ref{I: flat} of Definition \ref{D: constrained} implies that $\vec p_2(S)\supset \vec p_2(2)$, so $H^1(\mc E_x)$ is also $\vec p_2(S)$-torsion for all $x\in S$.

\item If $\vec p_1(S)=\codim(S)-2$ the growth condition implies that $\vec p_1(2)\in \{0,1\}$. If $\vec p_1(2)=1$ then since $\mc E\in {}^{\vec p(2)}D^\heartsuit(\text{Dom}(\mc E))$, we have $H^0(\mc E_x)=0$ for all $x$ in the domain of $\mc E$, in which case certainly $H^0(\mc E_x)$ is $\vec p_2(S)$-torsion free for all $x\in S$. If $\vec p_1(2)=0$ the growth condition implies that $\vec p_1(k)=k-2$ for $2\leq k\leq \codim(S)$. Since $\mc E\in {}^{\vec p(2)}D^\heartsuit(\text{Dom}(\mc E))$,  $H^0(\mc E_x)$ is $\vec p_2(2)$-torsion free. Condition \ref{I: top} of Definition \ref{D: constrained} implies that $\vec p_2(S)\subset \vec p_2(2)$, so $H^0(\mc E_x)$ is also $\vec p_2(S)$-torsion free for all $x\in S$.

\end{itemize}

In the second situation of the lemma, no singular stratum of $X$ is contained in a regular stratum of $\X$. The argument is therefore exactly as above except that we don't need to verify any of Condition \ref{I: sing in reg} of Definition \ref{D: comp perv}. Therefore, we don't need Condition \ref{I: p2 is 0} of Definition \ref{D: constrained} nor the assumption that $\mc E\in {}^{\vec p(2)}D^\heartsuit(\text{Dom}(\mc E))$.

The final statement of the proposition follows from Theorem \ref{T: top}.
\end{proof}

\begin{corollary}\label{C: common coarse}
Suppose that $\mc X$ and $\mc Y$ are \emph{any} two CS set stratifications of the same space, both of which are adapted to the ts-coefficient system $\mc E$. Suppose either 

\begin{enumerate}
\item $\vec p$ is a constrained ts-perversity, $\mc E\in {}^{\vec p(2)}D^\heartsuit(\text{Dom}(\mc E))$, and $\mc X$ and $\mc Y$ possess a common coarsening $\mc Z$ that is adapted to $\mc E$ such that no codimension one stratum of $\mc X$ or $\mc Y$ is contained in a regular stratum of $\mc Z$, or 

\item $\vec p$ is a weakly constrained ts-perversity and $\mc X$ and $\mc Y$ possess a common coarsening $\mc Z$ that is adapted to $\mc E$ and such that $\mc X^{n-1}=\mc Z^{n-1}=\mc Y^{n-1}$. 

\end{enumerate}
Then $\mc X$ and $\mc Y$ have quasi-isomorphic ts-perversity $\vec p$ ts-Deligne sheaves, i.e.\ $\mc P^*_{\mc X,\vec p,\mc E}\cong\mc P^*_{\mc Y,\vec p,\mc E}$. 
\end{corollary}
\begin{proof}
For each case apply Proposition \ref{P: constrained} twice.
\end{proof}

The preceding corollary shows that we can compare Deligne sheaves from different stratifications provided appropriate common coarsenings exist. With minor hypotheses, such common coarsenings are constructed in Section \ref{S: intrinsic} with the following result. Recall the definitions of maximal ts-coefficient systems from Definition \ref{D: maximal} and of ``fully adapted'' from Definition \ref{D: fully}.

\begin{theorem}\label{T: constrained}
Suppose that  $\mc X$ and $\mc Y$ are any two CS set stratifications of the same space, both fully adapted to the maximal ts-coefficient system $\mc E$. Suppose either

\begin{enumerate}

\item $\vec p$ is a constrained ts-perversity, $\mc E\in {}^{\vec p(2)}D^\heartsuit(\text{Dom}(\mc E))$,  and 
 the closure of the union of the codimension one strata of $\mc X$ (which may be empty)
is equal to the closure of the union of the codimension one strata of $\mc Y$, or

\item $\vec p$ is a weakly constrained ts-perversity and $\mc X^{n-1}=\mc Y^{n-1}$.

\end{enumerate}
Then $\mc X$ and $\mc Y$ have quasi-isomorphic ts-perversity $\vec p$ ts-Deligne sheaves, i.e.\ $\mc P^*_{\mc X,\vec p,\mc E}\cong\mc P^*_{\mc Y,\vec p,\mc E}$.  
\end{theorem}

\begin{proof}
Let $C_1$ denote the closure of the union of the codimension one strata of $\mc X$. In the first case, Proposition \ref{P: intrinsic} yields  an intrinsic stratification $\mf X_{\mc E,C_1}$ that is a common coarsening of $\mc X$ and $\mc Y$  that is adapted to $\mc E$ and such that no codimension one stratum of $\mc X$ or $\mc Y$ is contained in a regular stratum of $\mf X$ (in fact the closures of the unions of the codimension one strata of $\mc X$, $\mc Y$, and $\mf X$ are all the same set). In the second case, Proposition \ref{P: intrinsic} yields  an intrinsic stratification $\mf X_{\mc E,\mc X^{n-1}}=\mf X_{\mc E,\mc X^{n-1}}$ that is a common coarsening of $\mc X$ and $\mc Y$  that is adapted to $\mc E$ and such that $\mf X_{\mc E,\mc X^{n-1}}^{n-1}=\mc X^{n-1}=\mc Y^{n-1}$. This last statement holds because Property \ref{I: int3}
of Proposition \ref{P: intrinsic} shows that if $U_{\mc E}$ is the domain of $\mc E$ then in this case $\mf X_{\mc E,\mc X^{n-1}}-\mf X^{n-1}_{\mc E,\mc X^{n-1}}=U_{\mc E}-\mc X^{n-1}$, which must be $\mc X-\mc X^{n-1}$ as $\mc X$ is adapted to $\mc E$. Thus $\mc X$ and $\mf X_{\mc E,\mc X^{n-1}}$ have the same regular strata and hence the same $n-1$ skeleta, and similarly for $\mc Y$.

The theorem now follows from the preceding corollary.
\end{proof}

\begin{remark}
If we take $\mc E$ to be a constant sheaf concentrated in degree $0$, set $\vec p_2(k)=\emptyset$ for all $k$, and consider only pseudomanifold stratifications with no codimension one strata, we recover from the first case of the theorem the original topological invariance result of Goresky-MacPherson \cite[Theorem 4.1]{GM2}, though with a fairly different proof (see also \cite{GBF44}). Similarly, with the same assumptions on $\mc E$ and $\vec p_2$ and with $\vec p_1(2)> 0$, the second case of the theorem recovers the main result of \cite{GBF11}.
\end{remark}

\begin{remark}\label{R: uncoarsened}
In general, if we are given two CS set stratifications $\mc X$ and $\mc Y$ of the same space without any conditions on their codimension one strata there is not necessarily a common coarsening $\mf X$ such that no codimension one stratum of $\mc X$ or $\mc Y$ is contained in a regular stratum of $\mf X$. For example, just let $\mc X$ be the plane $\R^2$ with its trivial stratification and let $\mc Y$ be any other stratification with a codimension one stratum. The only common coarsening is $\mc X$ itself, and a regular stratum of $\mc X$ already contains a codimension one stratum of $\mc Y$. Even if we rule out starting with such a bad situation, consider letting $\mc X$ and $\mc Y$ be stratifications of $\R^2$, one with the $x$-axis as singular stratum and one with the $y$-axis as singular stratum. Again, there is no appropriate common coarsening.
\end{remark}

We have a similar result to Corollary \ref{C: common coarse} concerning common refinements instead of common coarsenings, though this is generally less useful as common refinements do not exist in general, even with restrictions on codimension one strata; see Remark \ref{R: unrefined}. Nonetheless, we include this corollary as it may occasionally be useful when working with spaces  for which the hypotheses of the preceding corollaries do not apply.

\begin{corollary}
Suppose that $\vec p$ is any ts-perversity depending only on codimension and that $\mc X$ and $\mc Y$ are any two CS set stratifications of the same space, both of which are adapted to the ts-coefficient system $\mc E$. If $\mc X$ and $\mc Y$ posses a common refinement $\mc Z$ such that no codimension one stratum of $\mc Z$ is contained in a regular stratum of $\mc X$ or $\mc Y$, then the ts-perversity $\vec p$ ts-Deligne sheaves on $\mc X$ and $\mc Y$ are quasi-isomorphic. 
\end{corollary}
\begin{proof}
This follows from Theorem \ref{T: pullback} using $\mc Z$ as an intermediary. 
\end{proof}

\begin{remark}\label{R: unrefined}
Unfortunately, unlike the case of coarsenings for which any two stratifications have a common coarsening (given appropriate assumptions on coefficients and codimension one strata), common refinements do not exist in general. For example, let  $\mc X$ be $\R^n$, $n\geq 2 $, stratified as $\R^n\supset \{x_1\text{-axis}\}$, and let $\mc Y$ be $\R^n$ stratified by $\R^n\supset A$, where $A$ is the image of $f:\R\to\R^n$ given by $f(t)=(t,t^3\sin(1/t),0,\ldots,0)$ for $t\neq 0$ and $f(0)=\vec 0$. Then $A$ is a $C^1$ submanifold and so possesses a tubular neighborhood, making $\mc Y$ a CS set. But the intersection of $A$ and the $x_1$-axis is the union of the origin with the points $\left\{\frac{1}{n\pi},0,\ldots, 0\right\}$ and so any common refinement $\mc Z$ of the stratifications $\mc X$ and $\mc Y$ would require $\mc Z^0$ to have an accumulation point, which is not possible  for CS sets.  
\end{remark}

\section{Necessity of the conditions}\label{S: necessity}

In this section we consider the necessity of the  $\mc E$-compatibility conditions on ts-perversities (Definition \ref{D: comp perv}) in order for our main topological invariance result (Theorem \ref{T: top}) to hold. We will show that the ``singular-in-regular'' conditions are strictly necessary, using that the links of strata in such situations are limited to homology spheres. For the ``singular-in-singular'' conditions, there are many more possibilities and so we only have necessity in general, meaning that we will construct examples that show that failure of the conditions of Definition \ref{D: comp perv} can result in non-quasi-isomorphic ts-Deligne sheaves unless certain stalk cohomology modules (or some of their torsion submodules or torsion-free quotient modules) vanish. As noted in Remark \ref{R: efficient}, such vanishing is assured when the first component of a ts-perversity has values less than $-1$ or greater than $\codim(S) -1$, but for efficient ts-perversities (those for which $-1\leq \vec p_1(S)\leq \codim(S)-1$ for all singular strata $S$) such vanishing depends on the particulars of $X$, $\mc E$, and $\vec p$. 
 There are certainly stronger conditions that can be imposed on $(X,\X,\mc E)$ that would allow some weakening of Definition \ref{D: comp perv}. As an extreme example we could take $\mc E=0$ in which case $\mc P^*=\PP^*=0$ regardless of the stratifications or perversities. Another less trivial, though still somewhat artificial, example is noted in Remark \ref{R: dabble}. In fact all such extra conditions are likely to be somewhat artificial in our current general context.

For historical context, we recall that if $\mc E$ is concentrated in degree $0$, if $\vec p$ and $\vec \p$ depend only on codimension, and if $\vec p_2(k)=\vec \p_2(k)=\emptyset$ for all $k\geq 2$ (with no codimension one strata allowed) then $\mc P^*$ and $\PP^*$ will be the original Deligne sheaves of Goresky and MacPherson \cite{GM2}. In this case the hypercohomology groups with compact supports will be the classical intersection homology groups. If $\bar p(k)\leq k-2$ for all $k$, as in \cite{GM2}, then these will agree with the singular intersection homology groups of King \cite{Ki}, called ``GM intersection homology'' in \cite{GBF35}. For GM intersection homology, the Goresky-MacPherson growth condition $\bar p(k)\leq \bar p(k+1)\leq \bar p(k)+1$ is known to be necessary in general for topological invariance; see King \cite[Section 2]{Ki}.
The main idea is to compare $c(\Sigma X)$ with $\R\times cX$, where $cX$ is the open cone on $X$ and $\Sigma X$ is the suspension. For compact stratified $X$ these spaces are topologically homeomorphic but with different natural stratifications coming from that on $X$. Goresky and MacPherson also assume $\bar p(2)=0$ (and no codimension one strata), though King shows that this assumption is not necessary for the topological invariance of singular chain GM intersection homology, only that $\bar p(1)\geq 0$.

If $\bar p(k)> k-2$ for any $k$ the hypercohomology of the Deligne sheaf is called 
``non-GM intersection homology'' in \cite{GBF35}. Non-GM intersection homology can only be obtained using singular chains after making some additional modifications; see \cite[Chapter 6]{GBF35}. 
We showed in \cite[Section 3]{GBF11} that if $\bar p(2)>0$ then the non-GM intersection homology cannot be a topological invariant in general, but that if the Goresky-MacPherson growth condition holds then we will have invariance under refinements such that $X^{n-1}=\X^{n-1}$. The general necessity of the growth condition for non-GM intersection homology can be argued identically as in the GM case, and we will use 
essentially the same basic argument  below for the case of singular strata of $X$ in singular strata of $\X$.

In the remainder of this section, we first develop some computational machinery in Section \ref{S: tools}. We then consider the ``singular-in-regular'' situation in Section \ref{S: sing in reg} and the ``singular-in-singular'' situation in Section \ref{S: sing in sing}. Throughout we continue our notation from the preceding sections, namely $\X$ coarsens $X$ with corresponding ts-perversities $\vec \p$ and $\vec p$ and ts-Deligne sheaves $\PP^*$ and $\mc P^*$. We also assume $\X$ adapted to a ts-coefficient system $\mc E$.

\subsection{Computational tools}\label{S: tools}

We need some tools for computation. Our first lemma is standard:

\begin{lemma}\label{L: prod}
Suppose $Y$ is a filtered space, and let $X=\R^k\times Y$ have the product filtration $X^j=\R^k\times X^{j-k}$. If $\mc S^*$ is $X$-clc then $\H^i(X;\mc S^*)\cong \H^i(Y;\mc S^*|_Y)$, identifying $Y$ with any $\{z\}\times Y\subset \R^k\times Y$. 
\end{lemma}
\begin{proof}
Let $\pi:\R^k\times Y\to Y$ be the projection and let $s:\{z\}\times Y\into \R^k\times Y$ be the inclusion. Since $\mc S^*$ is $X$-clc, in particular $H^j(\mc S^*)$ is constant along each $\R^k\times \{y\}$, and so  $\mc S^*=\pi^*R\pi_*\mc S^*$ by \cite[Proposition 2.7.8]{KS} (letting the $Y_n$ there be closed balls in $\R^k$). Then $s^*\mc S^*\cong s^*\pi^*R\pi_*\mc S^*=R\pi_*\mc S^*$, so $\H^i(\R^k\times Y;\mc S^*)\cong \H^i(Y;R\pi_*\mc S^*)\cong \H^i(Y; s^*\mc S^*)=\H^i(Y; \mc S^*|_Y)$.
\end{proof}

\begin{lemma}\label{L: ngbd}
Let $L$ be a compact filtered set such that $X=\R^k\times cL$ is a CS set. Let $x=(s,v)\in \R^k\times cL$ with $s\in \R^k$ arbitrary and $v\in cL$ the cone vertex, and let $S$ be the stratum $\R^k\times \{v\}$. If $\mc P^*$ is a ts-Deligne sheaf on $\R^k\times cL$, then identifying $L$ with some copy $(z,t,L)\subset \R^k\times (0,1)\times L\subset \R^k\times cL$, we have
\begin{equation*}
H^i(\mc P^*_x)\cong 
\begin{cases}
0,&i>\vec p_1(S)+1,\\
T^{\vec p_2(S)}\H^i(L;\mc P^*|_L),&i=\vec p_1(S)+1,\\
\H^i(L;\mc P^*|_L),&i\leq \vec p_1(S).
\end{cases}
\end{equation*}
\end{lemma}
\begin{proof}
 Let $W=X-S=\R^k\times (cL-\{v\})\cong \R^{k+1}\times L$, let $i:W\into X$, and let $\mc P^*_W=\mc P^*|_W$. From the definition of the ts-Deligne sheaf and the torsion tipped truncation, we have
$$H^i(\mc P^*_x)\cong 
\begin{cases}
0,&i>\vec p_1(S)+1,\\
T^{\vec p_2(S)}H^{\vec p_1(S)+1}((Ri_*\mc P^*_W)_x),&i=\vec p_1(S)+1,\\
H^i((Ri_*\mc P^*_W)_x),&i\leq \vec p_1(S).
\end{cases}$$
Since $\mc P^*_W$ is $W$-clc,  $Ri_*\mc P^*_W$ is $X$-clc by \cite[Proposition 4.0.2.3]{Sch03}, and so $H^i((Ri_*\mc P^*_W)_x)\cong \H^i(X;Ri_*\mc P^*_W)$ by \cite[Proposition 4.0.2.2]{Sch03}. But then 
$$\H^i(X;Ri_*\mc P^*_W)\cong \H^i(W;\mc P^*_W)\cong \H^i(\R^{k+1}\times L;\mc P^*_W)\cong \H^i(L;\mc P^*|_L),$$  
using Lemma \ref{L: prod}.
\end{proof}

If $L$ is itself a CS set, then $\mc P^*|_L$ will itself be a ts-Deligne sheaf, as we show in the next lemma. To establish notation, 
suppose $X$ is a CS set and $\vec p,\mc E, \mc P^*$ are a ts-perversity, ts-coefficient system, and ts-Deligne sheaf on $X$. If $Y\subset X$ is a CS set with the induced stratification (in the sense of the statement of the lemma below) we will denote by $\vec p_Y$ the ts-perversity on $Y$ such that $\vec p_Y(S)=\vec p(Z)$ if the singular stratum $S$ of $Y$ is contained in the singular stratum $Z$ of $X$. We also write $\mc E_Y$ for the restriction of $\mc E$ to the intersection of its domain with $Y$, and we will write $\mc P^*_Y$ for the Deligne sheaf on $Y$ with respect to $\vec p_Y,\mc E_Y$. 

\begin{lemma}
Suppose that either 

\begin{enumerate}
\item $Y$ is an open subset of the CS set $X$ stratified by $Y^j=Y\cap X^j$ or 

\item  $Y$ is a CS set and $X=\R^m\times Y$ with $X^j=\R^m\times X^{j-m}$ and we identify $Y$ with $\{z\}\times Y$ for some $z\in \R^m$.
\end{enumerate}
Let $\vec p,\mc E$ be a ts-perversity and ts-coefficient system on $X$. Then $\mc P^*|_Y\cong \mc P^*_Y$, i.e.\ the restriction of the ts-Deligne sheaf on $X$ to $Y$ is quasi-isomorphic to the $\vec p_Y,\mc E_Y$ ts-Deligne sheaf on $Y$. 
\end{lemma}
\begin{proof}
Note that in both cases both $X$ and $Y$ are CS sets by \cite[Lemmas 2.3.13 and 2.11.4]{GBF35}. We know that ts-Deligne sheaves are characterized up to quasi-isomorphic by the axioms TAx1'. That $\mc P^*$ satisfies these axioms on $X$ implies that $\mc P^*|_Y$ satisfies the axioms on $Y$. The only axiom that is not obvious is the last axiom when $X=\R^m\times Y$. In this case suppose $X$ is $n$-dimensional so that $Y$ is $n-m$ dimensional. Let $x\in Y_{(n-m)-k}\subset X_{n-k}$. Let $f_x:\{x\}\into X$ and $g_x:\{x\}\into Y$ for $x$ in a singular stratum. 
Now, exactly as in the proof of Theorem \ref{T: top}, if $\pi:\R^m\times Y\to Y$ is the projection we have $f^!_{x}\mc P^* \cong g_x^!R\pi_*\mc P^*[-m],$ while $R\pi_*\mc P^* \cong \mc P^*|_Y$ by the proof of Lemma \ref{L: prod}. Thus 
 $$H^i(f_x^!\mc P^*)\cong H^{i-m}(g_x^!\mc P^*|_Y).$$
The axiom for $\mc P^*|_Y$ now follows from the axiom for $\mc P^*$.
\end{proof}

In what follows we shall abuse notation and write the restrictions $\vec p_Y$, $\mc E_Y$, and $\mc P^*_Y$ as simply $\vec p$, $\mc E$, and $\mc P^*$  if it is clear what is meant from context. We let $\Sigma X$ denote the (unreduced) suspension, stratified by $(\Sigma X)^i=\Sigma(X^{i-1})$ with $(\Sigma X)^0=\{\mf n,\mf s\}$, the union of the two suspension points.

\begin{proposition}\label{P: susp}
Let $X^{n-1}$ be a compact CS set with suspension $\Sigma X$. Let $\vec p$ be a $ts$-perversity on $\Sigma X$ such that $\vec p(\mf n)=\vec p(\mf s)$, and denote the common value $(p,\wp)$. Let $\mc E$ be a ts-coefficient system to which $\Sigma X$ is adapted, and  let $\mc P^*$ be the associated ts-Deligne sheaf. Then 

\begin{equation*}
\H^i(\Sigma X;\mc P^*)\cong
\begin{cases}
\H^{i-1}(X; \mc P^*), &i\geq  p+3,\\
\H^{p+1}(X;\mc P^*)/ T^{\wp}\H^{p+1}(X;\mc P^*),&i=p+2,\\
T^{\wp}\H^{p+1}(X;\mc P^*), &i=p+1,\\
\H^{i}(X; \mc P^*), &i\leq p.
\end{cases}
\end{equation*}
\end{proposition}

\begin{proof}
Let $U_1=\Sigma X-\{\mf s\}\cong cX$ and  $U_2=\Sigma X-\{\mf n\}\cong cX$. We consider the Mayer-Vietoris sequence \cite[Remark 2.6.10]{KS}
$$\to \H^i(\Sigma X;\mc P^*)\to \H^i(U_1;\mc P^*)\oplus \H^i(U_2;\mc P^*)\to \H^i(U_1\cap U_2;\mc P^*)\to.$$
 Note that  $U_1\cap U_2\cong (0,1)\times X$ so we have $\H^i(U_1\cap U_2;\mc P^*)\cong H^i(X;\mc P^*)$ by Lemma \ref{L: prod}.

Now consider $\H^i(U_1;\mc P^*)$. Since $\mc P^*$ is $\Sigma X$-clc and since $U_1$ is a distinguished neighborhood of $\mf n$, we have $\H^i(U_1;\mc P^*)\cong H^i(\mc P^*_{\mf n})$ by \cite[Proposition 4.0.2.2]{Sch03}. Since this is $0$ for $i>p+1$, and similarly for $U_2$, we have $\H^i(\Sigma X;\mc P^*)\cong \H^{i-1}(U_1\cap U_2;\mc P^*)\cong \H^{i-1}(X; \mc P^*)$ for $i>p+2$.

For $i\leq p$, using \cite[Proposition 4.0.2.2]{Sch03} again implies that  the hypercohomology map from $U_1$ to $U_1-\{\mf n\}=U_1\cap U_2$ is isomorphic to the attaching map, which is an isomorphism by axiom TAx1\ref{A: attach}, and similarly for $U_2$. It follows that in this range $\H^i(\Sigma X;\mc P^*)\cong\H^i(U_j;\mc P^*)\cong  \H^i(U;\mc P^*)\cong \H^{i}(X; \mc P^*)$ for $j=1,2$.

In the middle range we have 
\begin{diagram}
0&\rTo&\H^{p+1}(\Sigma X;\mc P^*)&\rTo& \H^{p+1}(U_1;\mc P^*)\oplus \H^{p+1}(U_2;\mc P^*)\\
{}&\rTo& \H^{p+1}(U_1\cap U_2;\mc P^*)
&\rTo& \H^{p+2}(\Sigma X;\mc P^*)&\rTo& 0.
\end{diagram}
Once again we have that $\H^{p+1}(U_1;\mc P^*)\cong H^{p+1}(\mc P^*_{\mf n})$, which in this case gives the group $T^{\wp}\H^{p+1}(U_1\cap U_2;\mc P^*)\cong T^{\wp}\H^{p+1}(X;\mc P^*)$ by Lemma \ref{L: ngbd}.
Each of the maps $\H^{p+1}(U_j;\mc P^*)\to \H^{p+1}(U_1\cap U_2;\mc P^*)$ thus corresponds to the inclusion (up to sign) of the $\wp$-torsion subgroup. So we have $\H^{p+1}(\Sigma X;\mc P^*)\cong T^{\wp}\H^{p+1}(X;\mc P^*)$ and $\H^{p+2}(\Sigma X;\mc P^*)\cong\H^{p+1}(X;\mc P^*)/ T^{\wp}\H^{p+1}(X;\mc P^*)$.
\end{proof}

\begin{example}
We provide an example of a computation using Proposition \ref{P: susp} that also demonstrates how Poincar\'e duality can fail for CS sets, both for classical homology and cohomology and for ordinary intersection cohomology, while duality is recovered using ts-Deligne sheaves. 

Let $X = \R P^3$, which is orientable. Then, by standard computations, we have 

\begin{equation*}
H_i(\Sigma \R P^3) = 
\begin{cases}
\Z, & i=4,\\
0, & i=3, \\
\Z_2, & i=2,\\
0, & i=1,\\
\Z, & i=0,
\end{cases}
\qquad
\qquad
H^i(\Sigma \R P^3) = 
\begin{cases}
\Z, & i=4,\\
\Z_2, & i=3, \\
0, & i=2,\\
0, & i=1,\\
\Z, & i=0.
\end{cases}
\end{equation*}
This computation demonstrates the failure, in general, of classical Poincar\'e duality for singular spaces. Even with classical intersection homology, Poincar\'e duality typically only holds rationally; see \cite{GM1, GM2, Bo} and our remarks below the following computation. 

However, let now $\mc E$ be the constant coefficient system with stalk $\Z$, and let $\vec p$ take the value $(1,\emptyset)$ on the suspension points. The dual perversity $D\vec p$ will then take the value $(1, P(\Z))$. By \cite[Theorem 4.19]{GBF32}, we know $\mc D\mc P_{\vec p}^*[-n]\cong \mc P^*_{D\vec p}$, which implies by \cite[Corollary 4.21]{GBF32} the nonsingular pairings 
\begin{align*}
F\H^i(\Sigma \R P^3;\mc P_{\vec p}) \otimes F\H^{4-i}(\Sigma \R P^3;\mc P_{D\vec p}) &\to \Z\\
T\H^i(\Sigma \R P^3;\mc P_{\vec p}) \otimes T\H^{4-i+1}(\Sigma \R P^3;\mc P_{D\vec p}) &\to \Q/\Z,
\end{align*}
where $T\H^i$ denotes the torsion subgroup and $F\H^i$ denotes the torsion-free quotient.
Indeed, we can compute by Proposition \ref{P: susp}:

\begin{align*}
\H^i(\Sigma \R P^3, \mc P_{\vec p}^*) &=
\left\{
\begin{array}{rclr} 
H^3(\R P^3)&=&\Z, & i=4,\\
H^2(\R P^3)/T^\emptyset H^2(\R P^3)&=&\Z_2, & i=3, \\
T^\emptyset H^2(\R P^3)& =&0, & i=2,\\
H^1(\R P^3)&=& 0, & i=1,\\
H^0(\R P^3)&=& \Z, & i=0,
\end{array}
\right.\\
\H^i(\Sigma \R P^3, \mc P_{D\vec p}^*)) &= 
\left\{
\begin{array}{rclr} 
H^3(\R P^3)&=&\Z, & i=4,\\
\phantom{{}^\emptyset}H^2(\R P^3)/T H^2(\R P^3)&=&0, & i=3, \\
\phantom{{}^\emptyset}TH^2(\R P^3) &=& \Z_2, & i=2,\\
H^1(\R P^3)&=& 0, & i=1,\\
H^0(\R P^3)&=& \Z, & i=0.
\end{array}
\right.
\end{align*}
This is consistent with the existence of the claimed nonsingular pairings. By contrast, the first of these hypercohomology computations with ts-perversity $\vec p$ agrees with the classical Goresky-MacPherson intersection cohomology for the self dual perversity $\bar p$ with  $\bar p(4)=1$, but we see that these intersection cohomology groups cannot be self-dual. \hfill\qedsymbol
\end{example}

As a corollary, and as a nice example of an application of Theorem \ref{T: top}, we compute $\H^i(S^k*X;\mc P^*)$ for $k>0$, where $X$ is a compact CS set,  $S^k$ is the $k$-sphere with trivial stratification,  and $S^k*X$ is the join. Rather than use the join stratification of \cite[Section 2.11]{GBF35}, however, it will be more natural for us below to use the following stratification. Recall that we can decompose $S^k*X$ into $cS^k\times X$ and $S^k\times cX$ (see \cite[Section 2.11]{GBF35}). We give $S^k\times cX\subset S^k*X$ the stratification it obtains from the cone and product stratifications, while we stratify $cS^k\times X$ as $D^{k+1}\times X$, where $D^{k+1}$ is the interior of the unit disk with the trivial stratification. These two stratifications agree on the overlap $S^k\times (0,1)\times X$ and so patch together to give a CS set stratification of $S^k*X$. Letting $v$ denote the cone vertex, we identify $S^k$ with the stratum $S^k\times \{v\}\subset S^k\times cX\subset S^k*X$. 

\begin{corollary}\label{C: sphere join}
Let $X$ be a compact CS set, let  $S^k$, $k>0$, be the $k$-sphere with trivial stratification, and let $S^k*X$ be the join with the stratification as above. 
 Let $\vec p$ be a ts-perversity on $S^k*X$, and let $\mc E$ be a ts-coefficient system to which $S^k*X$ is adapted. Then

\begin{equation*}
\H^i(S^k*X;\mc P^*)\cong 
\begin{cases}
\H^{i-k-1}(X; \mc P^*), &i>\vec p_1(S^k)+k+2,\\
\H^{\vec p_1(S^k)+1}(X;\mc P^*)/ T^{\vec p_2(S^k)}\H^{\vec p_1(S^k)+1}(X;\mc P^*),&i=\vec p_1(S^k)+k+2,\\
0,&   \vec p_1(S^k)+1<i<\vec p_1(S^k)+k+2,\\
T^{\vec p_2(S^k)}\H^{\vec p_1(S^k)+1}(X;\mc P^*), &i=\vec p_1(S^k)+1,\\
\H^{i}(X; \mc P^*), &i\leq \vec p_1(S^k).
\end{cases}
\end{equation*}
\end{corollary}
\begin{proof}
Let $\Sigma^{k+1}X$ be the $k+1$ times iterated suspension of $X$. Topologically (ignoring stratifications) $\Sigma^{k+1}X\cong S^k*X$. Furthermore,  the stratification of $\Sigma^{k+1}(X)$  as an iterated suspension refines the stratification of $S^k*X$. In fact,  the stratifications are identical on $S^k*X-S^k$, which is $D^{k+1}\times X$ with the product stratification. In particular, no singular stratum of $\Sigma^{k+1}X$ is contained in a regular stratum of $S^k*X$. Let $\vec q$ be the ts-perversity on $\Sigma^kX$ such that $\vec q(Z)=\vec p(Z)$ for strata $S$ shared by $\Sigma^{k+1}X$ and $S^k*X$ and such that $\vec q(Z)=\vec p(S^k)$ if $Z$ is a stratum of $\Sigma^{k+1}X$ contained in $S^k$. Then $\vec p$ and $\vec q$ are $\mc E$-compatible for any $\mc E$, and so by Theorem \ref{T: top} the  $\vec p$ ts-Deligne sheaf on $S^k*X$ and the $\vec q$ ts-Deligne sheaf on $\Sigma^{k+1}X$ are quasi-isomorphic. The computation now follows by applying Proposition \ref{P: susp} iteratively.
\end{proof}

\subsection{Singular strata in regular strata}\label{S: sing in reg}

In this section we will show that if $S$ is a singular stratum of $X$ contained in a regular stratum $\bS$ of $\X$ then
 the conditions of Definition \ref{D: comp perv} for such strata  are strictly necessary in order to have $\mc P^*\cong \PP^*$ over $S$. On $\bS$ we have $\PP^*\cong \mc E$ so we will assume also that $\mc P^*\cong \mc E$ over $\bS$ and see that contradictions occur if any of the conditions of Definition \ref{D: comp perv} for this scenario fail.

Suppose $x\in S$. Then $x$ has a distinguished neighborhood $\R^{n-k}\times cL$, $k>0$. Topologically this is homeomorphic to $cS^{n-k-1}\times cL\cong c(S^{n-k-1}*L)$. Furthermore, since $x$ is in a regular stratum of $\X$,  we have $(c(S^{n-k-1}*L),x)\cong (\R^n,x)$ by \cite[Corollary 2.10.2]{GBF35} and its proof. Thus $H^*(S^{n-k-1}*L)\cong H^*(S^{n-1})$, and so $L$ must be a $k-1$ homology sphere. We will write $L$ as $L^{k-1}$ to remind us of this.

Returning to our ts-Deligne sheaves, we have $H^*(\PP^*_x)=H^*(\mc E_x)$.  Since $x$ has a Euclidean neighborhood, $\mc E$ is clc on our neighborhood of $x$ and so  each derived cohomology sheaf $\mc H^i(\mc E)$ is constant on this neighborhood. By assumption there is some $\wp\subset P(R)$ such that $H^1(\mc E_x )$ is $\wp$-torsion while $H^0(\mc E_x)$ is $\wp$-torsion free. All other $H^i(\mc E_x)$ are $0$. 

On the other hand, by Lemma  \ref{L: ngbd}, 
$$H^i(\mc P^*_x)\cong 
\begin{cases}
0,&i>\vec p_1(S)+1,\\
T^{\vec p_2(S)}\H^i(L^{k-1};\mc E), &i=\vec p_1(S)+1,\\
\H^i(L^{k-1};\mc E), &i\leq \vec p_1(S).
\end{cases}$$ 
Since there is some $\wp$ such that $H^1(\mc E_x)$ is $\wp$-torsion and $H^0(\mc E_x)$ is $\wp$-torsion free, any map $H^0(L^{k-1},H^1(\mc E_x))\to H^{k-1}(L^{k-1},H^0(\mc E_x))$ must be trivial, so the hypercohomology spectral sequence for $\H^*(L^{k-1};\mc E)$ degenerates (using also that $L^{k-1}$ has the cohomology of a $k-1$ sphere). It follows that if $k=1$ we have $\H^i(L^{0},\mc E)\cong H^i(\mc E_x)\oplus H^i(\mc E_x)$, and if $k\geq 3$ we have 
\begin{equation*}
\H^i(L^{k-1};\mc E)\cong
\begin{cases}
H^1(\mc E_x), &i=k,\\
H^0(\mc E_x), &i=k-1,\\
H^1(\mc E_x), &i=1,\\
H^0(\mc E_x), &i=0,\\
0,&\text{otherwise.}
\end{cases}
\end{equation*} 

Similarly if $k=2$, the only possibly nontrivial groups are for $i=0,1,2$ with $\H^0(L^1;\mc E)\cong H^0(L^1;\mc H^0(\mc E))\cong H^0(\mc E_x)$ and $\H^2(L^1;\mc E)\cong H^1(L^1;\mc H^1(\mc E))\cong H^1(\mc E_x)$. But the only easy information we have about $\H^1(L^1;\mc E)$ is that it must fit in the extension problem
\begin{equation}\label{E: ssext}
0\to  H^0(\mc E_x)\to \H^1(L^1;\mc E)\xr{q} H^1(\mc E_x)\to 0.
\end{equation}

Given these preliminaries, we can now consider the necessity of the conditions of Definition \ref{D: comp perv}, separately in the cases $k=1$, $k=2$, and $k\geq 3$, by comparing $H^*(\mc E_x)$ with $H^*(\mc P^*_x)$ as obtained from the preceding computations.

\paragraph{$k=1$.} 
In this case $H^i(\mc P_x)$ comes by torsion-tipped truncating $H^i(L^0;\mc E)\cong H^i(\mc E_x)\oplus H^i(\mc E_x)$. Thus there is no way that $H^i(\mc P_x)$ can equal $H^i(\PP^*_x)\cong H^i(\mc E_x)$, regardless of perversity unless $\mc E$ is trivial. This shows why we must always rule out codimension one strata of $X$ in regular strata of $\X$.

\paragraph{$k\geq 3$.} Given the computations above, we can readily see that if $\mc E$ is not trivial then $H^*(\mc P_x)\cong H^*(\mc E_x)$ (and hence $H^*(\mc P_x)\cong H^*(\PP_x)$) if and only one of the following scenarios holds:

\begin{enumerate}
\item $\vec p_1(S)=-1$, $H^0(\mc E_x)$ is $\vec p_2(S)$-torsion, $H^1(\mc E_x)=0$,
\item $\vec p_1(S)=0$, $H^1(\mc E_x)$ is $\vec p_2(S)$-torsion, $H^0(\mc E_x)$ arbitrary,
\item $1\leq \vec p_1(S)\leq k-3$, $H^*(\mc E_x)$ arbitrary,
\item $\vec p_1(S)=k-2$, $H^0(\mc E_x)$ is $\vec p_2(S)$-torsion free, $H^1(\mc E_x)$ is arbitrary,
\item $\vec p_1(S)=k-1$, $H^0(\mc E_x)=0$, $H^1(\mc E_x)$ is $\vec p_2(S)$-torsion free.
\end{enumerate}
Of course here ``arbitrary'' still means within the limitations of $\mc E$ being a $\wp$-coefficient system 
for some $\wp\subset P(R)$, and we have recovered  precisely the conditions from Definition \ref{D: comp perv}.

\paragraph{$k=2$.} 
This case is a bit more delicate as we can't pin down $\H^1(L^1;\mc E)$ in general. However, we 
see again that to have $H^*(\mc P_x)\cong H^*(\PP^*_x)\cong H^*(\mc E_x)$ with $\mc E$ nontrivial, it is first of all necessary to have $\vec p_1(S)\geq -1$, and also by the same arguments if  $\vec p_1(S)=-1$, then we must have $H^1(\mc E_x)=0$, and $H^0(\mc E_x)$ must be $\vec p_2(S)$ torsion.

If $\vec p_1(S)=1=k-1$, then to have $H^1(\mc P_x)\cong  H^1(\mc E_x)$ requires $\H^1(L^1;\mc E)\cong H^1(\mc E_x)$, which from the short exact sequence above requires $H^0(\mc E_x)=0$. And looking at degree $2$ we must again have that $H^1(\mc E_x)$ is $\vec p_2(S)$-torsion free.

If $\vec p_1(S)\geq 2$, then we must have $H^1(\mc E_x)=0$ for degree $2$ to work but again we also need $\H^1(L^1;\mc E)\cong H^1(\mc E_x)$ and so $H^0(\mc E_x)=0$, which forces $\mc E$ to be trivial. 

Finally, suppose $\vec p_1(S)=0=k-2$. Then to have $H^*(\mc P_x)\cong H^*(\PP^*_x)$ we must have $T^{\vec p_2(S)}\H^1(L^1;\mc E)\cong H^1(\mc E_x)$. In particular, $H^1(\mc E_x)$ must be $\vec p_2(S)$-torsion. 
Let $k:H^1(\mc E_x)\into \H^1(L^1;\mc E)$ take $H^1(\mc E_x)$ isomorphically onto $T^{\vec p_2(S)}\H^1(S;\mc E)$, and consider the composition $qk$, where $q$ is the quotient map in \eqref{E: ssext}. If $z\in H^1(\mc E_x)$, $z\neq 0$, then $z$ is $\wp$-torsion for some $\wp$ such that $\mc E$ is a $\wp$-coefficient system, so $qk(z)\neq 0$ or else there would be a $\wp$-torsion element in $\ker(q)\cong H^0(\mc E_z)$, violating that $\mc E$ is a $\wp$-coefficient system. Thus $qk$ is injective. Since $H^1(\mc E_x)$ is a finitely-generated torsion module over a PID, it is Artinian\footnote{In fact, since $H^1(\mc E_x)$ is torsion, we can treat it as a module over $R/\text{Ann}(H^1(\mc E_x))$. If $H^1(\mc E_x)\neq 0$ then Ann$(H^1(\mc E_x))\neq 0$, and this is an Artinian ring since $R$ is a PID. If $H^1(\mc E_x)=0$, it is clearly Artinian.}  and so every injective endomorphism is an isomorphism  \cite[Lemma II.4.$\alpha$]{RIB}. Precomposing $k$ with the inverse of this isomorphism provides a splitting of \eqref{E: ssext}. Consequently we also have $H^0(\mc E_x)\cong \H^1(L^1;\mc E)/T^{\vec p_2(S)}\H^1(L^1;\mc E)$, and so $H^0(\mc E_x)$ must be $\vec p_2(S)$-torsion free.

\subsection{Singular strata in singular strata}\label{S: sing in sing}

In this case we consider the necessity of the conditions of Definition \ref{D: comp perv} for a singular stratum $S$ of $X$ contained in singular stratum $\bS$ of $\X$. As noted above, these conditions won't always be necessary in the strictest sense since weaker conditions might suffice depending on the local cohomology computations resulting from specific choices of spaces, perversities, and coefficient systems. Instead, we show the conditions to be ``necessary in general,'' meaning that we will demonstrate the existence of examples where $\mc P^*$ is not quasi-isomorphic to $\PP^*$ because the conditions fail. 
Accordingly, we can choose to work in relatively simple settings.

We first discuss necessity in the codimension $0$ setting, followed by the codimension $>0$ case. For simplicity, we assume that our ts-perversities are efficient; see Remark \ref{R: efficient}. This allows us to focus on the degree ranges where the conditions of Definition \ref{D: comp perv} may be relevant rather than cases that are degenerate due to extreme degree values.

\subsubsection{Codimension $0$} Suppose that $S\subset \bS$ and $\dim(S)=\dim(\bS)$. In this case, a point $x\in S$ may have distinguished neighborhoods in $X$ and $\X$ that are filtered homeomorphic and so have the same link $L$. If $\mc P^*\cong \PP^*$ then in particular $\H^i(L;\mc P^*)\cong \H^i(L;\PP^*)$.  It is then clear from Lemma \ref{L: ngbd} that we will need to have $\vec p(S)=\vec \p(\bS)$ in order to have $H^*(\mc P^*_x)\cong H^*(\PP^*_x)$ unless there are further restrictions on $\H^i(L;\mc P^*)$.

\subsubsection{Codimension $>0$}
For simplicity, let us take  $L$ to be a trivially-stratified $n-k-2$ manifold, in which case we can assume that $H^i(L;\mc E)$ is nontrivial in any dimension we like or has $\wp$-torsion in any dimension $0$ to $n-k-2$  by manipulating $L$ and $\mc E$. 
We then consider $\X=\R^{k+1}\times cL$ for $k\geq 0$. Topologically, $\R^{k+1}\times cL\cong cS^{k}\times cL\cong c(S^{k}*L)$. If we instead stratify  this space as the cone on $S^{k}*L$, using the stratification of $S^{k}*L$ described just before Corollary \ref{C: sphere join} we obtain a refinement $X$ of $\X$. In fact, $X$ differs from $\X$ only by the addition of a single zero-dimensional stratum, namely the vertex $V$ of $c(S^{k}*L)$, which we can identify with $(0,v)\in \R^{k+1}\times cL$ if $v$ is the cone vertex of $cL$. 

By such a construction, we obtain an $X$ and $\X$ such that $X$ has a singular stratum $S=\{V\}$ contained in the singular stratum $\bS=\R^{k+1}\times \{v\}$ of $\X$ and such that $\codim(S)-\codim(\bS)=k+1$. The actual values of $\codim(S)$ and $\codim(\bS)$ will of course depend on $\dim(L)$. If we want examples with higher dimensional strata we can consider instead  $\R^m\times X$ and $\R^m\times \X$, though Lemma \ref{L: prod} shows that the cohomology computations will be the same. 
We will tend to use $V$ when referring to the point and $S$ when thinking of the stratum $S=\{V\}$.

We suppose that the ts-Deligne sheaves $\mc P^*$ and $\PP^*$ are quasi-isomorphic when restricted to the complement of $V$ and consider what would be necessary for them to be quasi-isomorphic at $V$.

Based on Lemma \ref{L: ngbd}, and our assumption that $\mc P^*\cong \PP^*$ off of $V$, we have

\begin{equation*}
H^i(\PP^*_V)\cong 
\begin{cases}
0,&i>\vec \p_1(\bS)+1,\\
T^{\vec \p_2(\bS)}\H^i(L;\mc P^*|_L) ,&i=\vec \p_1(\bS)+1,\\
\H^i(L;\mc P^*|_L),&i\leq \vec \p_1(\bS)
\end{cases}
\end{equation*}
and

\begin{equation*}
H^i(\mc P^*_V)\cong 
\begin{cases}
0,&i>\vec p_1(S)+1,\\
T^{\vec p_2(S)}\H^i(S^k*L;\mc P^*|_{S^k*L}) ,&i=\vec p_1(S)+1,\\
\H^i(S^k*L;\mc P^*|_{S^k*L}),&i\leq \vec p_1(S).
\end{cases}
\end{equation*}
We can compute $\H^i(S^k*L;\mc P^*|_{S^k*L})$ in terms of $\H^*(L;\mc P^*|_L)$ using  Corollary \ref{C: sphere join}:

\begin{equation}\label{E: susplink}
\H^i(S^k*L;\mc P^*|_{S^k*L})\cong 
\begin{cases}
\H^{i-k-1}(L; \mc P^*), &i>\vec \p_1(\bS)+k+2,\\
\H^{\vec \p_1(\bS)+1}(L;\mc P^*)/ T^{\vec p_2(\bS)}\H^{\vec p_1(\bS)+1}(L;\mc P^*),&i=\vec \p_1(\bS)+k+2,\\
0,&   \vec \p_1(\bS)+1<i<\vec \p_1(\bS)+k+2,\\
T^{\vec \p_2(\bS)}\H^{\vec \p_1(\bS)+1}(L;\mc P^*), &i=\vec \p_1(\bS)+1,\\
\H^{i}(L; \mc P^*), &i\leq \vec \p_1(\bS).
\end{cases}
\end{equation}

From these equations, we can see what constraints are necessary in this case and why: In order to have $H^i(\PP^*_V)\cong H^i(\mc P_V)$ we must truncate $\H^i(S^k*L;\mc P^*|_{S^k*L})$ in such a way that the result agrees with $H^i(\PP^*_V)$. If $\vec p_1(S)<\vec \p_1(\bS)$ then in general we will be forcing $H^i(\mc P^*_V)$ to be $0$ in some degrees $\leq \vec \p_1(\bS)+1$ in which $H^i(\PP^*_V)$ will not generally be $0$ without further vanishing assumptions. Furthermore, even if $\vec p_1(S)=\vec \p_1(\bS)$ we must have $\vec p_2(S)\supset \vec \p_2(\bS)$ in order to make sure we get all of $T^{\vec \p_2(\bS)}\H^i(L;\mc P^*|_L)$ in degree $\vec \p_1(\bS)+1$. Similarly, if $\vec p_1(S)\geq \vec \p_1(\bS)+k+2$, i.e.\ if $\vec p_1(S)>\vec \p_1(\bS)+\codim(S)-\codim(\bS)$, then the term $\H^{\vec \p_1(\bS)+1}(L;\mc P^*)/ T^{\vec p_2(\bS)}\H^{\vec p_1(\bS)+1}(L;\mc P^*)$  (as well as possibly some of those above it in \eqref{E: susplink}) will appear in $H^*(\mc P^*_V)$ even though it does not appear in $H^i(\PP^*_V)$, but these are not necessarily trivial. If
 $\vec p_1(S)=\vec \p_1(\bS)+k+1$, i.e.\ if $\vec p_1(S)=\vec \p_1(\bS)+\codim(S)-\codim(\bS)$, the only problematic degree is 
 $H^{\vec \p_1(\bS)+k+2}(\mc P^*_V)\cong T^{\vec p_2(S)}\left(\H^{\vec \p_1(\bS)+1}(L;\mc P^*)/ T^{\vec p_2(\bS)}\H^{\vec p_1(\bS)+1}(L;\mc P^*)\right)$. For this to vanish in general we need $\vec p_2(S)\subset \vec \p_2(\bS)$.

\section{Dimension axioms}\label{S: dimension}

In addition to the original Ax1 axioms of \cite[Section 3.3]{GM2} (and the slight modification Ax1'), the Deligne sheaves of Goresky-MacPherson can also be characterized by a very different set of axioms that were used in the original proofs   of topological invariance of intersection homology in \cite[Section 4]{GM2}. Called Ax2, these are phrased in terms of the support and ``cosupport'' dimensions of the Deligne sheaves, i.e.\ the dimensions of the sets on which $H^i(\mc P_x)$ and $H^i(f_x^!\mc P^*)$ are non-zero for the various $i$. This perspective has historically been very useful, to the extent that these axioms are sometimes used to define intersection homology, e.g.\ see \cite{HS91}. In this section we formulate a version of these axioms for ts-Deligne sheaves and show that they are equivalent to the TAx1 axioms discussed above. We culminate with Theorem \ref{T: constrained2}, which is another formulation of our topological invariance results more attuned to this context. 

\paragraph{Some algebra notation.} Starting in Section \ref{S: more constrained} we will need some notation generalizing $T^{\wp}$. Recall that if $A$ is a finitely-generated module over the PID $R$ then $A\cong R^{m_0}\oplus \bigoplus R/\langle p_i^{m_i}\rangle$, where  the $p_i$ are primes (not necessarily unique), $m_0\geq 0$, and $m_i\geq 1$ for $i\neq 0$. Alternatively, as $R\cong R/\langle 0^m\rangle$ for any $m>0$, we can also write more consistently $A\cong \bigoplus R/\langle p_i^{m_i}\rangle$, where each $p_i$ is a prime or $0$ and each $m_i>0$. Our previous construction $T^{\wp}A$ isolates the summands of $A$ for which $p_i\in \wp$. The following construction is analogous but allowing the possibility $p_i=0$.

\begin{definition}\label{D: mf f}
Let $P_+(R)=P(R)\cup\{0\}$, where $P(R)$ is the set of primes of $R$ (up to unit). 
If $A\cong \bigoplus R/\langle p_i^{m_i}\rangle$ is a finitely-generated $R$-module and $\wp\subset P_+(R)$, define the \emph{$\wp$-component} of $A$ to be the summand $C^\wp A= \bigoplus_{p_i\in \wp} R/\langle p_i^{m_i}\rangle$. If $\mf p\in P(R)$ is a single element, we abuse notation and write $C^{\mf p}A$ instead of $C^{\{\mf p\}}A$.
\end{definition}

For example, if $R=\Z$ and $A\cong \Z^3 \oplus  \Z_8\oplus \Z_{25}$, then $C^{\{5,7\}}A=\Z_{25}$, $C^{\{0,2,7\}}=\Z^3\oplus \Z_8$, and $C^0A=\Z^3$.
Note that $C^\wp A$ is not precisely a submodule of $A$ in general, as the isomorphisms $A\cong \bigoplus R/\langle p_i^{m_i}\rangle$ are not canonical, but this construction will be sufficient for our purposes as we will be interested primarily in whether or not $C^\wp A=0$ or, equivalently, whether or not $A$ has certain kinds of torsion and/or a nontrivial torsion-free quotient.

\subsection{Perversity and coefficient constraints}\label{S: dim const}

Unfortunately, support and cosupport axioms can only characterize ts-Deligne sheaves if we limit ourselves to constrained ts-perversities. In this subsection we will see why that is. The basic issue is that we know our ts-Deligne sheaves are characterized by the axioms TAx1' and so would like to see when it is possible to recover the information content of those axioms from support and cosupport information. 
 We will further assume below that our constrained ts-perversities satisfy $\vec p_1(2)=0$, i.e.\ that they are strongly constrained. While not strictly necessary, this stronger condition allows us to avoid dealing with a plethora of case analyses and strong restrictions on ts-coefficients.

We first note that knowing that some property holds on some $k$-dimensional union of strata is not enough by itself to tell us whether or not that property holds on \emph{all} $k$-dimensional strata. So in order to convert (co)support information into information about behavior on all strata of a given dimension, all strata of the same dimension need to be treated equivalently, i.e.\ we need to consider ts-perversities that are functions of codimension alone.

To see why $\vec p_1$ must be nondecreasing, let us simplify and consider field coefficients, in which case our ts-Deligne sheaves are simply the usual Deligne sheaves. The axioms TAx1' tell us that the key information in this case for characterizing Deligne sheaves is knowing for each stratum $S$ the degrees for which $H^i(\mc P^*_x)$ and $H^i(f_x^!\mc P^*)$ are $0$ for $x\in S$. Note that constructibility assumptions will tell us that these modules vanish for some $x\in S$ if and only if they vanish for all $x\in S$. 
Let us focus on $H^i(\mc P^*_x)$ and consider how TAx1' translates into support information.
The diagram \eqref{E: perv-} below serves as a good model (though focus only on the * entries for now) with codimension $k$ increasing to the right, degree $i$ increasing upward,  the heights of the columns given by the values of $\vec p_1(k)$, and so each $*$ representing a (possibly) non-zero $H^i(\mc P^*_x)$, $x\in X_{n-k}$.  

Suppose now that for a specific $i$ we know that the dimension of the support of $H^i(\mc P^*_x)$ is $n-k$. This tells us that $H^i(\mc P^*_x)=0$ for $x$ in strata of codimensions $<k$, corresponding in the diagram to no $*$ entries at height $i$ for columns left of $k$.  It also tells us that $\vec p_1(k)$ must be $\geq i$. Now if  
$\vec p_1$ is nondecreasing then as $k$ increases the columns get taller and as $i$ increases the support dimensions of the $H^i(\mc P^*_x)$ get smaller. Suppose, however, that we allow $\vec p_1$ to decrease at some point; e.g.\ suppose we change the $\vec p$ shown in the diagram so that $\vec p_1(9)=0$. Since the diagram suggests that the support dimension of $H^1(\mc P^*_x)$ is $n-5$, dropping $\vec p_1(9)$ to force $H^1(\mc P^*_x)=0$ on a set of dimension $n-9$ does not alter the overall support dimension of $H^1(\mc P^*_x)$. Conversely, if our only information is support dimensions and we are trying to recover values of $\vec p_1(k)$, the vanishing of $H^1(\mc P^*_x)$ on an $n-9$ dimensional subspace of an $n-5$ dimensional support will not be detectable, and so would not be enough to determine the true value of  $\vec p_1(9)$. Roughly said: if $\vec p_1$ is nondecreasing, then a diagram such as \eqref{E: perv-} allows us to recover the column heights from the row depths and vice versa; this is essentially the content of the support axiom. However,  if $\vec p_1$ is allowed to decrease, this is no longer possible.

A similar consideration implies that we need the dual perversity $D\vec p$  to be nondecreasing in its first component (recall Definition \ref{D: ts-perv}). To see this, let $\mc D$ denote Verdier duality; then \cite[Theorem 4.19]{GBF32} says that $\mc D\mc P_{\vec p}^*[-n]\cong \mc P^*_{D\vec p}$, where $\mc P^*_{D\vec p}$ is also taken with respect to the dual coefficient system $\mc D\mc E$. The key observation for our purposes is that $$f_x^!\mc P^*_{\vec p}\cong f_x^!\mc D\mc D\mc P^*_{\vec p}\cong \mc Df_x^*\mc D\mc P^*_{\vec p}\cong \mc Df_x^*\mc P^*_{D\vec p}[n].$$
So continuing to assume field coefficients and letting $X$ be a pseudomanifold for simplicity, we have $H^i(f_x^!\mc P^*_{\vec p})\cong H^{n-i}(f_x^*\mc P^*_{D\vec p})$, using the Universal Coefficient Theorem for Verdier duality \cite[Section V.7.7]{Bo} and the finite generation implied by the constructibility of the sheaves \cite[Theorem 4.10]{GBF32}. Consequently, information about $H^i(f_x^!\mc P^*_{\vec p})$ is equivalent to information about $H^{n-i}(f_x^*\mc P^*_{D\vec p})$, and so the same argument above applies to say that obtaining full TAx1' information from cosupport dimensions relies on the presupposition that $D\vec p_1$ is nondecreasing. 

Since $\vec p_1(k)+D\vec p_1(k)=k-2$, we can only have  $\vec p_1$ and $D\vec p_1$  both nondecreasing if $\vec p_1$ satisfies the Goresky-MacPherson condition $\vec p_1(k)\leq \vec p_1(k+1)\leq \vec p_1(k)+1$. Furthermore, if we want to allow
the full range of possible ts-coefficients so that $H^i(\mc E_x)$ may be nonzero for both $i=0,1$ then we should treat $\mc P^*$ as if it is also truncated over the regular strata using $\vec p_1(0)\geq 0$. The nondecreasing requirements on $\vec p_1$ and $D\vec p_1$ then imply that we must have both $\vec p_1(k)$ and $D\vec p_1(k)$ always $\geq 0$. This is only possible if there are no codimension one strata and $\vec p_1(2)=0$. Thus we see that $\vec p$ must be constrained with $\vec p_1(2)=0$, and codimension one strata must be disallowed.

\begin{remark}
This last choice of $\vec p_1(2)=0$ is a bit artificial. If we allow either $H^1(\mc E)=0$ or $H^0(\mc E)=0$ then we could again consider any constrained ts-perversity $\vec p$ with $\mc E\in {}^{\vec p(2)}D^\heartsuit(\text{Dom}(\mc E))$. As noted above, however,  we will leave these more general cases to the interested reader.
\end{remark}

Similar considerations imply that if $\vec p_1(k)$ stays constant over some range of $k$, the sets $\vec p_2(k)$ must be nondecreasing. The same restriction on $D\vec p$ implies that in a range where $\vec p_1(k)$ is strictly increasing, the sets $\vec p_2(k)$ must be nonincreasing. Altogether, we have now argued that we should limit ourselves to strongly constrained ts-perversities (or at least constrained ones).

Lastly, there is one other way in which we must constrain our data. Returning to PID coefficients, suppose that $T^{\mf p}H^1(\mc E_x)\neq 0$  for some $\mf p\in P(R)$. Then $\dim\{\text{supp}(T^{\mf p}H^1(\mc P^*_x))\}=n$, and analogously to the above arguments, support information would be insufficient to tell us about $T^{\mf p}H^1(\mc P^*_x)$ on higher codimension strata. To remedy this, we must assume that if $T^{\mf p}H^1(\mc E_x)\neq 0$ then $\mf p\in \vec p_2(k)$ for all $k$ such that $\vec p_1(k)=0$, so that $\mf p$ torsion is always allowed in degree $1$. In particular, we must have that $H^1(\mc E_x)$ is $\vec p_2(2)$-torsion.
Analogously, using that $\vec p_1(2)=0$, TAx1' also says that we will need ts-Deligne sheaves to have $T^{\vec p_2(2)}H^n(f^!_x\mc P^*)=0$. But on the manifold $U_2$, we have  $H^n(f^!_x\mc P^*)\cong H^n(f_x^!\mc E)\cong H^0(\mc E_x)$, so we will not be able to detect $T^{\vec p_2(2)}H^n(f^!_x\mc P^*)=0$ on $X_{n-2}$ if $T^{\vec p_2(2)}H^0(\mc E_x)$ is ever non-zero, as this would imply $\dim\{x\mid T^{\vec p_2(2)}H^0(\mc E_x)\neq 0\}=n$. So if $T^{\mf p}H^0(\mc E_x)\neq 0$ for some $x$ then we need to have $\mf p\notin \vec p_2(2)$. Alternatively, if $\mf p\in \vec p_2(2)$, then we must have $T^{\vec p_2(2)}H^0(\mc E_x)=0$ for all $x$. But together these conditions are equivalent to $\mc E\in {}^{\vec p(2)}D^\heartsuit(\text{Dom}(\mc E))$.

Therefore, we limit ourselves in this section primarily to the case where $\vec p$ is a strongly constrained ts-perversities (Definition \ref{D: constrained}) with $\mc E\in {}^{\vec p(2)}D^\heartsuit(\text{Dom}(\mc E))$, though we will see in Section \ref{S: weak} that we can also consider weakly constrained ts-perversities if we allow ourselves some additional information.

\subsection{More about constrained perversities}\label{S: more constrained}

Classically, one can visual perversities satisfying the Goresky-MacPherson condition as ``sub-step'' functions. Similarly, strongly constrained ts-perversities can be visualized in diagrams such as the following in which the ground ring is $\Z$:

\begin{equation}\label{E: perv-}
\begin{array}{c|ccccccccccc}
4 &  &  &  &  &  &  &   &\{5\} & \{2,5,7\}&*  \\ 
3 &  &  &  &  &  &  &\{5\}   &* & *&*  \\ 
2 &  &  &  &  &  &  \{2,5\} &*  & * & *  & * \\ 
1 &  & \{2\} & \{2,3\} & * & * & * & * & * & * & *  \\ 
0 & * & * & * & * & * & * & * & * & * & * & \\
\hline
&2&3&4&5&6&7&8&9&10&11
\end{array} 
\end{equation}
Here the column labels along the bottom indicate  codimension parameters and the row labels on the left indicate  degrees. The asterisks display the height of $\vec p_1(k)$ while the sets of primes give $\vec p_2(k)$. The diagram is meant to evoke the cut-off degrees for the truncations determined by $\vec p$ with the primes in each $\vec p_2(k)$ surviving for an extra degree. In particular, the displayed ts-perversity is given by

\begin{align*}
\vec p(2)&=(0,\emptyset)&
\vec p(3)&=(0,\{2\})&
\vec p(4)&=(0,\{2,3\})\\
\vec p(5)&=(1,\emptyset)&
\vec p(6)&=(1,\emptyset)&
\vec p(7)&=(1,\{2,5\})\\
\vec p(8)&=(2,\{5\})&
\vec p(9)&=(3,\{5\})&
\vec p(10)&=(3,\{2,5,7\})\\
\vec p(11)&=(4,\emptyset)\\
\end{align*}
Note that $\vec p_1$ satisfies the Goresky-MacPherson conditions while $\vec p_2$ grows in each row but shrinks with each step up.

Now let $\vec p$ be a strongly constrained ts-perversity, let $m\in \Z_{\geq 0}$, and let $\mf p\in \primeset{P}_+(R)$. If $m>\vec p_1(n)+1$ or if $m=\vec p_1(n)$ and $\mf p\notin \vec p_2(n)$, set $\zvec p^{-1}(m,\mf p)=\infty$. Otherwise, generalizing the definition in \cite[Section 4.1]{GM2}, we define

\begin{equation*}
\zvec p^{-1}(m,\mf p)=
\begin{cases}
\min\{c\geq 2\mid \vec p_1(c)=m-1, \mf p\in \vec p_2(c)\}, &\text{if such a $c$ exists},\\
 \min\{c\geq 2\mid \vec p_1(c)=m\}, &\text{otherwise}.\\
\end{cases}
\end{equation*}

In terms of our diagram above, $\zvec p^{-1}(m,\mf p)$ is the column number of the leftmost entry in the $m$th row containing $\mf p$. If $\mf p$ is not listed explicitly in the $m$th row, then this is the column of the leftmost $*$. In our example above, $\zvec p^{-1}(1,2)=3$ while $\zvec p^{-1}(1,5)=5$. Note that since $\vec p_1(c)\neq -1$ for any $c$, we have $\zvec p^{-1}(0,\mf p)=2$ for any $\mf p$. Also, if $\mf p=0$ then $\mf p\in \vec p_2(c)$ is impossible and so $\zvec p^{-1}(m,0)= \min\{c\geq 2\mid \vec p_1(c)=m\}$ as in \cite{GM2}.

\begin{remark}\label{R: empty}
If $\vec p_2(k)=\emptyset$ for all $k\geq 2$, then $\zvec p^{-1}$ reduces to the $\zvec p_1^{-1}$ of \cite[Section 4.1]{GM2} (in \cite[Section V.4.6]{Bo}, Borel writes $\geq$ instead of $=$ in the definition, but under the Goresky-MacPherson perversity restriction, $\min\{c\mid \vec p_1(c)=m\}=\min\{c\mid \vec p_1(c)\geq m\}$ since $\vec p_1$ must take all values between $0$ and $\vec p_1(n)$). 

On the other hand, if $\vec p_2(k)=P(R)$ for all $k\geq 2$ then $\zvec p^{-1}(m,\mf p)=\min\{c\geq 2\mid \vec p_1(c)=m-1\}=\zvec p^{-1}(m-1,0)$ for all $\mf p\in P(R)$ and all $m>0$. 
\end{remark}

\begin{remark}\label{R: inverse}
Thinking in terms of diagrams as above and using that $\vec p$ is strongly constrained, we see that for a fixed $\mf p\in \primeset{P}_+ (R)$  we have $k\geq \zvec p^{-1}(m,\mf p)$ if and only if either

\begin{enumerate}
\item $\vec p_1(k)\geq m$ or 

\item $\vec p_1(k)=m-1$ and $\mf p\in \vec p_2(k)$.

\end{enumerate}
In particular if $\mf p=0$ then $k\geq \zvec p^{-1}(m,0)$ if and only 
$\vec p_1(k)\geq m$. If $m=0$ this says that $k\geq \zvec p^{-1}(0,\mf p)=2$ if and only if 
$\vec p_1(k)\geq 0$, though this is tautological as both statements are always true.

This observation does not require the last property of constrained ts-perversities, only the first four. However, below we will need these statements for both $\vec p$ and its dual $D\vec p$. The last property is needed for the dual to also be strongly constrained.
\end{remark}

\begin{lemma}
$\vec p$ is a strongly constrained ts-perversity if and only if its dual $D\vec p$ is also strongly constrained.
\end{lemma}
\begin{proof}
For ease of notation, let $\vec q=D\vec p$. As $DD\vec p=\vec p$, it suffices to show that if $\vec p$ is strongly constrained then so is $\vec q$. Clearly $\vec p$ is a function of codimension if and only if $\vec q$ is. The growth condition only concerns $\vec p_1$ and $\vec q_1$ and is true of classical Goresky-MacPherson perversities; it follows from $\vec p_1(k)+\vec q_1(k)=k-2$. Similarly $\vec p_1(2)=0$ if and only if $\vec q_1(2)=0$. Next, note that 
$\vec p_1(k+1)=\vec p_1(k)+1$ if and only if $\vec q_1(k+1)=\vec q_1(k)$ and $\vec p_1(k+1)=\vec p_1(k)$ if and only if  $\vec q_1(k+1)=\vec q_1(k)+1$. Further, $\vec p_2(k)$ and $\vec q_2(k)$ are complementary sets of primes. So if $\vec p_2(k+1)\supset \vec p_2(k)$ whenever $\vec p_1(k+1)=\vec p_1(k)$ then $\vec q_2(k+1)\subset \vec q_2(k)$ whenever $\vec q_1(k+1)=\vec q_1(k)+1$. Similarly   if $\vec p_2(k+1)\subset \vec p_2(k)$ whenever $\vec p_1(k+1)=\vec p_1(k)+1$ then $\vec q_2(k+1)\supset \vec q_2(k)$ whenever $\vec q_1(k+1)=\vec q_1(k)$.
\end{proof}

\subsection{Support and cosupport axioms for strongly constrained ts-Deligne sheaves}

We can now formulate a version of the Goresky-MacPherson axioms Ax2. We follow more closely the exposition in \cite[Section V.4]{Bo}, which is more detailed than \cite{GM2}. In the following definition we assume $X$ to be a CS set of dimension $n$ with no codimension one strata, that $\vec p$ is a strongly constrained ts-perversity, that $\vec q=D\vec p$, that $X$ is adapted to the ts-coefficient system $\mc E$, and that $\mc E\in {}^{\vec p(2)}D^\heartsuit(\text{Dom}(\mc E))$.

\begin{definition}\label{T: Ax2X}
We say the sheaf complex $\ms S^*$ satisfies the \emph{Axioms TAx2$(X,\vec p, \mc E)$} if 

\begin{enumerate}[label=\alph*., ref=\alph*]
\item\label{A2: bounded} $\ms S^*$ is $X$-clc and it is 
quasi-isomorphic to a complex that is bounded and that is $0$ in negative degrees;
\item \label{A2: coeffs} $\ms S^*|_{U_2}\cong\mc E|_{U_2}$;

\item \label{A2: supp} \begin{enumerate}
\item If  $j>1$ then $\dim\{x\in X\mid C^{\mf p}H^j(\ms S^*_x)\neq 0\}\leq n-\zvec p^{-1}(j,\mf p)$ for all $\mf p\in P_+(R)$.
\item $\dim\{x\in X\mid C^{\mf p}H^1(\ms S^*_x)\neq 0\}\leq n-\zvec p^{-1}(1,\mf p)$ 
 for all $\mf p\in P_+(R)$ such that  $\mf p\notin \vec p_2(2)$. 
\end{enumerate}

\item \label{A2: cosupp} 
\begin{enumerate}
\item If $j<n$ then  $\dim\{x\in X\mid C^{\mf p}H^j(f^!_x\ms S^*)\neq 0\}\leq n-\zvec q^{-1}(n-j+1,\mf p)$ for all $\mf p\in P(R)$ and $\dim\{x\in X\mid C^{0}H^j(f^!_x\ms S^*)\neq 0\}\leq n-\zvec q^{-1}(n-j,0)$. 

\item  $\dim\{x\in X\mid C^{\mf p}H^n(f^!_x\ms S^*)\neq 0\}\leq n-
\zvec q^{-1}(1,\mf p)$ for all $\mf p\in\vec p_2(2)$. 
\end{enumerate}
\end{enumerate}
\end{definition}

Note that in axiom \eqref{A2: supp} we may have $\mf p=0$, but in axiom \eqref{A2: cosupp} the only appearance of $0$ as an element of $P_+(R)$ is explicit as $0$ is not in $\vec p_2(2)$ nor $P(R)$. If $\vec p_2(k)=\emptyset$ for all $k$ then these axioms reduce to those of Borel in \cite[Section 4.7]{Bo}.

\begin{proposition}
The sheaf complex  $\ms S^*$  satisfies  TAx1'$(X,\vec p, \mc E)$ if and only it satisfies TAx2$(X,\vec p, \mc E)$.
\end{proposition}
\begin{proof}
The proof emulates that of \cite[Proposition V.4.9]{Bo}, though it is a bit more complicated since we must consider the torsion effects and also  take more care with some special cases when the degree $j$ is near $0$ or $n$. We label codimension by $k$ and assume $2\leq k\leq n$ throughout, since $X$ has no codimension one strata by assumption.
The first two axioms of  TAx1' and  TAx2 agree, so we will show that the two third axioms and the two fourth axioms are equivalent given the hypotheses. 

\paragraph{$(1'c\Rightarrow 2c)$.} First suppose $\ms S^*$ satisfies TAx1'\ref{A': truncate}. We first observe that it is possible to have $\dim\{x\in X\mid C^{\mf p}H^j(\ms S^*_x)\neq 0\}=n$ if $j=0$ or if $j=1$ and $\mf p\in \vec p_2(2)$ since $\ms S^*|_{U_2}\cong \mc E$ and these properties are true of $\mc E$. However, also thanks to the properties of $\mc E$, these are the only cases for which  $\dim\{x\in X\mid C^{\mf p}H^j(\ms S^*_x)\neq 0\}=n$. In all other cases, $C^{\mf p}H^j(\ms S^*_x)$ is supported in $X^{n-2}$. So consider these other cases, i.e.\ either $j=1$ and $\mf p\notin \vec p_2(2)$ or $j>1$. If $x\in X_{n-k}$ for $k\geq 2$ and $C^{\mf p}H^j(\ms S^*_x)\neq 0$   then by TAx1'\ref{A': truncate} either $j\leq \vec p_1(k)$ or we have $j=\vec p_1(k)+1$ and $\mf p\in \vec p_2(k)$. By Remark \ref{R: inverse}, this implies $k\geq \zvec p^{-1}(j,\mf p)$. Then $\dim X_{n-k}\leq n-k\leq n-p^{-1}(j,\mf p)$. This yields TAx2\ref{A2: supp}.

\paragraph{$(1'c\Leftarrow 2c)$.} Conversely, suppose TAx2\ref{A2: supp} holds.
 Now fix $k\geq 2$ and $x\in X_{n-k}$. We must show that $C^{\mf p}H^j(\ms S^*_x)=0$ if $j>\vec p_1(k)+1$ or if we have $j=\vec p_1(k)+1$ and $\mf p\in P_+(R)-\vec p_2(k)$. So first suppose $j\geq \vec p_1(k)+2\geq 2$, the last inequality by the assumptions on $\vec p$.  Then $\zvec p^{-1}(j,\mf p)>k$ using Remark \ref{R: inverse}. Similarly Remark \ref{R: inverse} implies that if  $j=\vec p_1(k)+1\geq 1$ and $\mf p\notin \vec p_2(k)$  
then $\zvec p^{-1}(j,\mf p)>k$. 
 We also note that  since $\vec p_2(2)\subset \vec p_2(k)$ for all $k$ such that $\vec p_1(k)=0$, if $\vec p_1(k)=0$ and $\mf p\notin \vec p_2(k)$, then $\mf p\notin\vec p_2(2)$. So if $j\geq \vec p_1(k)+2$ or if we have $j=\vec p_1(k)+1\geq 1$ and $\mf p\notin \vec p_2(k)$ then either $j\geq 2$ or $j=1$ with $\mf p\notin \vec p_2(2)$. In either case the assumptions say that   $\dim\{x\in X\mid C^{\mf p}H^j(\ms S^*_x)\neq 0\}\leq n-\zvec p^{-1}(j,\mf p)<n-k$. Since $\ms S^*$ is $X$-clc, if $C^{\mf p}H^j(\ms S^*_x)\neq 0$ then also $C^{\mf p}H^j(\ms S^*_y)\neq 0$ for all $y$ in the $n-k$ dimensional stratum containing $x$. Hence we must have in these cases  $C^{\mf p}H^j(\ms S^*_x)= 0$. This implies  
 TAx1'\ref{A': truncate}.

\paragraph{$(1'd\Rightarrow 2d)$.} 
Next suppose $\ms S^*$ satisfies TAx1'\ref{A': attach}. Since $\ms S^*|_{U_2}\cong \mc E$ and since $f^!_x=f^*[-n]$ on $U_2$, we have for $x\in U_2$ that $C^{\mf p}H^j(f_x^!\ms S^*)\neq 0$ only for $j=n,n+1$ and furthermore that   $C^{\mf p}H^n(f_x^!\ms S^*)=0$ if $\mf p\in \vec p_2(2)$. So for $j<n$ or for $j=n$ and $\mf p\in \vec p_2(2)$, the dimension $\dim\{x\in X\mid C^{\mf p}H^j(f^!_x\ms S^*)\neq 0\}$ is determined entirely by points in $X^{n-2}$.

So suppose $x\in X_{n-k}$ for $k\geq 2$  and that $j<n$ or that  $j=n$ and $\mf p\in \vec p_2(2)$. If $C^{\mf p}H^j(f_x^!\ms S^*)\neq 0$ then by assumption either

\begin{enumerate}
\item $j\geq \vec p_1(k)+n-k+3=n-\vec q_1(k)+1$, or
\item $j= \vec p_1(k)+n-k+2=n-\vec q_1(k)$ and  $\mf p\in \vec q_2(k)\cup \{0\}$. 
\end{enumerate}
In the first scenario we can conclude by Remark \ref{R: inverse} that $k\geq \zvec q^{-1}(n-j+1,\mf p)$ while the second scenario gives us $k\geq \zvec q^{-1}(n-j+1,\mf p)$ if $\mf p\in \vec q_2(k)$ and $k\geq \zvec q^{-1}(n-j,0)$ if $\mf p=0$. Note that in either case\footnote{We remark that if $\vec p_2(k)=\emptyset$ for all $k$, which corresponds to $\ms S^*$ satisfying the original Goresky-MacPherson axioms, then $\vec q_2(k)=P(R)$ for all $k$. In this case $q^{-1}(n-j+1,\mf p)=q^{-1}(n-j,0)$ by Remark \ref{R: empty}, which is again consistent with the expectation from the classical case.} if $\mf p\neq 0$ then we conclude  $k\geq \zvec q^{-1}(n-j+1,\mf p)$, though of course the value of $q^{-1}(n-j+1,\mf p)$ can depend on $\mf p$. So if $\mf p\neq 0$, then $\dim(X_{n-k})\leq n-k\leq n-\zvec q^{-1}(n-j+1,\mf p)$ and so 
$\dim\{x\in X\mid C^{\mf p}H^j(f^!_x\ms S^*)\neq 0\}\leq n-\zvec q^{-1}(n-j+1,\mf p)$. Similarly, if $\mf p=0$ we obtain $\dim\{x\in X\mid C^{0}H^j(f^!_x\ms S^*)\neq 0\}\leq n-\zvec q^{-1}(n-j,0)$.
So we have TAx2\ref{A2: cosupp}.

\paragraph{$(1'd\Leftarrow 2d)$.} Finally, suppose  TAx2\ref{A2: supp} holds. If $n\leq j\leq \vec p_1(k)+n-k+1$, then $k-1\leq \vec p_1(k)$, which is impossible. So in considering the first condition of TAx1'\ref{A': attach} we may assume $j<n$. Similarly, $n\leq j= \vec p_1(k)+n-k+2$ implies $k-2\leq \vec p_1(k)$, which is possible only when we have equalities and so $j=n$. 
Thus for the second condition of TAx1'\ref{A': attach} we may consider only $j\leq n$.

First suppose $j<n$. 
 Since $\ms S^*$ is $X$-clc (by either set of axioms), $j^!_k\ms S^*$ is clc by \cite[Proposition 4.0.2.3]{Sch03}, so if $C^{\mf p}H^j(f_x^!\ms S^*)\neq 0$ for some $x\in X_{n-k}$ then the same is true for all other points $X_{n-k}$. Thus for $\mf p\in P(R)$, if $x\in X_{n-k}$ and $C^{\mf p}H^j(f_x^!\ms S^*)\neq 0$ then $n-k\leq n-\zvec q^{-1}(n-j+1,\mf p)$, so $k\geq \zvec q^{-1}(n-j+1,\mf p)$. Thus by Remark \ref{R: inverse} either $\vec q_1(k)\geq n-j+1$ or 
$\vec q_1(k)=n-j$ and $\mf p\in \vec q_2(k)$. This translates to $j\geq \vec p_1(k)+n-k+3$ or $j= \vec p_1(k)+n-k+2$ and $\mf p\notin \vec p_2(k)$. Similarly, if $\mf p=0$, the assumptions imply $n-k\leq  n-\zvec q^{-1}(n-j,0)$ or $k\geq \zvec q^{-1}(n-j,0)$, which means that $\vec q_1(k)\geq n-j$. This translates to  
$j\geq \vec p_1(k)+n-k+2$. So, altogether, if $C^{\mf p}H^j(f_x^!\ms S^*)\neq 0$ then $j\geq \vec p_1(k)+n-k+2$ and if $j=\vec p_1(k)+n-k+2$ then $\mf p\notin \vec p_2(k)$. This is TAx1'\ref{A': attach}.

Now suppose $x\in X_{n-k}$,  $j=n=\vec p_1(k)+n-k+2$, and  $\mf p\in \vec p_2(k)$. We must show that $C^{\mf p}H^n(f_x^!\mc S^*)=0$.  In this case $\vec p_1(k)=k-2$ which is possible only if $\vec p_1(k)=k-2$ up through codimension $k$. So $\vec q_1$ must be $\bar 0$ up through $k$. In this case the hypotheses on constrained ts-perversities imply that $\vec p_2(k)\subset \vec p_2(c)\subset \vec p_2(2)$ for all $2\leq c\leq k$, so in particular we may use the second condition of TAx2\ref{A2: cosupp}. Further, since $\mf p\in\vec p_2(c)$ for all $2\leq c\leq k$ then $\mf p\notin \vec q_2(c)$ for all $2\leq c\leq k$ and our hypotheses imply $n-j+1=1$. Consequently, $\zvec q^{-1}(n-j+1,\mf p)=\zvec q^{-1}(1,\mf p)>k$. So  TAx2\ref{A2: cosupp} implies that  $\dim\{C^{\mf p}H^n(f_x^!\mc S^*)\neq0\}<n-k$. So $C^{\mf p}H^n(f_x^!\mc S^*)=0$ as needed. 

Altogether this shows TAx1.\ref{A': attach}.
\end{proof}

\begin{corollary}\label{C: all axioms}
Suppose $\vec p$ is a strongly constrained ts-perversity, $X$ is a CS set without codimension one strata, $X$ is adapted to the ts-coefficient system $\mc E$, and  $\mc E\in {}^{\vec p(2)}D^\heartsuit(\text{Dom}(\mc E))$. Then $\ms S^*$ satisfies TAx1$(X, \vec p, \mc E)$ if and only if it satisfies  TAx1'$(X, \vec p, \mc E)$ if and only if it satisfies TAx2$(X, \vec p, \mc E)$. Any of these axioms characterize $\ms S^*$ uniquely up to isomorphism as the ts-Deligne sheaf $\mc P^*_{X,\vec p,\mc E}$.
\end{corollary}
\begin{proof}
This follows directly from the preceding proposition, Theorem \ref{T: ax equiv}, and \cite[Theorem 4.8]{GBF32}. 
\end{proof}

\begin{definition}\label{T: Ax2X'}
Let $|X|$ be a space, let $\mc E$ be a maximal  ts-coefficient system on $|X|$, and let $\vec p$ be a strongly constrained perversity with  $\mc E\in {}^{\vec p(2)}D^\heartsuit(\text{Dom}(\mc E))$.  We say $\ms S^*$ satisfies the \emph{Axioms TAx2'$(\vec p, \mc E)$} if 

\begin{enumerate}[label=\alph*., ref=\alph*]
\item\label{A2': bounded} $\ms S^*$ is  
quasi-isomorphic to a complex that is bounded and that is $0$ in negative degrees;

\item \label{A2': coeffs} $\ms S^*$ is $X$-clc for some CS set stratification $X$ of $|X|$ without codimension one strata that is adapted to $\mc E$, and 
$\ms S^*|_{U_2}\cong\mc E|_{U_2}$;

\item \label{A2': supp} \begin{enumerate}
\item If  $j>1$ then $\dim\{x\in |X|\mid C^{\mf p}H^j(\ms S^*_x)\neq 0\}\leq n-\zvec p^{-1}(j,\mf p)$ for all $\mf p\in P_+(R)$.
\item $\dim\{x\in |X|\mid C^{\mf p}H^1(\ms S^*_x)\neq 0\}\leq n-\zvec p^{-1}(1,\mf p)$ 
 for all $\mf p\in P_+(R)$ such that  $\mf p\notin \vec p_2(2)$. 
\end{enumerate}

\item \label{A2': cosupp} 
\begin{enumerate}
\item If $j<n$ then  $\dim\{x\in |X|\mid C^{\mf p}H^j(f^!_x\ms S^*)\neq 0\}\leq n-\zvec q^{-1}(n-j+1,\mf p)$ for all $\mf p\in P(R)$ and $\dim\{x\in |X|\mid C^{0}H^j(f^!_x\ms S^*)\neq 0\}\leq n-\zvec q^{-1}(n-j,0)$. 

\item  $\dim\{x\in |X|\mid C^{\mf p}H^n(f^!_x\ms S^*)\neq 0\}\leq n-
\zvec q^{-1}(1,\mf p)$ for all $\mf p\in\vec p_2(2)$. 
\end{enumerate}

\end{enumerate}
\end{definition}

Our Axioms TAx2'$(\vec p, \mc E)$ are slightly different from the Axioms (Ax2)${}_{\mc E}$ of \cite[Section 4.13]{Bo} even beyond the incorporation of torsion information and the generalization to CS sets. As observed in \cite[Remark V.4.14.b]{Bo}, the axioms there do not assume any relation between the stratification of $X$ and the coefficient system $\mc E$ as we have done in the second axiom. However, it is also observed in this remark that (in that setting), a sheaf complex satisfies (Ax2)${}_{\mc E}$ if and only if it satisfies (Ax2)${}_{X,\mc E}$ for some stratification  (in Borel's case a pseudomanifold stratification) adapted to $\mc E$. Our  Axioms TAx2'$(\vec p, \mc E)$ are therefore a bit less general than Borel's Axioms (Ax2)${}_{\mc E}$ in this sense, though as in Section \ref{S: maximal E} we can adapt Borel's remark if each $\mc E^i$ is a local system and $\mc E^i=0$ for sufficiently large $|i|$. In this case we need not assume $X$ adapted to $\mc E$ in the second axiom.

In either case, the upshot is the same:  a sheaf complex satisfies our TAx2'$(\vec p, \mc E)$ if and only if it satisfies TAx2$(X, \vec p, \mc E)$ for some stratification $X$ of $|X|$ (adapted to $\mc E$); furthermore the axioms TAx2'$(\vec p, \mc E)$ are stratification independent. Putting this together with our prior results we obtain a torsion sensitive analogue of \cite[Theorem V.4.15]{Bo}:

\begin{theorem}\label{T: constrained2}
Suppose  $\mc E$ is a maximal  ts-coefficient system with domain $U_{\mc E}$ on a space $|X|$, that $\vec p$ is a strongly constrained perversity, and that  $\mc E\in {}^{\vec p(2)}D^\heartsuit(U_{\mc E})$. Suppose $X$ has a CS set stratification with no codimension one strata that is fully adapted to $\mc E$. Then there is a sheaf complex $\mc P^*$ satisfying TAx2'$(\vec p, \mc E)$ with $\mc P^*|_{U_{\mc E}}\cong \mc E$ and such that $\mc P^*$ satisfies TAx2$(X,\vec p, \mc E)$ for every CS set stratification $X$ of $|X|$ without codimension one strata that is fully adapted to $\mc E$. 
\end{theorem}
\begin{proof}
As  $X$ has a CS set stratification with no codimension one strata that is fully adapted to $\mc E$, there is an intrinsic stratification $\mf X$ fully adapted to $\mc E$ by Proposition \ref{P: intrinsic}. Let $\mc P^*$  be the ts-Deligne sheaf with respect to $\mf X$. Then $\mc P^*|_{U_{\mc E}}\cong \mc E$ since $\mf X-\mf X^{n-2}=\mf X-\mf X^{n-1}=U_{\mc E}$ by Proposition \ref{P: intrinsic}. Proposition \ref{P: constrained} implies $\mc P^*$ is quasi-isomorphic to the ts-Deligne sheaves coming from any of the other stratifications, and we know these satisfy the axioms by Corollary \ref{C: all axioms}. 
\end{proof}

\subsection{Support and cosupport axioms for weakly constrained ts-Deligne sheaves}\label{S: weak}

Analogously to the Goresky-MacPherson axioms Ax2, our Axioms TAx2  depend only very weakly on the stratification: TAx2 only mentions a particular stratification to specify that it is adapted to the coefficients, that $\mc P^*$ is $X$-clc, and that $\mc P^*\cong \mc E$ over the regular strata. TAx2' only asks for this with respect to \emph{some} stratification. Consequently we obtain our version of topological invariance in Theorem \ref{T: constrained2}.

In Section \ref{S: dim const} we argued that in order for the support and cosupport axioms to imply our earlier TAx1' axioms it is necessary to use constrained perversities with  $\mc E\in {}^{\vec p(2)}D^\heartsuit(\text{Dom}(\mc E))$ and to forbid codimension one strata. However, it is possible to avoid all of these constraints except for the Goresky-MacPherson growth condition at the expense of modifying the TAx2 axioms to depend more heavily on the stratification. In fact, we can obtain Theorem \ref{T: constrainedS}, below, which generalizes the main theorem of \cite{GBF11}. The proofs are all analogous to those above, though in fact simpler since special care no longer needs to be taken in extreme degrees. 

For the following, we let $\vec p$ be a weakly constrained perversity with domain $\Z_{\geq 1}$. We can then extend $\zvec p^{-1}$ to be a function $\Z\times P_+(R)\to \Z_{\geq 1}$ by declaring that if $m<\vec p_1(1)$ then $\zvec p^{-1}(m,\mf p)=1$.

\begin{definition}\label{T: Ax2XS}
Let $X$ be a CS set (possibly with codimension one strata) adapted to a ts-coefficient system $\mc E$. Let $\vec p$ be a weakly constrained ts-perversity. We say the sheaf complex $\ms S^*$ satisfies the \emph{Axioms TAx2$(X,\vec p, \mc E,X^{n-1})$} if 

\begin{enumerate}[label=\alph*., ref=\alph*]
\item $\ms S^*$ is $X$-clc and it is 
quasi-isomorphic to a complex that is bounded and that is $0$ in negative degrees;
\item  $\ms S^*|_{U_1}\cong\mc E|_{U_1}$;

\item $\dim\{x\in X^{n-1}\mid C^{\mf p}H^j(\ms S^*_x)\neq 0\}\leq n-\zvec p^{-1}(j,\mf p)$ for all $\mf p\in P_+(R)$.

\item  $\dim\{x\in X^{n-1}\mid C^{\mf p}H^j(f^!_x\ms S^*)\neq 0\}\leq n-\zvec q^{-1}(n-j+1,\mf p)$ for all $\mf p\in P(R)$ and $\dim\{x\in X^{n-1}\mid C^{0}H^j(f^!_x\ms S^*)\neq 0\}\leq n-\zvec q^{-1}(n-j,0)$. 

\end{enumerate}
\end{definition}

\begin{proposition}
Let $\vec p$ be a weakly constrained perversity, and suppose $X$ is a CS set, possibly with codimension one strata, adapted to the  ts-coefficient system $\mc E$. Then the sheaf complex  $\ms S^*$  satisfies  TAx1'$(X,\vec p, \mc E)$ if and only it satisfies TAx2$(X,\vec p, \mc E,X^{n-1})$.
\end{proposition}

\begin{definition}\label{T: Ax2X'S}
Let $|X|$ be a space, $\Sigma$ a closed subspace,  $\vec p$ a weakly constrained perversity, and $\mc E$ any maximal ts-coefficient system on $|X|$.  We say $\ms S^*$ satisfies the \emph{Axioms TAx2'$(\vec p, \mc E,\Sigma)$} if 

\begin{enumerate}[label=\alph*., ref=\alph*]
\item\label{A2': boundedS} $\ms S^*$ is  
quasi-isomorphic to a complex that is bounded and that is $0$ in negative degrees;

\item \label{A2': coeffsS} $\ms S^*$ is $X$-clc for some CS set stratification $X$ of $|X|$ (possibly with codimension one strata) such that $X^{n-1}=\Sigma$ and $X$ that is adapted to $\mc E$, and 
$\ms S^*|_{U_1}\cong\mc E|_{U_1}$;

\item $\dim\{x\in \Sigma\mid C^{\mf p}H^j(\ms S^*_x)\neq 0\}\leq n-\zvec p^{-1}(j,\mf p)$ for all $\mf p\in P_+(R)$.

\item $\dim\{x\in \Sigma \mid C^{\mf p}H^j(f^!_x\ms S^*)\neq 0\}\leq n-\zvec q^{-1}(n-j+1,\mf p)$ for all $\mf p\in P(R)$ and $\dim\{x\in \Sigma\mid C^{0}H^j(f^!_x\ms S^*)\neq 0\}\leq n-\zvec q^{-1}(n-j,0)$.

\end{enumerate}
\end{definition}

\begin{theorem}\label{T: constrainedS}
Suppose $\vec p$ is a weakly constrained perversity and that $\mc E$ is a maximal ts-coefficient system with domain $U_{\mc E}$ on a space $|X|$ with closed subspace $\Sigma$. Suppose $|X|$ has a CS set stratification $X$ that is fully adapted to $\mc E$ and such that $X^{n-1}=\Sigma$. Then there is a sheaf complex $\mc P^*$ satisfying TAx2'$(\vec p, \mc E,\Sigma)$ with $\mc P^*|_{U_{\mc E}-\Sigma}\cong \mc E|_{U_{\mc E}-\Sigma}$ and such that $\mc P^*$ satisfies TAx2$(X,\vec p, \mc E,\Sigma)$ for every CS set stratification $X$ with $X^{n-1}=\Sigma$  that is fully adapted to $\mc E$. 
\end{theorem}

\section{Intrinsic stratifications}\label{S: intrinsic}

In this section we consider common coarsenings of CS sets. In particular, each CS set possesses an intrinsic stratification as a CS set that coarsens all others; this is due to King and Sullivan \cite{Ki} and a thorough discussion can be found in \cite[Section 2.10]{GBF35}. This generalizes the classical situation for PL spaces, which is treated in \cite[Section 2.10.1]{GBF35}. 
Since not all stratifications are adapted to a given coefficient system, it will be necessary to refine the construction of a coarsest stratification to take the coefficients into account.\footnote{A version of such a construction for fairly general sheaf complexes can be found in Habegger-Saper \cite[Section 3]{HS91}.} Furthermore, we have seen in Section \ref{S: inv} that we may wish to only consider coarsenings that preserve some subspace, without letting strata in the subspace  ``merge'' with other strata not in the subspace.  
So this is a further ingredient we will consider for our common coarsenings. 

Many of the basic ideas are the same as in the above references, but in order to account for the additional ingredients we provide most of the details. No doubt some of our arguments would simplify if we were to consider only piecewise linear spaces, but there would still be intricacies in accounting for the coefficient systems and the fixed subspaces. Consequently, we choose to make the minor additional effort to work with CS sets in the topological category and so capture the greater generality.  
For example, working with not-necessarily-PL spaces allows us to apply these results when $\mc E$ is a local system defined only on the complement of a codimension-two locally flat submanifold of a topological manifold. Such a setting occurs, for example, in the study of topological locally flat knots $S^{n-2} \into S^2$. Another minor feature of working with CS sets is that all of our constructions in this section will be essentially ``by hand,'' without the need to invoke any major theorems of PL topology.

We recall that the usual intrinsic stratification of a CS set $X$ is determined  by an equivalence relation so that $x_0, x_1\in X$ are equivalent, denoted $x_0\sim x_1$,  if they possess neighborhoods $U_0,U_1$ such that $(U_0,x_0)\cong (U_1,x_1)$ as topological space pairs (i.e.\ ignoring the filtrations) \cite[Definition 2.10.3]{GBF35}. If $x_0,x_1$ are both in the same  stratum of $X$, then $x_0\sim x_1$ \cite[Lemma 2.10.4]{GBF35}. 
Furthermore, if we  let ${\mf X}^i$ be the union of the equivalence classes that only contain strata of $X$ of dimension $\leq i$, then the ${\mf X}^i$ filter $X$ as a CS set that does not depend on the initial filtration of $X$ as a CS set and that coarsens all other CS set stratifications \cite[Proposition 2.10.5]{GBF35}. This provides an intrinsic coarsest CS set stratification of $X$. 

To account for subspaces, the Frontier Condition \cite[Definition 2.2.16., Lemma 2.3.7]{GBF35} implies that if we have a stratum $S$ that we don't wish to merge with some other stratum $T$ of lower codimension then points in the closure of $S$ also cannot merge with $T$.  Consequently, it makes sense for our fixed subspaces to be closed unions of strata. We therefore make the following definition. The assumption that $X$ be fully adapted to a maximal ts-coefficient system $\mc E$ will be critical in the following arguments; see Section \ref{S: maximal E} for those definitions.

\begin{definition}
Let $X$ be a CS set fully adapted to the maximal ts-coefficient system $\mc E$ with domain the open $n$-manifold $U_{\mc E}$, and let C be a closed union of strata of $X$. We say $x\sim_{\mc E,C} y$ for points $x,y\in X$ if  there is a homeomorphism of space pairs (ignoring the stratifications) $h:(U_x,x)\to (U_y ,y)$ so that
\begin{enumerate}

\item $h(U_x\cap C)=U_y\cap C$,

\item $h(U_x\cap U_{\mc E})=U_y\cap U_{\mc E}$, and

\item $h^*(\mc E|_{U_y\cap U_{\mc E}})$ is quasi-isomorphic to $\mc E|_{U_x\cap U_{\mc E}}$, i.e.\ $h^*(\mc E|_{U_y\cap U_{\mc E}}) \cong \mc E|_{U_x\cap U_{\mc E}}$ in the derived category.

\end{enumerate}

\end{definition}

For the rest of our discussion we fix $\mc E$ and $C$ and so simply write $\sim $ for our relation.

\begin{lemma}\label{L: simE}\hfill
\begin{enumerate}

\item $\sim $ is an equivalence relation.

\item If $x,y\in X$ are in the same stratum of $X$ then $x\sim  y$.
\end{enumerate}
\end{lemma}
\begin{proof}
The relation is clearly reflexive. For symmetry and transitivity, the only parts that are not obvious are the behavior of $\mc E$. Suppose $x\sim y$ and $y\sim z$ with homeomorphisms $h:U_x\to U_y$ and $g:U_y\to U_z$. Then  $h^*\mc E|_{U_y\cap U_{\mc E}}\cong\mc E|_{U_x\cap U_{\mc E}}$ and  $g^*\mc E|_{U_z\cap U_{\mc E}}\cong\mc E|_{U_y\cap U_{\mc E}}$, so $(gh)^*\mc E|_{U_z\cap U_{\mc E}}\cong h^*g^*\mc E|_{U_z\cap U_{\mc E}}\cong h^*\mc E|_{U_y\cap U_{\mc E}}\cong \mc E|_{U_x\cap U_{\mc E}}$ demonstrating transitivity. Similarly, if $x\sim  y$ then $(h^{-1})^*\mc E|_{U_x\cap U_{\mc E}}\cong (h^{-1})^*h^*\mc E|_{U_y\cap U_{\mc E}}\cong (hh^{-1})^*\mc E|_{U_y\cap U_{\mc E}}\cong \mc E|_{U_y\cap U_{\mc E}}$. So $\sim $ is an equivalence relation.

Now suppose $U_x$ is a distinguished neighborhood of $x$ by the filtered homeomorphism $g:\R^k\times cL\to U_x$. Let $y\in U_x$ be 
 contained in the same stratum as $x$, in which case $U_x$ is also a distinguished neighborhood of $y$. Then $g^{-1}(x),g^{-1}(y)\subset \R^k\times \{v\}=\R^k$, where $v$ is the cone point. Let $f$ be a homeomorphism of $\R^k$ that takes $g^{-1}(x)$ to $g^{-1}(y)$, and let $h=f\times \id$ on $\R^k\times cL$. Then $ghg^{-1}$ is a homeomorphism $U_x\to U_y=U_x$ that takes $x$ to $y$. Since $ghg^{-1}$ preserves strata and since $X$ is fully adapted to $\mc E$ and $C$ is a union of strata, the map  $ghg^{-1}$ restricts to a homeomorphism from $U_x\cap U_{\mc E}$ to itself and also from $U_x\cap C$ to itself. Furthermore, since $\mc E$ is clc on its domain of definition and since $X$ is fully adapted to $\mc E$, $g^*\mc E$ will be clc on $g^{-1}(U_x\cap U_{\mc E})$, which will be a set of the form $\R^k\times V$. Let $\pi:\R^k\times cL\to cL$ be the projection. By \cite[Proposition 2.7.8]{KS}, $g^*\mc E\cong \pi^*R\pi_*g^*\mc E$ on its domain. So since $\pi h=\pi$ we have
$$(ghg^{-1})^*\mc E\cong (g^{-1})^*h^*g^*\mc E\cong (g^{-1})^*h^* \pi^*R\pi_*g^*\mc E \cong (g^{-1})^* \pi^*R\pi_*g^*\mc E\cong  (g^{-1})^*g^*\mc E\cong  \mc E$$ on its domain.
So $x\sim y$. 

Now suppose $x$ is in the stratum $S$ of $x$ and let $W$ be the set of points in $S$ equivalent to $x$. By the above argument, $W$ is an open subset of $S$.  On the other hand, if $y\in S$ is in the closure of $W$, then any distinguished neighborhood of $y$ must intersect $W$, so $y\in W$ by the above. Thus $W$ is closed. Since strata are connected, $W$ must be all of $S$.
\end{proof}

By the lemma, the equivalence classes under $\sim $ are unions of strata of $X$. Let $\mf X^i_{\mc E,C}$ be the union of the equivalence classes made up only of strata of dimension $\leq i$, and let $\mf X_{\mc E,C}$ be the stratification of $|X|$ with these skeleta.

\begin{definition}
We call $\mf X_{\mc E,C}$ the \emph{intrinsic stratification of $X$ rel $(\mc E,C)$}.
\end{definition}

The following proposition contains the properties of  $\mf X_{\mc E,C}$, including that this is a CS set and that it provides a common coarsening of all CS set stratifications of $|X|$ that are fully adapted to $\mc E$ and for which $C$ is a closed union of strata (Property \ref{I: int6}). We will only need the case $k=1$ of the last statement, Property \ref{I: int2}, which concerns the closure of the union of strata of codimension one. However,  the proof is equivalent for any $k$ so we provide the more general version.

\begin{proposition}\label{P: intrinsic}
Let $X$ be an $n$-dimensional CS set fully adapted to the maximal ts-coefficient system $\mc E$ with domain the open $n$-manifold $U_{\mc E}$, and let $C$ be a closed union of strata of $X$ of codimension $\geq 1$. Then:
\begin{enumerate}

\item\label{I: int1} The sets $\mf X^i_{\mc E, C}$ filter $|X|$ as a CS set.

\item\label{I: int2a} If $x$ and $y$ are in the same stratum of $\mf X_{\mc E,C}$ then $x\sim y$.

\item\label{I: int2b} $C$ is a union of strata of $\mf X_{\mc E,C}$.

\item\label{I: int3} $\mf X-\mf X^{n-1}=U_{\mc E}-C$.

\item\label{I: int4} $\mf X_{\mc E, C}$ is fully adapted to $\mc E$.

\item\label{I: int5} Suppose  $\mc Y$ is another CS set stratification of $|X|$ that is fully adapted to $\mc E$  and such that $C$ is also a closed union of strata of $\mc Y$. Then starting with $\mc Y$ results in the same $\mf X^i_{\mc E, C}$, i.e.\ the intrinsic stratification of $\mc Y$ rel $(\mc E,C)$ is also $\mf X_{\mc E, C}$.

\item\label{I: int6} Suppose  $\mc Y$ is another CS set stratification of $|X|$ that is fully adapted to $\mc E$  and such that $C$ is also a closed union of strata of $\mc Y$. Then $\mc Y$ refines $\mf X_{\mc E,C}$, i.e.\ each stratum of $\mf X_{\mc E,C}$ is a union of strata of $\mc Y$. Hence, $\mf X_{\mc E,C}$ is a common coarsening of all such stratifications.

\item\label{I: int2} Suppose  $C_k$ is the closure of the union of strata of $X$ of codimension $k$. Then $C_k$ is also the closure of the union of strata of $\mf X_{\mc E,C_k}$ of codimension $k$.

\end{enumerate}
\end{proposition}

\begin{proof}
We write simply $\mf X$ rather than $\mf X_{\mc E,C}$. We take each statement in turn.

\paragraph{\ref{I: int1}.}
The proof that the sets $\mf X^i$ filter $|X|$ as a CS set is essentially identically to the classical case \cite[Proposition 2.10.5]{GBF35}, using Lemma \ref{L: simE}. We sketch the argument as we will use some of the details below. 

First observe that if $x\sim y$ by the homeomorphism $h:U_x\to U_y$ and if $z\in U_x$ then $z\sim h(z)$ letting $U_z=U_x$ and $U_{h(z)}=U_y$. Now suppose $x\in |X|-\mf X^i$ so that  $x$ is equivalent to a point $y$ in a stratum of $X$ of dimension $>i$. By restricting to a smaller $U_x$ if necessary, we can assume $h(U_x)$ is contained in a distinguished neighborhood of $y$ in $X$. Then $h(U_x)$ intersects only strata of dimension $>i$ and so each point of $U_x$ is equivalent to a point in a stratum of dimension $>i$. So $U_x\in |X|-\mf X^i$. Thus  $|X|-\mf X^i$ is open so $\mf X^i$ is closed and the $\mf X^i$ provide a closed filtration of  $X$, as clearly $\mf X^i\subset \mf X^{i+1}$. 

Next suppose $x\in X_i\cap \mf X_i$. Then $x$ has a distinguished neighborhood $N\cong \R^i\times cL$ in $X$. It is shown in the proof of \cite[Proposition 2.10.5]{GBF35} that if we think of $L$ as embedded as the image of $\{0\}\times \{1/2\}\times L$ and refilter $|L|$ by $\ell^{j-i-1}=|L|\cap \mf X^j$ then the image of $\R^i\times c\ell$ becomes a distinguished neighborhood of $x$ in $\mf X$. If $z\in \mf X_i-\mf X_i\cap X_i$ then $z$ is equivalent to some point $x\in X_i\cap \mf X_i$, and we obtain a distinguished neighborhood for $z$ in $\mf X$ as the filtered homeomorphic image of a distinguished neighborhood of $x$ in $\mf X$ under the homeomorphism of the equivalence. See \cite{GBF35} for details.

\paragraph{\ref{I: int2a}.} It follows from Lemma \ref{L: simE} and the construction of distinguished neighborhoods in the proof of Property \ref{I: int1} that each point $x$ has a distinguished neighborhood $\R^i\times cL$ in $\mf X$ such that the points of $\R^i\times \{v\}$, i.e.\ all the points in the neighborhood that are in the same stratum of $\mf X$ as $x$, are equivalent. Property \ref{I: int2a} now follows from the same sort of open/closed argument as in the proof of Lemma \ref{L: simE}.

\paragraph{\ref{I: int2b}.} By the preceding property, all points in any fixed stratum of $\mf X$ are equivalent. From the definition of the equivalence relation, this would not be possible if any stratum of $\mf X$ intersected both $C$ and its complement. Hence any stratum intersecting $C$ is contained in $C$, and $C$ is a union of strata. 

\paragraph{\ref{I: int3}.}
Suppose $x\in \mf X^n-\mf X^{n-1}$. Then by definition $x$ is equivalent to a point $z$ in $X^n-X^{n-1}$. But since $X$ is adapted to $\mc E$, the point $z$ has a Euclidean neighborhood on which $\mc E$ is defined. Hence so does $x$. Thus $\mf X-\mf X^{n-1}\subset U_{\mc E}$. Furthermore, by construction of $\sim$ and the assumption that $C$ is the closure of a union of strata of codimension $\geq 1$,  no point in $C$ can be equivalent to a point in $X-X^{n-1}$, so  $\mf X-\mf X^{n-1}\subset U_{\mc E}- C$. 

Next suppose $x\in U_{\mc E}-C$. Then $x$ has a Euclidean neighborhood in $|X|-C$ on which $\mc E$ is defined. 
 By the argument in the proof of Lemma \ref{L: simE}, $x$ will be equivalent to every other point in this neighborhood. In particular $x$ is equivalent to a point in an $n$-dimensional stratum of $X$ and so $x\in \mf X^n-\mf X^{n-1}$.

\paragraph{\ref{I: int4}.} We have already seen that $\mf X-\mf X^{n-1}=U_{\mc E} - C$, so in particular  $\mf X-\mf X^{n-1}\subset U_{\mc E} $.
It remains to show that $U_{\mc E}\cap C$ is a union of strata of $\mf X$. Let $\mf S\subset \mf X_i$, $i\leq n-1$, be any stratum of $\mf X$. Then $U_{\mc E}\cap \mf S$ is open in $\mf S$ since $U_{\mc E}$ is an open set. Thus it suffices to show that $U_{\mc E}\cap \mf S$ is also closed in $\mf S$. Let $x$ be in the closure of $U_{\mc E}\cap \mf S$ so that every neighborhood of $x$ intersects $U_{\mc E}\cap \mf S$. By the proof of Property \ref{I: int1}, $x$ is equivalent to a point  $z\in X_i$ and the homeomorphism $h: U_z\to U_x$ induces (possibly after restriction to a subspace) a filtered homeomorphism from a distinguished neighborhood of $z$ in $\mf X$ to a distinguished neighborhood of $x$ in $\mf X$. In particular, $h$ takes a neighborhood $B$ of $z$ in $X_i$ to a neighborhood of $x$ in $\mf X_i$. Since $X$ is fully adapted to $\mc E$, either $B\subset U_{\mc E}$ or $B\cap U_{\mc E}=\emptyset$. But since $h(B)$ is a neighborhood of $x$ in $\mf S$, there is some $y\in h(B)\cap U_{\mc E}$, and hence $h^{-1}(y)\in B\cap U_{\mc E}$ by definition of $\sim$. Thus $z\in U_{\mc E}$ and so is $x$.

\paragraph{\ref{I: int5}.} Here we modify the proof of \cite[Proposition 2.10.5]{GBF35}.  
 Let $X$ and $\mc Y$ be two CS set stratifications of $|X|$ fully adapted to $\mc E$ and such that $C$ is a closed union of strata of $\mc Y$. Let $\mf X$ and $\mf Y$ be the resulting coarsenings. The equivalence relation $\sim$ does not depend on the stratifications, and so the equivalence relations used to define $\mf X$ and $\mf Y$ are the same and we will use the same symbol for both. However,  the definitions of the skeleta $\mf X^i$ and $\mf Y^i$ do \emph{a priori} depend on the stratifications, so this is what we must consider.

Clearly $\mf X^n=|X|=\mf Y^n$, and by our preceding arguments $\mf X -\mf X^{n-1}=U_{\mc E}-C=\mf Y-\mf Y^{n-1}$ so that $\mf X^{n-1}=\mf Y^{n-1}$. 
Now let $x\in \mf X^i$ for some $i< n-1$. Then $x$ cannot be equivalent to any point in $X_j$ with $j>i$ by definition. Suppose $x\in \mf Y_j$ for some $j>i$, and let $\mf S$ be the stratum of $\mf Y_j$ containing $x$. By  Property \ref{I: int2a}, the points of $\mf S$ are all equivalent to $x$. But now dimension considerations show that there must be points arbitrarily close to $x$ that are equivalent to $x$ but not contained in $X^i$, and in particular there is therefore some stratum of $X$ of dimension $>i$ containing points equivalent to $x$, a contradiction. So $x$ is not in any $\mf Y_j$ with $j>i$ and so $x\in \mf Y^i$. Thus $\mf X^i\subset \mf Y^i$, and the same argument shows the converse. So $\mf X^i= \mf Y^i$ for all $i$.

\paragraph{\ref{I: int6}.}
The statement follows from the preceding one and Lemma \ref{L: simE}.

\paragraph{\ref{I: int2}.}
Let $\mf C_k$ denote the closure of the strata of $\mf X$ of codimension $k$. By the Frontier Condition, $C_k\subset X^{n-k}$. So if $x\in X_{n-k}$ then $x$ cannot be equivalent to any point in $X-X^{n-k}$ as such points have neighborhoods that do not intersect $C_k$. So $x\in \mf X^{n-k}$. But also $x$ clearly cannot be equivalent to a point that is only equivalent to points in strata of $X$ of dimension $<n-k$, so $x\in \mf X^{n-k}-\mf X^{n-k-1}$.  Thus $X^{n-k}-X^{n-k-1}\subset \mf  X^{n-k}-\mf X^{n-k-1}$. Taking closures, $C_k\subset \mf C_k$. 

Next, suppose $x\in \mf X^{n-k}-\mf X^{n-k-1}$. 
By definition $x$ is equivalent to a point in a stratum of $X$ of dimension $n-k$. If $x\notin C_k$ then $x$ has a neighborhood that does not intersect $C_k$ and so $x$ is not equivalent to a point in $C_k$, a contradiction. So $x\in C_k$. Thus $ \mf X^{n-k}-\mf X^{n-k-1}\subset C_k$, so, taking closures, $\mf C_k\subset C_k$. 
\end{proof}

\providecommand{\bysame}{\leavevmode\hbox to3em{\hrulefill}\thinspace}
\providecommand{\MR}{\relax\ifhmode\unskip\space\fi MR }
\providecommand{\MRhref}[2]{%
  \href{http://www.ams.org/mathscinet-getitem?mr=#1}{#2}
}
\providecommand{\href}[2]{#2}

Some diagrams in this paper were typeset using the \TeX\, commutative
diagrams package by Paul Taylor.

Data availability statement:
Data sharing not applicable to this article as no datasets were generated or analyzed during the current study.


\begin{thebibliography}{10}

\bibitem{Ak69}
Ethan Akin, \emph{Manifold phenomena in the theory of polyhedra}, Trans. Amer.
  Math. Soc. \textbf{143} (1969), 413--473.

\bibitem{BBD}
A.~A. Be{\u\i}linson, J.~Bernstein, and P.~Deligne, \emph{Faisceaux pervers},
  Analysis and topology on singular spaces, {I} ({L}uminy, 1981), Ast\'erisque,
  vol. 100, Soc. Math. France, Paris, 1982, pp.~5--171.

\bibitem{Bo}
A.~Borel et~al., \emph{Intersection cohomology}, Progress in Mathematics,
  vol.~50, Birkh\"auser Boston, Inc., Boston, MA, 1984, Notes on the seminar
  held at the University of Bern, Bern, 1983, Swiss Seminars.

\bibitem{Br}
Glen~E. Bredon, \emph{Sheaf theory}, second ed., Graduate Texts in Mathematics,
  vol. 170, Springer-Verlag, New York, 1997.

\bibitem{CS91}
Sylvain~E. Cappell and Julius~L. Shaneson, \emph{Singular spaces,
  characteristic classes, and intersection homology}, Ann. of Math. (2)
  \textbf{134} (1991), no.~2, 325--374.

\bibitem{CST-inv}
David Chataur, Martintxo Saralegi-Aranguren, and Daniel Tanr\'{e},
  \emph{Intersection homology: general perversities and topological
  invariance}, Illinois J. Math. \textbf{63} (2019), no.~1, 127--163.

\bibitem{GBF11}
Greg Friedman, \emph{Superperverse intersection cohomology: stratification
  (in)dependence}, Math. Z. \textbf{252} (2006), no.~1, 49--70.

\bibitem{GBF10}
\bysame, \emph{Singular chain intersection homology for traditional and
  super-perversities}, Trans. Amer. Math. Soc. \textbf{359} (2007), no.~5,
  1977--2019.

\bibitem{GBF23}
\bysame, \emph{Intersection homology with general perversities}, Geom. Dedicata
  \textbf{148} (2010), 103--135.

\bibitem{GBF26}
\bysame, \emph{An introduction to intersection homology with general perversity
  functions}, Topology of stratified spaces, Math. Sci. Res. Inst. Publ.,
  vol.~58, Cambridge Univ. Press, Cambridge, 2011, pp.~177--222.

\bibitem{GBF32}
\bysame, \emph{Generalizations of intersection homology and perverse sheaves
  with duality over the integers}, Michigan Math. J. \textbf{68} (2019), no.~4,
  675--726.

\bibitem{GBF35}
\bysame, \emph{Singular intersection homology}, New Mathematical Monographs,
  vol.~33, Cambridge University Press, 2020.

\bibitem{GBF44}
\bysame, \emph{Two short proofs of the topological invariance of intersection
  homology}, J. Singul. \textbf{25} (2022), 144--149.

\bibitem{GM1}
Mark Goresky and Robert MacPherson, \emph{Intersection homology theory},
  Topology \textbf{19} (1980), no.~2, 135--162.

\bibitem{GM2}
\bysame, \emph{Intersection homology. {II}}, Invent. Math. \textbf{72} (1983),
  no.~1, 77--129.

\bibitem{GS83}
Mark Goresky and Paul Siegel, \emph{Linking pairings on singular spaces},
  Comment. Math. Helv. \textbf{58} (1983), no.~1, 96--110.

\bibitem{HS91}
Nathan Habegger and Leslie Saper, \emph{Intersection cohomology of cs-spaces
  and {Z}eeman's filtration}, Invent. Math. \textbf{105} (1991), no.~2,
  247--272.

\bibitem{KS}
Masaki Kashiwara and Pierre Schapira, \emph{Sheaves on manifolds}, Grundlehren
  der Mathematischen Wissenschaften [Fundamental Principles of Mathematical
  Sciences], vol. 292, Springer-Verlag, Berlin, 1994.

\bibitem{Ki}
Henry~C. King, \emph{Topological invariance of intersection homology without
  sheaves}, Topology Appl. \textbf{20} (1985), no.~2, 149--160.

\bibitem{RIB}
Paulo Ribenboim, \emph{Rings and modules}, Interscience Tracts in Pure and
  Applied Mathematics, No. 24, Interscience Publishers John Wiley \& Sons,
  Inc., New York-London-Sydney, 1969.

\bibitem{Sa05}
Martintxo Saralegi-Aranguren, \emph{de {R}ham intersection cohomology for
  general perversities}, Illinois J. Math. \textbf{49} (2005), no.~3, 737--758
  (electronic).

\bibitem{Sch03}
J\"org Sch\"urmann, \emph{Topology of singular spaces and constructible
  sheaves}, Instytut Matematyczny Polskiej Akademii Nauk. Monografie
  Matematyczne (New Series) [Mathematics Institute of the Polish Academy of
  Sciences. Mathematical Monographs (New Series)], vol.~63, Birkh\"auser
  Verlag, Basel, 2003.

\bibitem{Si72}
L.~C. Siebenmann, \emph{Deformation of homeomorphisms on stratified sets. {I},
  {II}}, Comment. Math. Helv. \textbf{47} (1972), 123--136; ibid. 47 (1972),
  137--163.

\bibitem{Va14}
Guillaume Valette, \emph{A {L}efschetz duality for intersection homology},
  Geom. Dedicata \textbf{169} (2014), 283--299.

\end{thebibliography}
\end{document}